\documentclass[12pt]{article}
\usepackage{amsmath}
\usepackage{graphicx}
\usepackage{enumerate}
\usepackage{natbib}
\usepackage{url} 
\usepackage{hyperref}
\hypersetup{colorlinks=true,linkcolor=blue,urlcolor=blue,allcolors=blue}
\newcommand{\blind}{0}

\usepackage{graphicx,epsfig}
\usepackage{amsmath,amsthm,amssymb,amscd,mathrsfs}
\usepackage{amsfonts}
\usepackage{ifthen,latexsym,syntonly}
\usepackage{natbib}  
\usepackage{verbatim}
\usepackage{soul,color}
\usepackage{pstricks}
\usepackage{bbm}
\usepackage{CJK}
\usepackage[ansinew]{inputenc}
\usepackage{latexsym}
\usepackage{epsf}
\usepackage{extarrows}
\usepackage[]{authblk}
\usepackage{multirow}
\usepackage{float,paralist,enumerate}

\renewcommand{\theenumi}{\Alph{enumi}}

\makeatletter
\renewcommand{\p@enumii}{\theenumi.}
\makeatother

\newtheorem{thm}{Theorem}
\newtheorem{cor}{Corollary}

\newtheorem{lem}{Lemma}
\newtheorem{prop}{Proposition}

\newtheorem{rmk}{Remark}

\DeclareMathOperator{\var}{Var}   \DeclareMathOperator{\cov}{Cov}

\DeclareMathOperator{\tr}{tr}         
\DeclareMathOperator{\im}{Im}      \DeclareMathOperator{\diag}{diag}
        \DeclareMathOperator{\E}{E}

\renewcommand{\(}{\left(}
\renewcommand{\)}{\right)}
\newcommand{\lb}{\left(}
\newcommand{\rb}{\right)}

\newcommand{\De}{\Delta}
\newcommand{\lam}{\lambda}

\renewcommand{\th}{\theta}

\newcommand{\si}{\sigma}

\newcommand{\ga}{\gamma}
\newcommand{\om}{\omega}

\newcommand{\vep}{\varepsilon}

\newcommand{\fSi}{{\pmb\Sigma}}
\newcommand{\fGa}{{\pmb\Gamma}}
\newcommand{\fLa}{{\pmb\Lambda}}

\newcommand{\fOm}{{\pmb\Omega}}

\newcommand{\fvep}{{\pmb\varepsilon}}

\newcommand{\fXi}{\pmb\Xi}

\newcommand{\fA}{{\bf A}}      \newcommand{\fa}{{\bf a}}
\newcommand{\fB}{{\bf B}}      \newcommand{\fb}{{\bf b}}
\newcommand{\fC}{{\bf C}}      \newcommand{\fc}{{\bf c}}
\newcommand{\fD}{{\bf D}}      
      \newcommand{\fe}{{\bf e}}
\newcommand{\fF}{{\bf F}}      \newcommand{\ff}{{\bf f}}

\newcommand{\fI}{{\bf I}}
      
\newcommand{\fK}{{\bf K}}      
\newcommand{\fM}{{\bf M}}

\newcommand{\fQ}{{\bf Q}}      \newcommand{\fq}{{\bf q}}
\newcommand{\fR}{{\bf R}}      \newcommand{\fr}{{\bf r}}
\newcommand{\fS}{{\bf S}}      
      
      \newcommand{\fu}{{\bf u}}
\newcommand{\fV}{{\bf V}}      \newcommand{\fv}{{\bf v}}
\newcommand{\fW}{{\bf W}}      
\newcommand{\fX}{{\bf X}}      \newcommand{\fx}{{\bf x}}
      \newcommand{\fy}{{\bf y}}
\newcommand{\fZ}{{\bf Z}}      \newcommand{\fz}{{\bf z}}

\newcommand{\cO}{{\mathcal O}}

\newcommand{\cF}{{\mathcal F}}
\newcommand{\cA}{{\mathcal A}}
\newcommand{\bC}{{\mathbb C}}
\newcommand{\bR}{{\mathbb R}}

\def\indic{{\mathbf{1}}}

\def\ul{\underline}
\def\wt{\widetilde}
\def\wh{\widehat}
\def\to{\rightarrow}

\def\q{\quad}

\def\pc{\,{\buildrel p \over \longrightarrow}\,}
\def\wk{\stackrel{\mathcal{D}}{\to}}
\def\iid{\sim_{{i.i.d.}}}
\def\mT{\mathrm{T}}

\graphicspath{{./figures/}}

\begin{document}

\def\spacingset#1{\renewcommand{\baselinestretch}%
{#1}\small\normalsize} \spacingset{1}


\if0\blind
{
  \title{\bf Tests for principal eigenvalues and eigenvectors}
  \author{Jianqing Fan \\ 
    Operations Research and Financial Engineering, Princeton University\\ \vspace{.2cm}
    Yingying Li
\hspace{.2cm}\\ \vspace{.2cm}
Department of Information System, Business Statistics and Operations Management, Hong Kong University of Science and Technology\\ \vspace{.2cm}
    Ningning Xia
\hspace{.2cm}\\ \vspace{.2cm}
School of Statistics and Management, Shanghai University of Finance and Economics
     \\ \vspace{.2cm}
    Xinghua Zheng  
\hspace{.2cm}\\ \vspace{.2cm}
    Department of Information System, Business Statistics and Operations Management, Hong Kong University of Science and Technology}
  \maketitle
} \fi

\if1\blind
{
  \bigskip
  \bigskip
  \bigskip
  \begin{center}
    {\LARGE\bf Tests for principal eigenvalues and eigenvectors}
\end{center}
  \medskip
} \fi

\bigskip
\begin{abstract}
We establish central limit theorems for  principal eigenvalues and eigenvectors under a large factor model setting,
and develop two-sample tests of both principal eigenvalues and principal eigenvectors.
One important application is to detect structural breaks in large factor models. Compared with existing methods for detecting structural breaks, our tests provide unique insights into the source of structural breaks because they can distinguish between individual principal eigenvalues and/or eigenvectors. We demonstrate the application by comparing the principal eigenvalues and principal eigenvectors of S\&P500 Index constituents' daily returns over different years.
\end{abstract}

\noindent%
\vfill

\newpage
\spacingset{1.9} 
\section{Introduction}
\label{sec:intro}

Factor models have been widely adopted in many disciplines,  most notably, economics and finance. Some of the most famous examples include the capital asset pricing model (CAPM, \cite{Sharpe64}), arbitrage pricing theory (\cite{Ross76}), approximate factor model (\cite{CR83}),  Fama-French three factor model (\cite{FF92}) and the more recent five-factor model (\cite{FF15}).

Statistically, the analysis of factor models is closely related to principal component analysis (PCA). For example, finding the number of factors is equivalent to determining the number of principal eigenvalues (\cite{BN02,Onatski10,AH13});  estimating factor loadings as well as factors relies on principal eigenvectors (\cite {SW98,SW02,Bai03,BN06,FLM11, FLM13,WF17}).

A factor model typically reads as follows:
\begin{equation}\label{model:factor}
y_{it}=\fb_i^\mT \ff_t +\vep_{it}, ~~~~ i=1,2,\ldots, N, ~ t=1,2,\ldots, T,
\end{equation}
where $y_{it}$ is the observation from the $i$th subject at time $t$,
$\ff_t$ is a set of factors, and $\vep_{it}$ is the idiosyncratic component. The number of factors, $r=\dim(f_t)$, is small compared with the dimension $N$, and  is assumed to be fixed throughout the paper.
The factor model \eqref{model:factor} can be put in a matrix form as
\[
\fy_t=\fB\ff_t +\fvep_t, t=1,2,\ldots, T,
\]
where $\fy_t=(y_{1t}, \ldots, y_{Nt})^\mT$, $\fB=(\fb_1,\ldots,\fb_N)^\mT$ and $\fvep_t=(\vep_{1t},\ldots,\vep_{Nt})^\mT$.
If follows that the covariance matrix  $\fSi$ of $\fy_t$ satisfies
\[
\fSi=\fB\cov(\ff_t)\fB^\mT+\fSi_{\vep},
\]
where $\fSi_\vep$ is the covariance matrix of $(\fvep_t)$.

The factors $(\ff_t)$ in some situations are taken to be observable. Examples include the market factor in CAPM and the  Fama-French three factors.   In some other situations, factors are latent and hence unobservable. In this paper, we focus on the latent factor case.

Factor models provide a parsimonious way to describe the dynamics of large dimensional variables. In the study of factor models, time invariance of factor loadings is a standard assumption. For example, in order to apply PCA, the loadings need to be time invariant or at least roughly so, otherwise the estimation will  be inconsistent.
However, parameter instability has been a pervasive phenomenon in time series data. Such instability could be due to policy regime switches, changes in  economic/finanncial   fundamentals, etc. Because of this reason, caution has to be exercised about  potential structural changes in  real data. Statistical analysis of  structural change in large factor model is challenging because the factors are unobserved and factor loadings have to be estimated.

There are some existing work on detecting structural breaks. Typically, the setup is as follows: suppose there are two time periods, one from time 1 to $T_1$, the second from $T_1 + 1$ to $T_1+ T_2$, where $T_1$ and $T_2$ do not necessarily equal. The first period has loading $\fB_1$, and the second period has loading $\fB_2$. One then tests whether $\fB_1$ equals~$\fB_2$. Specifically, one considers the following model:
\begin{eqnarray*}
\fy_t &=& \fB_1 \fF_t +\fvep_t, ~~~ t=1,2,\ldots, T_1, \\
\fy_t &=& \fB_2 \fF_t +\fvep_t, ~~~ t=T_1+1,\ldots, T_1 +T_2,
\end{eqnarray*}
and tests  the following hypothesis for detecting structural breaks
\[
H_0: ~ \fB_1=\fB_2 ~~~~ \textrm{vs.} ~~~~ H_a: ~ \fB_1\neq\fB_2.
\]
 Existing works include \cite{SW09, BE11,CDG14,HI15}, among others.

Let us connect the factor loadings with principal eigenvalues and eigenvectors.
Recall that $\fSi$ stands for the covariance matrix of $(\fy_t)$.  Write its spectral decomposition as
  \[
\fSi = \fV \fLa {\fV}^T,
\]	
where
\[
\fV=(\fv_1, \ldots , \fv_N), \mbox{ and }
\fLa=\diag(\lam_1, \ldots, \lam_N).
\]
The diagonal matrix $\fLa$ consists of eigenvalues in descending order, and $\fV$ consists of corresponding eigenvectors.
 Under  the convention that $\cov(\ff_t) = \fI$, the factor loading matrix
 \[
    \fB = \(\sqrt{\lam_1} \fv_1,\ldots,\sqrt{\lam_r} \fv_r\).
 \]
Therefore structural breaks can be due to changes in
 \begin{enumerate}[(i)]
 \item  one or more $\lam_i$,   or
 \item  one or more $\fv_i$,   or
 \item both.
\end{enumerate}
The economic and/or financial implications of these possibilities are, however, completely different. If a structural break is only due to change in eigenvalues, then in many applications, the structural break has no essential impact. For example, from dimension reduction point of view, if the principal eigenvalues change while the principal eigenvectors do not change, then projecting onto the principal eigenvectors is still valid. In contrast, if a structural break is caused by eigenvectors, then it may indicate a much more fundamental change,  possibly associated with important economic or market condition changes, to which one should be alerted.

Such observations bring up the aim of this paper: instead of testing whether the whole matrix $\fB$ is the same during two periods, we want to detect changes in \emph{individual}  principal eigenvalues and eigenvectors. By doing so, we can pinpoint the source of structural changes. Specifically, when a structural break occurs, we can determine whether it is caused by a change in a principal eigenvalue, a change in a principal eigenvector, or perhaps changes in both  principal eigenvalues and eigenvectors.

To be more specific, we consider the the following three tests. Let $\fSi^{(1)}$ and $\fSi^{(2)}$ be the population covariance matrices for the two periods under study. For any symmetric matrix $A$ and any integer $k$, we let $\lam_k(A)$ denote the $k$th largest eigenvalue of $A$,  $\fv_k(A)$  the corresponding eigenvector, and $\tr(\fA)$  its trace.
\begin{enumerate}[(i)]
\item Test equality of principal eigenvalues: for each $k=1,\ldots,r$, we test
\[
H_0^{(I,k)}: ~ \lam_k^{(1)}=\lam_k^{(2)}  ~~~~~  \textrm{vs} ~~~~
H_a^{(I,k)}: ~ \lam_k^{(1)}\neq \lam_k^{(2)},
\]
where $\lam_k^{(1)}:=\lam_k(\fSi^{(i)}),\ i=1,2$.
\item Considering that the total variation may vary, we  test about equality of the ratio of principal eigenvalues: for each $k=1,\ldots,r$,  test
\[
H_0^{(II,k)}: ~ \dfrac{\lambda_k^{(1)}}{\tr(\fSi^{(1)})}
= \dfrac{\lambda_k^{(2)}}{\tr(\fSi^{(2)})} ~~~~~  \textrm{vs} ~~~~
H_a^{(II,k)}: ~ \dfrac{\lambda_k^{(1)}}{\tr(\fSi^{(1)})}
\neq \dfrac{\lambda_k^{(2)}}{\tr(\fSi^{(2)})}.
\]
\item Most importantly, we  test equality of principal eigenvectors: for each $k=1,2,\ldots,r$,  test
\[
H_0^{(III,k)}: ~
|\langle \fv_k^{(1)}, \fv_k^{(2)}\rangle|=1,
 ~~~~~  \textrm{vs} ~~~~
H_a^{(III,k)}: ~
|\langle \fv_k^{(1)}, \fv_k^{(2)} \rangle| < 1,
\]
\end{enumerate}
where $\fv_k^{(i)}:=\fv_k(\fSi^{(i)}),\ i=1,2$, and $\langle \fa,\fb\rangle $ denotes the inner product of two vectors $\fa$ and $\fb$.

In this paper, we establish central limit theorems (CLT) for principal eigenvalues, eigenvalue ratios, as well as eigenvectors. We then develop two-sample tests based on these CLTs.

Due to the wide application of  PCA, a lot of work has been devoted to investigating principal eigenvalues.
However, the study of principal eigenvectors is very limited.  This paper represents a significant advancement in this direction.

We remark that there  is an independent work, \cite{BD2022}, that study similar questions. Nevertheless, there are several significant  differences  between \cite{BD2022} and our paper.
First, the non-principal eigenvalues are assumed to be equal  in \cite{BD2022}; see equation (1.2) therein. This is an unrealistic assumption in many applications.
In our paper, we allow the non-principal eigenvalues to follow an arbitrary distribution, rendering our results readily applicable in practice.
Second, in \cite{BD2022}, the dimension to the sample size ratio needs to be away from one; see Assumption 2.4 therein. We do not impose such a restriction in our paper. Third, \cite{BD2022} only consider the one-sample situation and study the projection of sample leading eigenvectors onto a given direction.
In our paper, we establish two-sample CLT, where  the projection of a principal eigenvector onto a random direction is considered.
Establishing such a result  presents a significant challenge. In summary, the setting of our paper is practically appropriate, and the results are of significant importance.

The organization of the paper is as follows.
Theoretical results are presented in Sections \ref{sec:setting}-\ref{sec:two_sample_test}. Simulation and Empirical studies are presented in Section \ref{sec:simulation} and \ref{sec:emp}, respectively. Proofs are collected in the Appendix.

\textit{Notation}: we use the following notation in addition to what have been introduced above. For a symmetric $N\times N$ matrix $\fA$, its {\it empirical spectral distribution} (ESD) is defined as
\[
F^{\fA} (x)=\dfrac 1N \ \sum_{j=1}^N \ \indic(\lam_j(\fA) \le x), ~~~~~ x\in\bR,
\]
where  $\indic(\cdot)$ is the indicator function.
The limit of ESD as $N\to\infty$, if it exists, is referred to as the {\it limiting spectral distribution}, or LSD for short.
For any vector $\fa$, let ${\fa}[k]$ be its $k$th entry.
We use  $`` \wk "$ to denote weak convergence.

\section{Setting and Assumptions}\label{sec:setting}

We assume that  $(\fy_t)_{t=1}^T$ is a sequence of \hbox{i.i.d.} $N-$dimensional random vectors with mean zero and covariance matrix $\fSi$.
Let $\lambda_1,\ldots,\lambda_N$ be the eigenvalues of $\fSi$ in descending order, and $\fv_1,\ldots,\fv_N$ be the corresponding eigenvectors.
Write $\fLa=\diag(\lam_1,\ldots, \lam_N)$ and $\fV=(\fv_1, \ldots, \fv_N)$. Then the spectral decomposition of $\fSi$ is given by $\fSi=\fV \fLa \fV^\mT$.

We make the following assumptions.

\smallskip
\noindent{\bf Assumption  A:}
\begin{compactenum}\setcounter{enumi}{1}
	\item[]
	The eigenvalues $\lambda_1>\lambda_2>\ldots>\lambda_r > \lambda_{r+1}\geq\ldots\geq \lambda_N$ satisfy that
	\begin{compactenum}
		\item\label{asm:princ} for the principal part, one has $\lim_{N\to\infty} \lam_k/N = \theta_k \in  (0,+\infty) $ for $1\le k \le r$, where $r>1$ is a fixed integer and $\theta_k$'s are distinct.
		\item\label{asm:bulk} for the non-principal part, there exists a $C_0<\infty$ such that $\lambda_j\le C_0$ for $j>r$, and the empirical  distribution of $\{\lam_{r+1}, \ldots, \lam_N\}$ tends to a distribution $H$.
	\end{compactenum}
\end{compactenum}

\begin{rmk}
Assumption (A.i) implies that the factors are strong. When the factors are weak, say $\lambda_i\asymp N^{\alpha}$ for some $\alpha\in (0,1)$, the convergence of sample principal components  still holds  with the convergence rate depending on  $\alpha$. In this paper, we only focus on the strong factor case and leave the study of weak factors for future work.
\end{rmk}

\noindent{\bf Assumption  B:}
The observations  $(\fy_i )_{i=1}^T$ can be written as $\fy_i=\fV \fLa^{1/2} \fz_i$, where\\
 $\{\fz_i = (\fz_i[1], \fz_i[2], \ldots, \fz_i[N] )^\mT, \ i=1, 2, \ldots, T \}$ are \hbox{i.i.d.} random vectors, and $\fz_i[\ell], \ell=1,\ldots,N, $ are independent random variables with zero mean, unit variance and satisfying $\sup_N \max_{1\le \ell \le N} \E (\fz_1[\ell])^4 < \infty $.

\begin{rmk}
Assumption B covers the multivariate normal case and coincides  with the idea of PCA.
Specifically, if $\fy_i$ follows a multivariate normal distribution, then
$\fz_i = \fLa^{-1/2} \fV^\mT \fy_i$ is an $N$-dimensional standard normal random vector and Assumption~B holds naturally.
On the other hand, under the orthogonal basis $\fV=(\fv_1, \ldots, \fv_N)$, the coordinates of $\fy_i$ are $(\sqrt{\lam_1} \fz_i[1], \ldots, \sqrt{\lam_N} \fz_i[N] )$. Assumption~B says that the coordinate variables are independent with mean zero and variance $\lam_i, i=1,\ldots,N$.
\end{rmk}

\noindent{\bf Assumption  C:}
The dimension $N$ and sample size $T$ are such that $\rho_N:=N/T \to \rho \in (0,+\infty)$ as $N \to\infty$.

\section{One-sample Asymptotics} \label{sec:one_sample_thms}
Let $\wh{\fSi}_N$  be the sample covariance matrix defined as
\[
\wh{\fSi}_N = \dfrac 1T \sum_{t=1}^\mT \fy_t \fy_t^\mT  .
\]
Denote its eigenvalues by
$\wh{\lam}_1\geq \cdots \geq \wh{\lam}_N$, and let   $\wh{\fv}_1,\ldots,\wh{\fv}_N$ be the corresponding eigenvectors.

\begin{thm}\label{thm:clt_spike}
Under Assumptions A--C, the  principal eigenvalues converge weakly to a multivariate normal distribution:
\begin{eqnarray}\label{eqn:limit_spike_evalue}
\sqrt{T}
\begin{pmatrix}
\wh{\lam}_1/\lam_1 - 1 \\
\vdots \\
\wh{\lam}_r/\lam_r - 1
\end{pmatrix}
~ \wk ~  N(0, \fSi_{\lam}),
\end{eqnarray}
where
$\fSi_{\lam} = \diag(\si_{\lam_1}^2, \ldots, \si_{\lam_r}^2 )$ is a diagonal matrix with
$\si_{\lam_k}^2=\E\(\fz_1[k]\)^4-1$.
\end{thm}

\begin{rmk}
The marginal convergence in \eqref{eqn:limit_spike_evalue} has been established in \cite{WF17} under a sub-Gaussian assumption. We generalize their result to joint convergence and under a weaker moment assumption. {See also \cite{CHP20} for a related result under a different setting.}
\end{rmk}
\begin{rmk}
By Theorem \ref{thm:clt_ev} below, the variance $\si_{\lambda_k}^2$ can be consistently estimated by
\[
\wh{\si_{\lam_k}^2} = \dfrac{1}{T (\wh{\lam}_k)^2  } \sum_{t=1}^{T} (\wh{\fv}_k^\mT \fy_t)^4 -1,
\]
hence a feasible CLT is readily available.
\end{rmk}

\begin{thm}\label{thm:clt_ratio}
Under Assumptions  A--C, for each $1\le k\le r$, we have
\begin{eqnarray} \label{eq3.2}
\sqrt{\dfrac{T}{ \wh{\si_{-k}^2} } }
\(
\dfrac{\wh{\lam}_k}{\tr(\wh{\fSi}_N)-\wh{\lam}_k}
- \dfrac{\lambda_k}{\tr(\fSi)-\lambda_k} \)
\ \wk \ N(0,1),
\end{eqnarray}
where
\[
\wh{\si_{-k}^2} = \dfrac{\wh{\lam}_k^2 }{\( \tr(\wh{ \fSi}_N)-\wh{\lam}_k \)^2 } \left[
\wh{\si_{\lam_k}^2}
+ \dfrac{\sum_{j\neq k,j=1}^r \wh{\lam}_j^2 \ \wh{\si_{\lam_j}^2}  }{\( \tr(\wh{\fSi}_N)- \wh{\lambda}_k \)^2 }
 \right] .
\]
\end{thm}

\begin{rmk}
Theorem~\ref{thm:clt_ratio} can be used to construct the confidence interval for the ratio $\varrho_k: = {\lambda_k}/{\tr(\fSi)}$.  This follows easily from \eqref{eq3.2} and the fact that,
if we write $\wt{\varrho}={\lambda_k}/{(\tr(\fSi)-\lambda_k)}$, then
$\varrho_k = \wt{\varrho}_k/(1+\wt{\varrho}_k)$,
which is a strictly increasing function of~$\wt{\varrho}_k$.
\end{rmk}

\begin{thm}\label{thm:clt_ev}
Under Assumptions A--C, for each $1\le k\le r$, the principal sample eigenvector $\wh{\fv}_k$ satisfies
\[
T \(1-\langle \fv_k, \wh{\fv}_k \rangle^2 - \dfrac{1}{T\wh{\lambda}_k} \sum_{j=r+1}^N \dfrac{\wh{\lam}_j}{(1-\wh{\lam}_j/\wh{\lambda}_k)^2} \) \ \wk \ \sum_{i\neq k,i=1}^r \om_{ki}\cdot Z_i^2,
\]
where
$\om_{ki} = {\theta_k \theta_i}/{(\theta_k - \theta_i )^2} $, which can be consistently estimated by
\[
\wh{\om_{ki}} = \dfrac{\wh{\lam}_k \wh{\lam}_i}{(\wh{\lam}_k - \wh{\lam}_i )^2},
\]
and
$Z_i$'s are i.i.d standard normal random variables.
\end{thm}

\begin{rmk}
The convergence rate of $\langle \fv_k, \wh{\fv}_k \rangle^2$ has been established in Theorem 3.2 of \cite{WF17}. We derive the corresponding limiting distribution at the boundary of the parameter space, which is much more difficult to prove.
\end{rmk}

The proofs of Theorems \ref{thm:clt_spike}, \ref{thm:clt_ratio} and \ref{thm:clt_ev} are given in  the supplementary material.

\section{Two-sample Tests} \label{sec:two_sample_test}
We now discuss how to conduct the three tests mentioned in the Introduction.

Suppose that we have two groups of observations of the same dimension $N$:
\[
\fy_1^{(1)},\ldots,\fy_{T_1}^{(1)}, ~~~~
\textrm{and} ~~~
\fy_1^{(2)},\ldots,\fy_{T_2}^{(2)},
\]
which are drawn independently from two  populations with mean zero and covariance matrices $\fSi^{(1)}$ and $\fSi^{(2)}$. We assume that Assumption B holds for each group of observations. Moreover, analogous to Assumption C, we have
\[
\lim_{N\to\infty} \frac{N}{T_i}= \rho_i \in (0,+\infty),\q i=1,2.
\]
Finally, analogous to Assumption A,  with the spectral decompositions of $\fSi^{(i)},i=1,2$, given by
\begin{eqnarray*}
\fSi^{(i)} &=& {\fV}^{(i)} \fLa^{(i)} {{\fV}^{(i)}}^\mT,
\end{eqnarray*}
where
$
\fLa^{(i)}= \diag(\lambda_1^{(i)},\ldots,\lambda_r^{(i)},\lambda_{r+1}^{(i)}, \ldots, \lambda_N^{(i)}),$ and
${\fV}^{(i)}=({\fv}_1^{(i)},\ldots,{\fv}_r^{(i)},{\fv}_{r+1}^{(i)},\ldots{\fv}_N^{(i)}),$
 we assume, for each population, there are $r$  principal  eigenvalues, which satisfy
\[
\lim_{N\to\infty} \dfrac{\lambda_k^{(i)}}{N} \ = \ \theta_k^{(i)} \in (0,+\infty), ~~~~~~~ \textrm{for } ~ 1\le k \le r,
~ i=1,2;
\]
while the remaining eigenvalues $\lambda_j^{(i)}, r+1 \le j \le N$, are {uniformly}  bounded {and have a limiting empirical distribution.} The two liming empirical distributions for the two populations can be different.

Naturally, our tests will be based on the sample covariance matrices
\[
\wh{\fSi}_N^{(i)}=\frac{ 1}{T_i} \sum_{j=1}^{T_i} \fy_j^{(i)}{\fy^{(i)}_j}^\mT, \q i=1,2.
\]
Write their spectral decompositions as
\begin{eqnarray*}
\wh{\fSi}_N^{(i)}
=\wh{\fV}^{(i)} \wh{\fLa}^{(i)}
{(\wh{\fV}^{(i)})}^\mT
\end{eqnarray*}
where
\[
\wh{\fLa}^{(i)}=\diag(\wh{\lam}_1^{(i)}, \ldots, \wh{\lam}_N^{(i)}), ~~
\wh{\fV}^{(i)}=(\wh{\fv}_1^{(i)}, \ldots, \wh{\fv}_N^{(i)} ).
\]

\subsection{Testing equality of principal  eigenvalues}\label{ssec:test_eq_lam}

To test $H_0^{(I,k)}: ~ \lambda_k^{(1)}=\lambda_k^{(2)}$, we  use the following test statistic
\[
T_{\lam k}= \sqrt{\dfrac{T_1T_2}{T_1 (\wh{\si_{\lam_k}^2})^{(2)} +
	T_2  (\wh{\si_{\lam_k}^2})^{(1)}
	}} \cdot
\(
\dfrac{\wh{\lam}_k^{(1)}}{\wh{\lam}_k^{(2)}} -1
\) ,
\]
where
\[
(\wh{\si_{\lam_k}^2})^{(i)}
=
\dfrac{1}{(\wh{\lam}_k^{(i)})^2 \ T_i} \sum_{j=1}^{T_i} \( (\wh{\fv}_k^{(i)})^\mT \fy_j^{(i)} \)^4-1,\q i=1,2.
\]

\begin{thm}\label{thm:test_tlam}
Under  null hypothesis $H_0^{(I,k)}$ and Assumptions  A--C,
the proposed test statistic
$T_{\lam k}$ converges weakly to the standard normal distribution.
\end{thm}	
Theorem \ref{thm:test_tlam} follows directly from Theorem \ref{thm:clt_spike} and the Delta method.

\subsection{Testing equality of ratios of principal  eigenvalues}\label{ssec:test_eq_ratio}

The null hypothesis
\[
H_0^{(II,k)}: ~ \dfrac{\lambda_k^{(1)}}{\tr(\fSi^{(1)})}
= \dfrac{\lambda_k^{(2)}}{\tr(\fSi^{(2)})}
\]
is equivalent to
\[
H_0^{(II,k)'}: \dfrac{\lambda_k^{(1)}}{\tr(\fSi^{(1)})-\lambda_k^{(1)}}  \ = \
\dfrac{\lambda_k^{(2)}}{\tr(\fSi^{(2)})-\lambda_k^{(2)}}.
\]
Based on such an observation, we propose the following test statistic
\[
T_{ek}:= \dfrac{\sqrt N
	\(
	\dfrac{\wh{\lam}_k^{(1)}}{\tr(\wh{\fSi}_N^{(1)})-\wh{\lam}_k^{(1)} } -
	\dfrac{\wh{\lam}_k^{(2)}}{\tr(\wh{\fSi}_N^{(2)})-\wh{\lam}_k^{(2)}}
	\)}
{\sqrt{\dfrac{N}{T_1} \ \wh{{\si_{-k}^2}}^{(1)} + \dfrac{N}{T_2} \ \wh{{\si_{-k}^2}}^{(2)} } }
\]
where
\[
\wh{{\si_{-k}^2}}^{(i)}=
\dfrac{(\wh{\lam}_k^{(i)})^2}{ \(\tr(\wh{\fSi}_N^{(i)}) - \wh{\lam}_k^{(i)} \)^2}
\lb
\wh{\si_{\lam_k}^2 }^{(i)} +\dfrac{\sum_{j\neq k,j=1}^r(\wh{\lam}_j^{(i)})^2 \ \wh{\si_{\lam_j}^2 }^{(i)}  }{\( \tr(\wh{\fSi}_N^{(i)}) - \wh{\lam}_k^{(i)} \)^2}
\rb,\q    i=1,2.
\]

\begin{thm}\label{thm:test_te}
Under the null hypothesis $H_0^{(II,k)}$ and Assumptions A--C,  the  test statistic $T_{ek}$ converges weakly to the standard normal distribution.
\end{thm}
Theorem \ref{thm:test_te} is a straightforward consequence of Theorem \ref{thm:clt_ratio}.

\subsection{Testing equality of principal eigenvectors}
\label{ssec:test_tv}
To test $H_0^{(III,k)}: |\langle \fv_k^{(1)}, \fv_k^{(2)} \rangle|=1$, we propose the following test statistic
\begin{equation}\label{stat:Tvk}
\aligned
T_{vk}&:= 2N\(1-|\langle\wh{\fv}_k^{(1)}, \wh{\fv}_k^{(2)}\rangle| \) \\
&\q- \dfrac{N^2}{T_1(N-r)\wh{\lam}_k^{(1)}} \sum_{j=r+1}^N \wh{\lam}_j^{(1)}
 - \dfrac{N^2}{T_2(N-r)\wh{\lam}_k^{(2)}} \sum_{j=r+1}^N  \wh{\lam}_j^{(2)}.
\endaligned
\end{equation}

\begin{thm}\label{thm:test_tv}
Under null hypothesis $H_0^{(III,k)}$ and Assumptions A--C,
suppose further  that  $\fXi_{11}^*:=\left(\lim_{N\to\infty} \langle {\fv}_s^{(1)}, {\fv}_t^{(2)} \rangle\right)_{s,t=1,\ldots,r}$ exist. Then
the proposed test statistic~$T_{vk}$ converges weakly as follows:
	\[
	T_{vk} \ \wk \ \fq_k^\mT
	\begin{pmatrix}
	\fI_{r-1} & -\fXi_{11,-k}^* \\
	(-{\fXi_{11,-k}^*})^\mT & \fI_{r-1}
	\end{pmatrix}
	\fq_k,
	\]
where $\fq_k$ follows a multivariate normal distribution with mean zero and covariance matrix
	\[
	\aligned \fD_k=\diag
      &\( \rho_1\om_{k1}^{(1)}, \ldots, \rho_1\om_{k(k-1)}^{(1)},\rho_1\om_{k(k+1)}^{(1)},\ldots,\rho_1\om_{kr}^{(1)},\right.\\
	&\ \left.\rho_2\om_{k1}^{(2)},\ldots, \rho_2\om_{k(k-1)}^{(2)},\rho_2\om_{k(k+1)}^{(2)},\ldots,\rho_2\om_{kr}^{(2)} \),
    \endaligned
	\]
	\[
	\om_{kj}^{(i)}=
           \dfrac{ \theta_k^{(i)}\theta_j^{(i)} }{( \theta_k^{(i)} -\theta_j^{(i)} )^2 },~~~~
	\textrm{for}~ i=1,2, ~~ j=1,\ldots,k-1,k+1,\ldots,r,
	\]
and
$\fXi_{11,-k}^*$ is the matrix obtained  by deleting the $k$th row and $k$th column of $\fXi_{11}^*$. Furthermore, $\om_{kj}^{(i)}$ and $\fXi_{11}^*$ can  be consistently estimated by
\[
\wh{\om_{kj}^{(i)}} = \dfrac{\wh{\lam}_k^{(i)} \wh{\lam}_j^{(i)}}{(\wh{\lam}_k^{(i)} - \wh{\lam}_j^{(i)} )^2}, \q
\mbox{ and }
\q \wh{\fXi_{11}^*}=\left( \langle \wh{\fv_s^{(1)}}, \wh{\fv_t^{(2)}} \rangle\right)_{s,t=1,\ldots,r},
\]
respectively.
\end{thm}

\begin{cor}\label{cor:test_tvk}
	Under the stronger null hypothesis that
$|\langle \fv_k^{(1)}, \fv_k^{(2)}\rangle|=1$ for all $k=1,\ldots,r$,
we have $\fXi_{11}^*=\fI_{r}$, and the proposed test statistic $T_{vk}$ converges as follows:
\[
T_{vk} \ \wk \ \sum_{j\neq k, j=1}^r \(\rho_1\om_{kj}^{(1)} + \rho_2\om_{kj}^{(2)} \)
\cdot Z_j^2,
\]
where $Z_j$'s are \hbox{i.i.d.} standard normal random variables.
\end{cor}

Theorem \ref{thm:test_tv} and Corollary \ref{cor:test_tvk} are proved in the supplementary material.

\section{Simulation Studies \label{sec:simulation}}

\subsection{Design}
We consider five population covariance matrices as follows:
\begin{eqnarray*}
\fSi_1 &=& \fV^{(1)} \fLa^{(1)} (\fV^{(1)})^T ,
~~~
\fSi_2 = \fV^{(1)} \fLa^{(2)} (\fV^{(1)})^T ,
~~~
\fSi_3 = \fV^{(2)} \fLa^{(1)} (\fV^{(2)})^T ,
\\
\fSi_4 &=& \fLa^{(1)},
~~~~~~~~~~~~~~~~~~
\fSi_5 = \fV^{(5)} \fLa^{(1)} (\fV^{(5)})^T ,
\end{eqnarray*}
where $\fV^{(1)}$, $\fV^{(2)}$ are two random orthogonal matrices, and
\begin{eqnarray*}
\fLa^{(1)} &=& \diag(5N/2, N, N/2, \lam_4^{(1)}, \cdots, \lam_N^{(1)}),
~~~ \lam_j^{(1)} \iid \mbox{Unif}(1,3), \\
\fLa^{(2)} &=& \diag(7N/2, 2N, N, \lam_4^{(2)}, \cdots, \lam_N^{(2)}),
~~~~ \lam_j^{(2)} \iid \mbox{Unif}(2,5), \\
\fV^{(5)} &=& (\fv_1^{(5)}, \fv_2^{(5)}, \fe_3, \cdots, \fe_N), \\
\fv_1^{(5)} &=& (\cos\theta, \sin\theta, 0, \cdots, 0)^T, \\
\fv_2^{(5)} &=& (-\sin\theta, \cos\theta, 0, \cdots, 0)^T,
~~~~ \theta \in [0,\pi/2].
\end{eqnarray*}
The observations will be simulated as the following: for a given $\fSi$, which will be one of the five covariance matrices above, write its  spectral decomposition as $\fSi= \fV \fLa \fV^T$. Then we simulate observations with covariance matrix $\fSi$ by $\fV \fLa^{1/2}  \fz_t$, where
 $\fz_t = (\fz_t[1], \fz_t[2], \ldots, \fz_t[N] )^\mT$ consists of $\hbox{i.i.d.}$ standardized student $t(8)$ random variables.

Theorems \ref{thm:test_tlam}, \ref{thm:test_te} and \ref{thm:test_tv} are associated with different null hypotheses.  When evaluating the sizes of the tests proposed in these theorems,  we adopt the following setting:
\begin{itemize}
  \item For both Theorems \ref{thm:test_tlam} and \ref{thm:test_te}, we simulate two samples of observations with $\fSi_1$ and $\fSi_3$ as their respective population covariance matrices. Note that  $\fSi_1$ and $\fSi_3$ share the same eigenvalues but have different eigenvectors.
  \item For Theorem \ref{thm:test_tv}, the two samples of observations are simulated with $\fSi_1$ and $\fSi_2$ as their respective population covariance matrices. The two matrices share the same eigenvectors but have different eigenvalues.
\end{itemize}

On the other hand, when evaluating powers, we use the following design:
\begin{itemize}
  \item For testing equality of eigenvalues/eigenvalue ratios, the two samples of observations are simulated with $\fSi_1$ and $\fSi_2$ as their respective population covariance matrices;
\item For testing equality of principal eigenvectors,  we simulate two samples of observations with $\fSi_4$ and $\fSi_5$ as their respective population covariance matrices. The difference between the principal eigenvectors of the two matrices is a function of the angle $\theta$. We will change the value of $\theta$ to see how the power varies as a function of~$\theta$.
\end{itemize}

\subsection{Visual check}
We firstly visually examine Theorems \ref{thm:test_tlam}, \ref{thm:test_te} and \ref{thm:test_tv} by comparing the empirical distributions of the test statistics with their respective asymptotic distributions under the null hypotheses.

For Theorem \ref{thm:test_tlam}, the asymptotic distribution of the test statistic $T_{\lam 1}$ is the standard normal distribution. This is clearly supported by Figure \ref{fig:dist_value_t8}, which give the normal Q-Q plot and histogram of $T_{\lam 1}$ based on 5,000 replications.

\begin{figure}[htbp]
\centering	
\includegraphics[height=6cm]{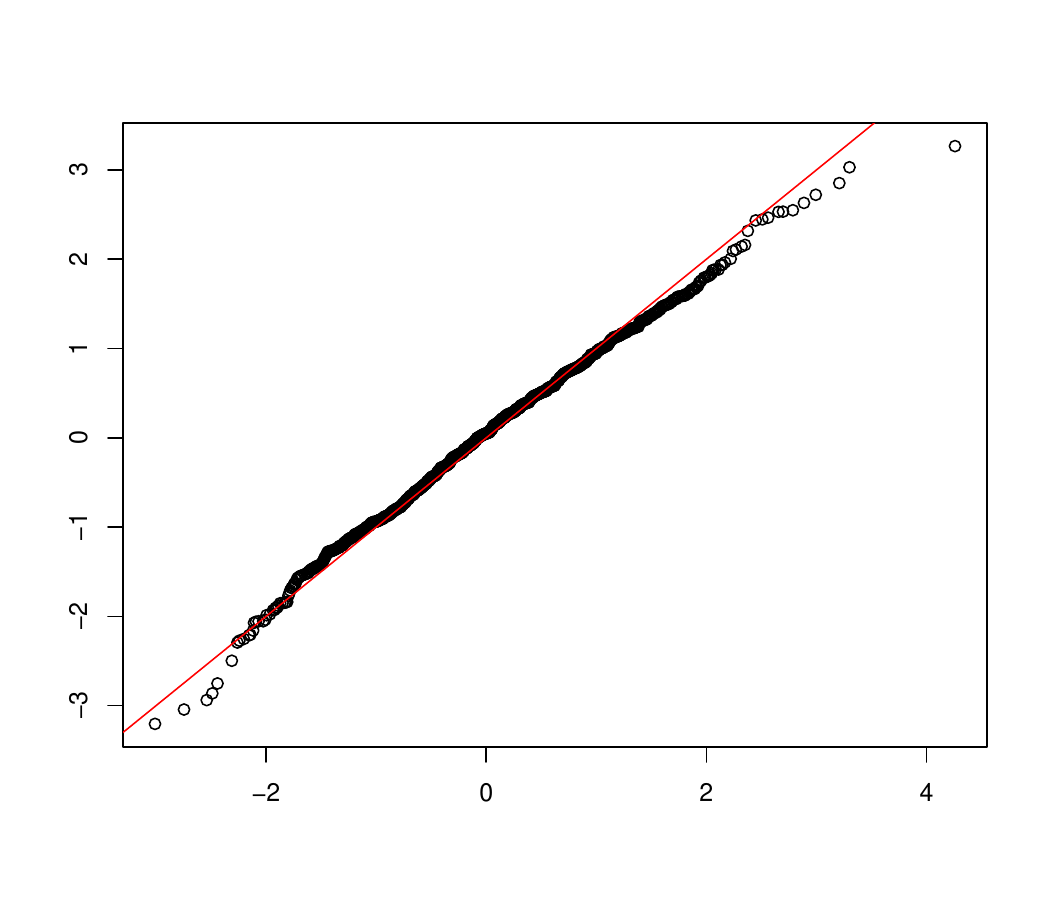}
\includegraphics[height=6cm]{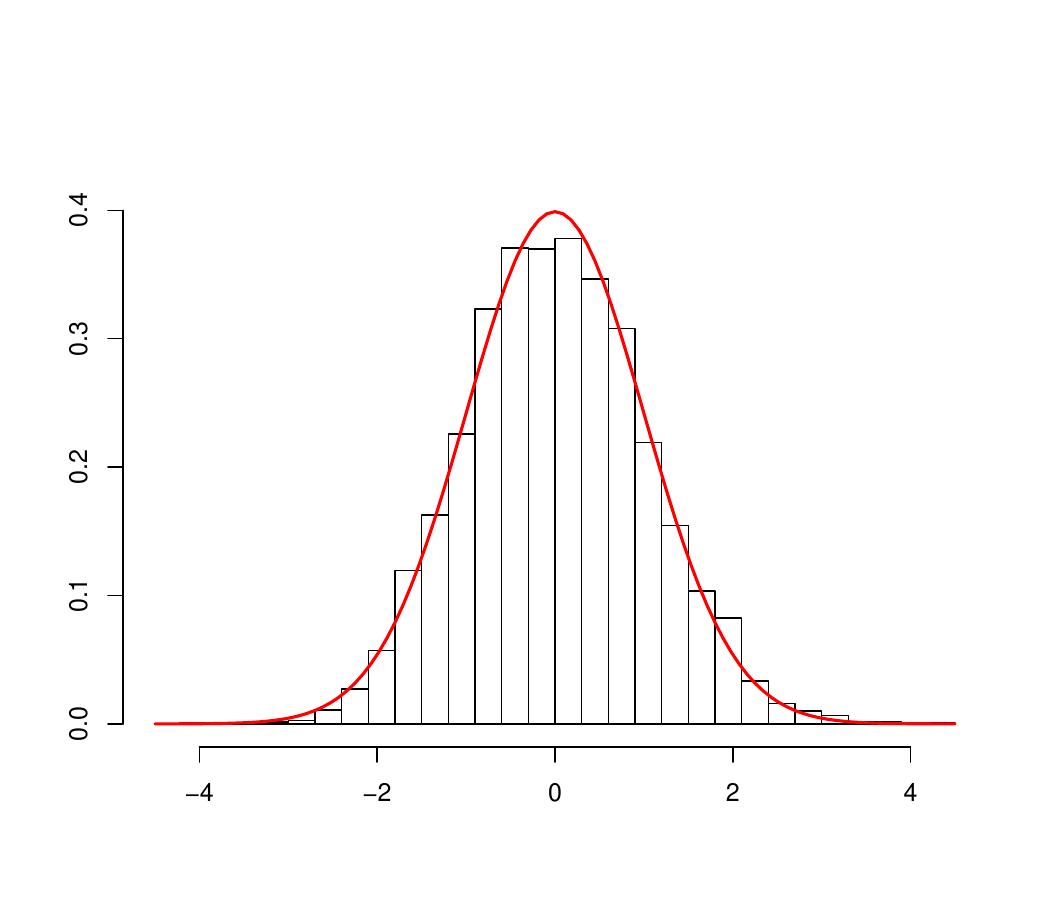}
\caption{Normal Q-Q plot and histogram of $T_{\lam 1}$ based on 5,000 replications with $N=500, T_1=500$ and $T_2=750$.  }\label{fig:dist_value_t8}
\end{figure}

For Theorem \ref{thm:test_te}, the asymptotic distribution of the test statistic $T_{e1}$ is again the standard normal distribution. This is supported by Figure \ref{fig:dist_ratio_t8}.

\begin{figure}[htbp]
\centering	
\includegraphics[height=6cm]{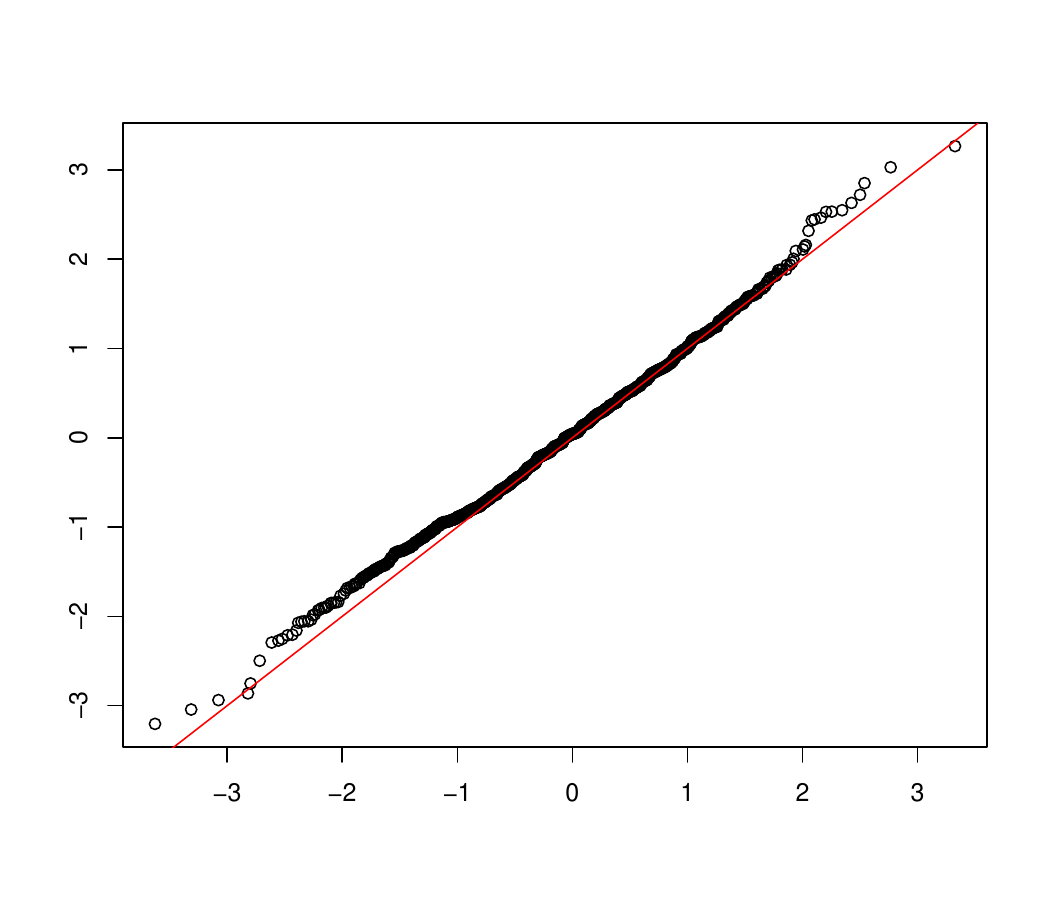}
\includegraphics[height=6cm]{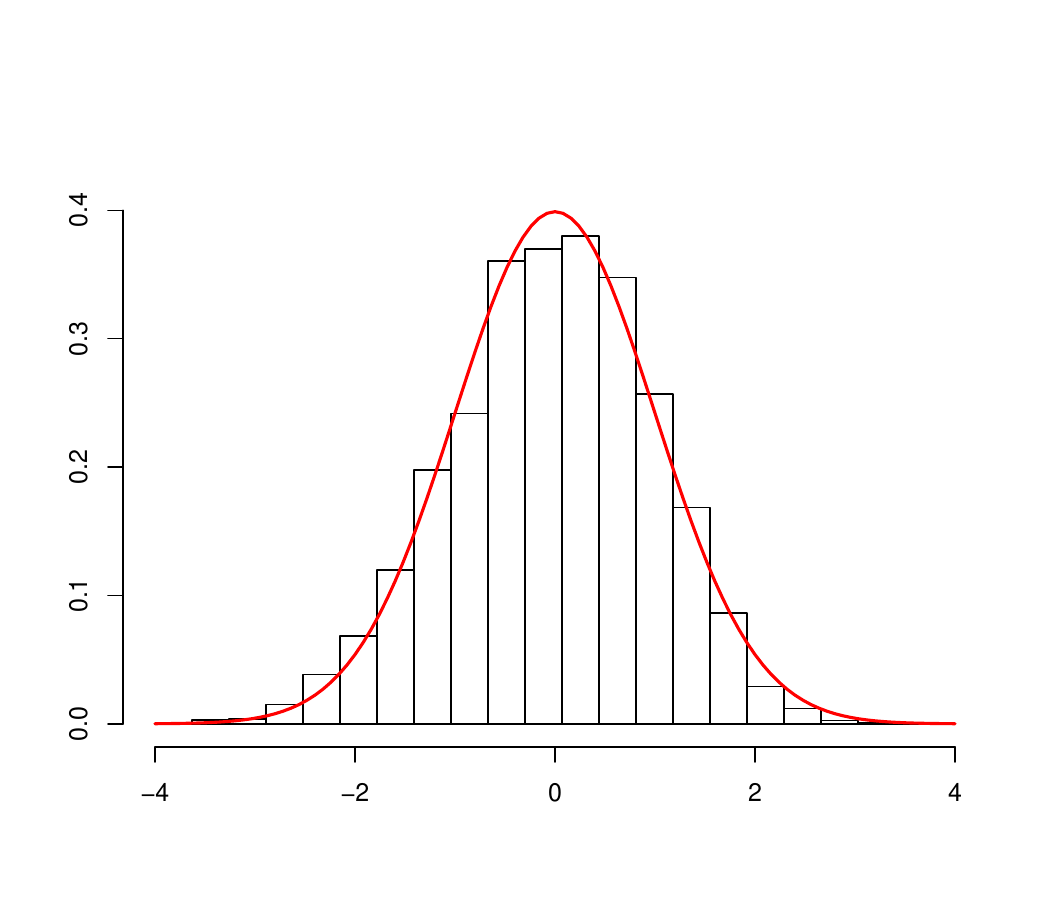}
\caption{Normal Q-Q plot and histogram of $T_{e1}$ based on 5,000 replications with $N=500, T_1=500$ and $T_2=750$. }\label{fig:dist_ratio_t8}
\end{figure}

For Theorem \ref{thm:test_tv}, the asymptotic distribution of the test statistic $T_{v 1}$ is a generalized $\chi^2$-distribution, which does not have an explicit density formula. To examine the asymptotics, we compare the empirical distribution of the test statistic $T_{v 1}$ with that of  Monte-Carlo samples from the asymptotic distribution. The comparison is conducted via both Q-Q plot and  density estimation. The results are given in  Figure~\ref{fig:dist_vector_t8}. We can see that the empirical distribution of the test statistic $T_{v 1}$ match well with the asymptotic distribution.

\begin{figure}[htbp]
\centering	
\includegraphics[height=6cm]{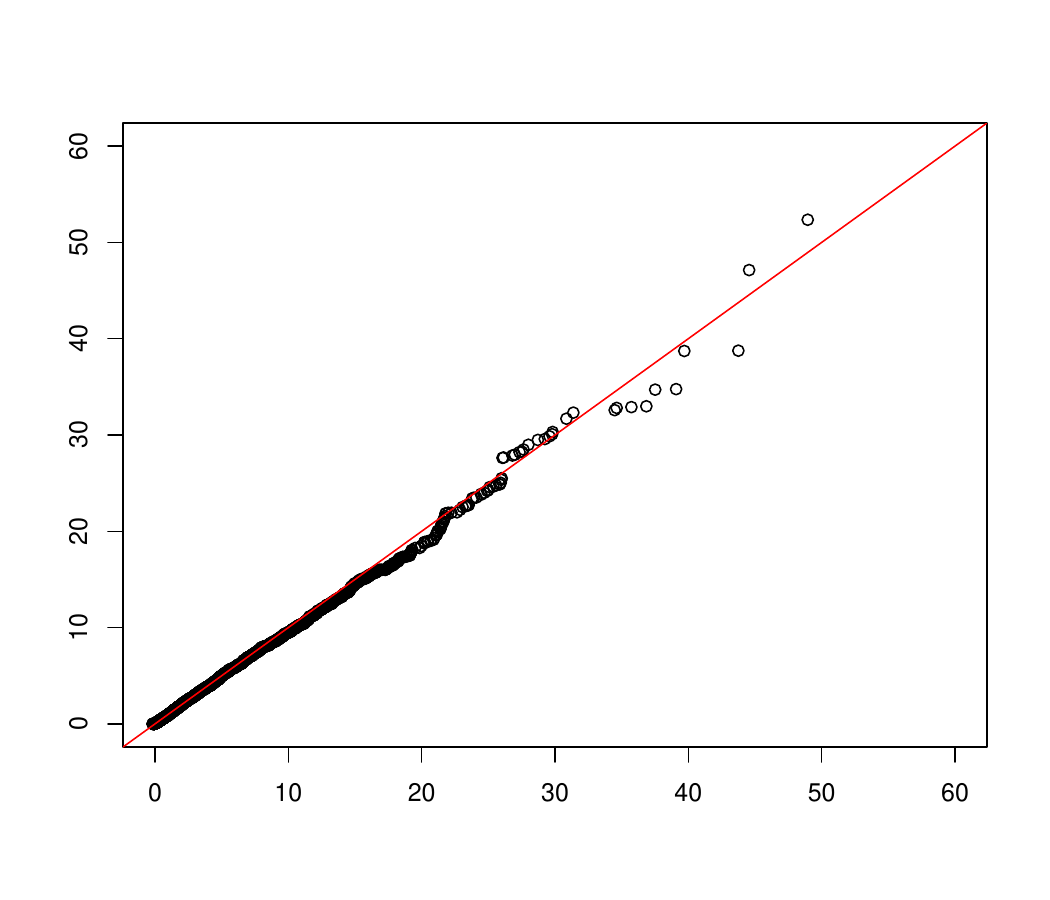}
\includegraphics[height=6cm]{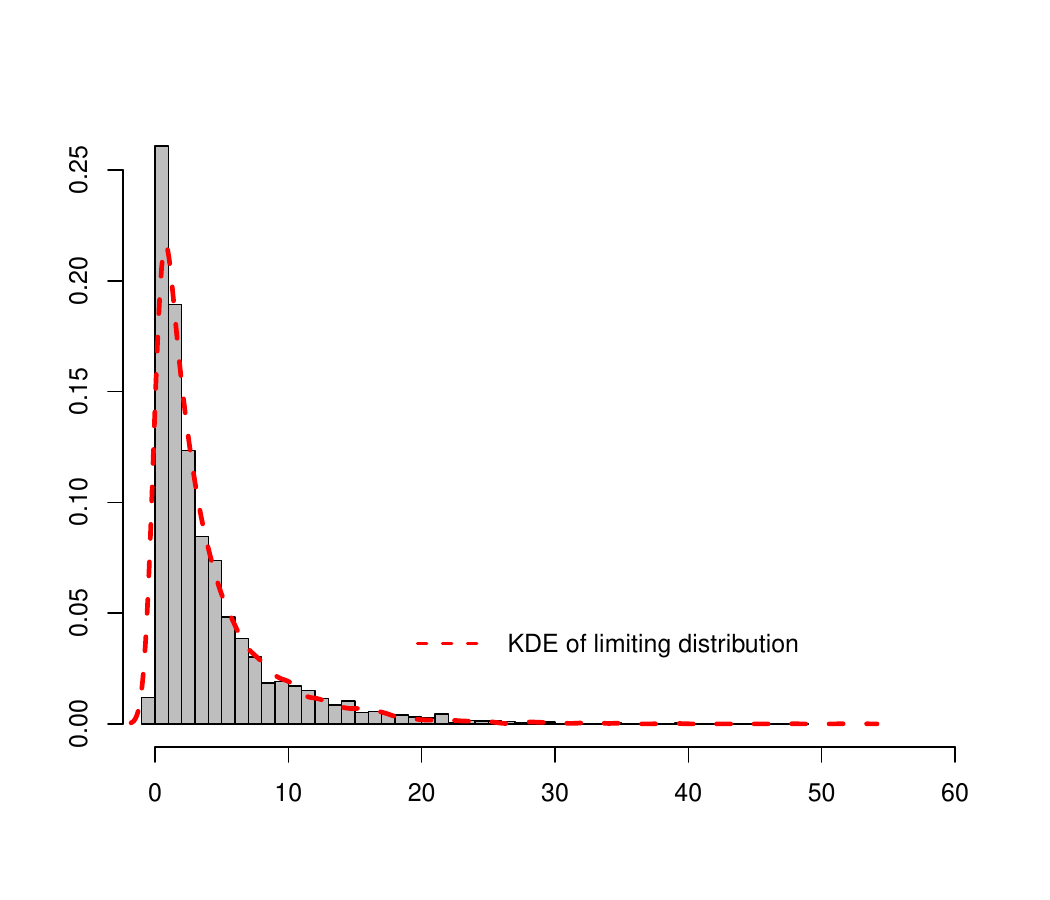}
\caption{Comparisons of the empirical distribution of the test statistic $T_{v 1}$ with the asymptotic distribution when $N=500, T_1=500$ and $T_2=750$. Left: Q-Q plot of $T_{v 1}$ versus Monte-Carlo samples from the asymptotic distribution; right: histogram of $T_{v 1}$ versus the kernel density estimate of the asymptotic distribution.  }\label{fig:dist_vector_t8}
\end{figure}

\subsection{Size and power evaluation}
In this subsection, we evaluate the sizes and powers of the three tests in Theorems~\ref{thm:test_tlam}, \ref{thm:test_te} and \ref{thm:test_tv}.

Table \ref{table:size_t8} reports the empirical sizes of the three tests based on $T_{\lam k}, T_{e k}$ and $T_{v k}$, $k=1,2,3,$ at 5\% significance level for different combinations of $N$, $T_1$ and $T_2$. Tests based on $T_{e k }$ and $T_{v k}$, $k=1,2,3,$ involve the number of factors, which is unknown in practice. There are several estimators available, including those given in \cite{BN02} and \cite{AH13}.
We evaluate the sizes based on a given estimated number of factors, specified by $\wh{r}$ in the table.
 We see that for the first two sets of tests, for different estimated number of factors and  different $N$ and $T_1, T_2$, the empirical sizes are close to the nominal level of 5\%. For the third set of tests based on $T_{v i}$, $i=1,2,3,$, the size approaches 5\% as the dimension $N$ and samples sizes $T_1, T_2$ get larger.

\begin{table}[H]
\small
\centering
\begin{tabular}{c c c c c c c c c c}
\hline
\multirow{2}{*}{$N$} & \multirow{2}{*}{$T_{\lam 1}$} &&  \multicolumn{3}{c}{$T_{e1}$} & & \multicolumn{3}{c}{$T_{v1}$} \\
\cline{4-6} \cline{8-10}
& & & $\wh{r}=2$ & $\wh{r}=3$ (true) & $\wh{r}=4$ & &  $\wh{r}=2$ & $\wh{r}=3$ (true) & $\wh{r}=4$ \\
\hline
100 & $0.052\ $ && 0.052 & 0.050 & 0.050 & & 0.083 & 0.086 & 0.086 \\
300 & $0.052\ $ &&  0.051 & 0.049 & 0.049 & & 0.060 & 0.063 & 0.063 \\
500 & $0.053\ $ && 0.053 & 0.051 & 0.051 & & 0.051 & 0.055 & 0.055 \\
\hline
\end{tabular}

\begin{tabular}{c c c c c c c c c c}
\hline
\multirow{2}{*}{$N$} & \multirow{2}{*}{$T_{\lam 2}$} &&  \multicolumn{3}{c}{$T_{e2}$} & & \multicolumn{3}{c}{$T_{v2}$} \\
\cline{4-6} \cline{8-10}
& & & $\wh{r}=2$ & $\wh{r}=3$ (true) & $\wh{r}=4$ & &  $\wh{r}=2$ & $\wh{r}=3$ (true) & $\wh{r}=4$ \\
\hline
100 & $0.057\ $ && 0.048 & 0.047 & 0.047 & & 0.102 & 0.108 & 0.109 \\
300 & $0.053\ $ &&  0.055 & 0.054 & 0.054 & & 0.057 & 0.063 & 0.063\\
500 & $0.056\ $ && 0.053 & 0.052 & 0.052 & & 0.048 & 0.055 & 0.055 \\
\hline
\end{tabular}

\begin{tabular}{c c c c c c c c c c}
\hline
\multirow{2}{*}{$N$} & \multirow{2}{*}{$T_{\lam 3}$} &&  \multicolumn{3}{c}{$T_{e3}$} & & \multicolumn{3}{c}{$T_{v3}$} \\
\cline{4-6} \cline{8-10}
& & & $\wh{r}=2$ & $\wh{r}=3$ (true) & $\wh{r}=4$ & &  $\wh{r}=2$ & $\wh{r}=3$ (true) & $\wh{r}=4$ \\
\hline
100 & $0.065\ $ && NA & 0.049 & 0.049 & & NA & 0.096 & 0.099 \\
300 & $0.052\ $ &&  NA & 0.052 & 0.052 & & NA & 0.058 & 0.059 \\
500 & $0.062\ $ && NA & 0.055 & 0.055 & & NA & 0.057 & 0.057 \\
\hline
\end{tabular}
\caption{ Empirical sizes based on 5,000 replications of $T_{\lam k}, T_{e k}$ and $T_{v k}$, $k=1,2,3,$ at 5\% significance level with $N/T_1=1$ and $N/T_2=2/3$.}
\label{table:size_t8}
\end{table}

Power evaluation results are given in Table \ref{table:power_value_t8}. We see that the powers are in general quite high especially as the dimension $N$ and samples sizes $T_1, T_2$ all get larger.
\begin{table}[H]
\small
\centering
\begin{tabular}{c c c c c c c c c c}
\hline
\multirow{2}{*}{$N$} & \multirow{2}{*}{$T_{\lam 1}$} &&  \multicolumn{3}{c}{$T_{e1}$} & & \multicolumn{3}{c}{$T_{v1}$} \\
\cline{4-6} \cline{8-10}
& & & $\wh{r}=2$ & $\wh{r}=3$ (true) & $\wh{r}=4$ & &  $\wh{r}=2$ & $\wh{r}=3$ (true) & $\wh{r}=4$ \\
\hline
100 & $0.196\ $ && 0.146 & 0.138 & 0.138 && 0.494 & 0.508 & 0.509\\
300 & $0.617\ $ && 0.457 & 0.450 &  0.450 && 0.909 & 0.913 & 0.914\\
500 & $0.842\ $ && 0.705 & 0.697 &  0.697 && 0.987 & 0.988 & 0.988\\
\hline
\end{tabular}

\begin{tabular}{c c c c c c c c c c}
\hline
\multirow{2}{*}{$N$} & \multirow{2}{*}{$T_{\lam 2}$} &&  \multicolumn{3}{c}{$T_{e2}$} & & \multicolumn{3}{c}{$T_{v2}$} \\
\cline{4-6} \cline{8-10}
& & & $\wh{r}=2$ & $\wh{r}=3$ (true) & $\wh{r}=4$ & &  $\wh{r}=2$ & $\wh{r}=3$ (true) & $\wh{r}=4$ \\
\hline
100 & $0.712\ $ && 0.160 & 0.158 & 0.158 && 0.407 & 0.428 &  0.430\\
300 & $0.989\ $ && 0.360 & 0.355 &  0.355 && 0.833 & 0.844 & 0.844\\
500 & $0.996\ $ && 0.522& 0.520 &  0.520 && 0.970 & 0.974 & 0.974\\
\hline
\end{tabular}

\begin{tabular}{c c c c c c c c c c}
\hline
\multirow{2}{*}{$N$} & \multirow{2}{*}{$T_{\lam 3}$} &&  \multicolumn{3}{c}{$T_{e3}$} & & \multicolumn{3}{c}{$T_{v3}$} \\
\cline{4-6} \cline{8-10}
& & & $\wh{r}=2$ & $\wh{r}=3$ (true) & $\wh{r}=4$ & &  $\wh{r}=2$ & $\wh{r}=3$ (true) & $\wh{r}=4$ \\
\hline
100 & $0.705\ $ && NA & 0.145 & 0.145 && NA & NA & NA\\
300 & $0.990\ $ && NA & 0.303 &  0.303&& NA & NA & NA\\
500 & $1\ $ && NA & 0.446 &  0.446&& NA & NA & NA\\
\hline
\end{tabular}
\caption{ Empirical powers based on 5,000 replications of $T_{\lam k}, T_{e k}$ and $T_{v k}$ (for $\theta=\pi/9$), $k=1,2,3,$ at 5\% significance level with $N/T_1=1$ and $N/T_2=2/3$.}
\label{table:power_value_t8}
\end{table}

Finally, in Figure \ref{fig:power_Tv_theta}, we evaluate the power of the eigenvector test $T_{v k}, k=1,2$ as a function of~$\theta$. For the three $\theta$ values tested, $i\pi/9, i =1,2,3$, the bigger the value, the bigger the difference between the principal eigenvectors in the two populations, and the higher the power. Moreover, even for the smallest value $\pi/9$, the power quickly increases to close to 1 as the dimension $N$ and samples sizes $T_1, T_2$  get larger.

\begin{figure}[htbp]
\centering	
\includegraphics[width=0.49 \textwidth]{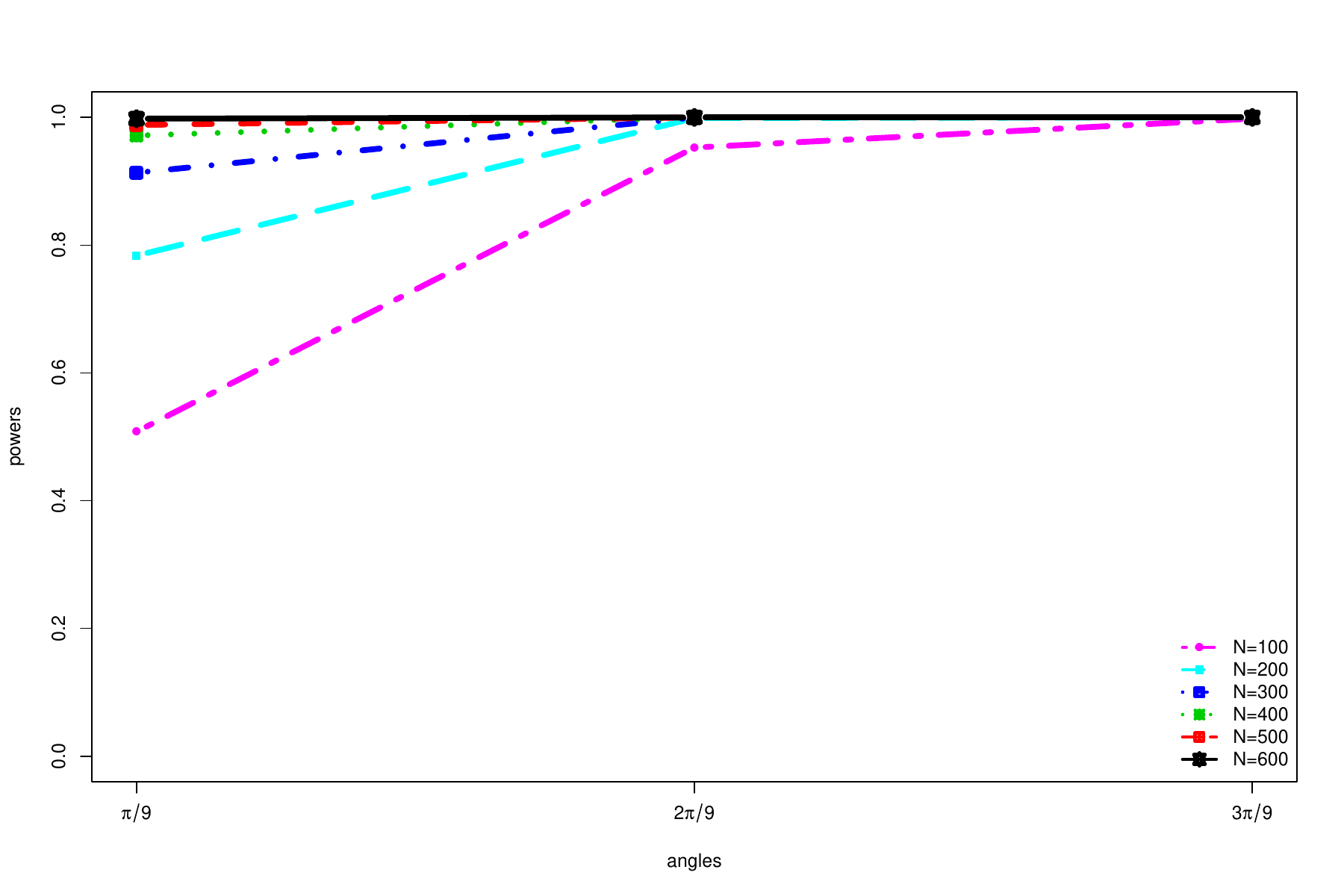}
\includegraphics[width=0.49 \textwidth]{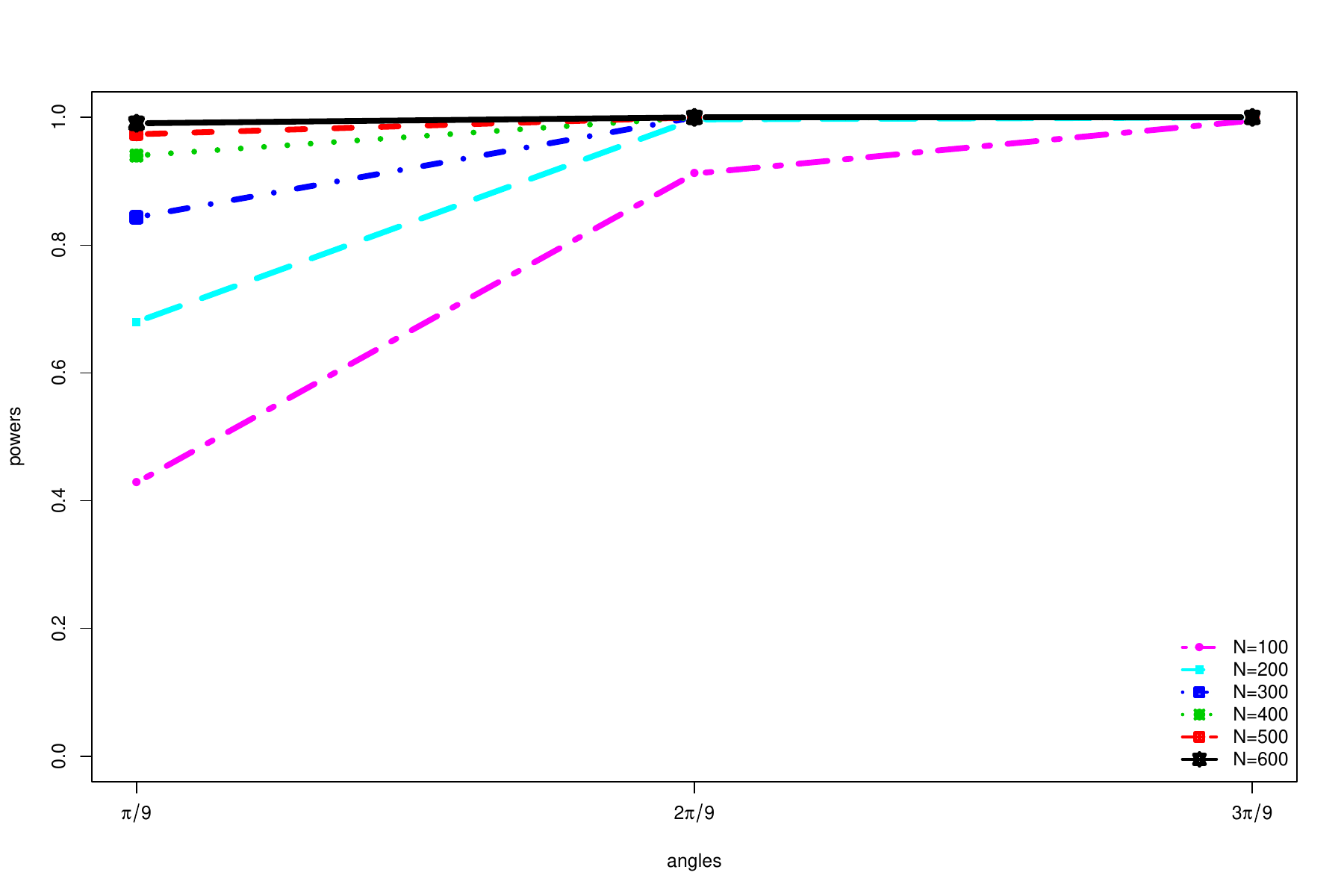}
\caption{Empirical power of $T_{v_i}, i=1,2,$ as a function of $\theta$ at 5\% significance level for different  $N$ and $T_1, T_2$ with $N/T_1=1$ and $N/T_2=2/3$. Left: powers for $T_{v_1}$; right: powers for $T_{v_2}$.}
\label{fig:power_Tv_theta}
\end{figure}


\section{Empirical Studies \label{sec:emp}}

In this section, we conduct empirical studies based on daily returns of S\&P500 Index constituents from January 2000 to December 2020. The objective is to test, between two consecutive years, whether the principal  eigenvalues, eigenvalue ratios and principal eigenvectors are equal or not.

\subsection{Tests about principal eigenvalues}
We plot in Figure \ref{fig:test_value_byyear} the values of the test statistic, $T_{\lambda k}, k =1,2,3,$ together with the critical values at 5\% significance level based on Theorem \ref{thm:test_tlam}.

\begin{figure}[htbp]
\begin{center}
\centering
\includegraphics[width=0.6\textwidth]{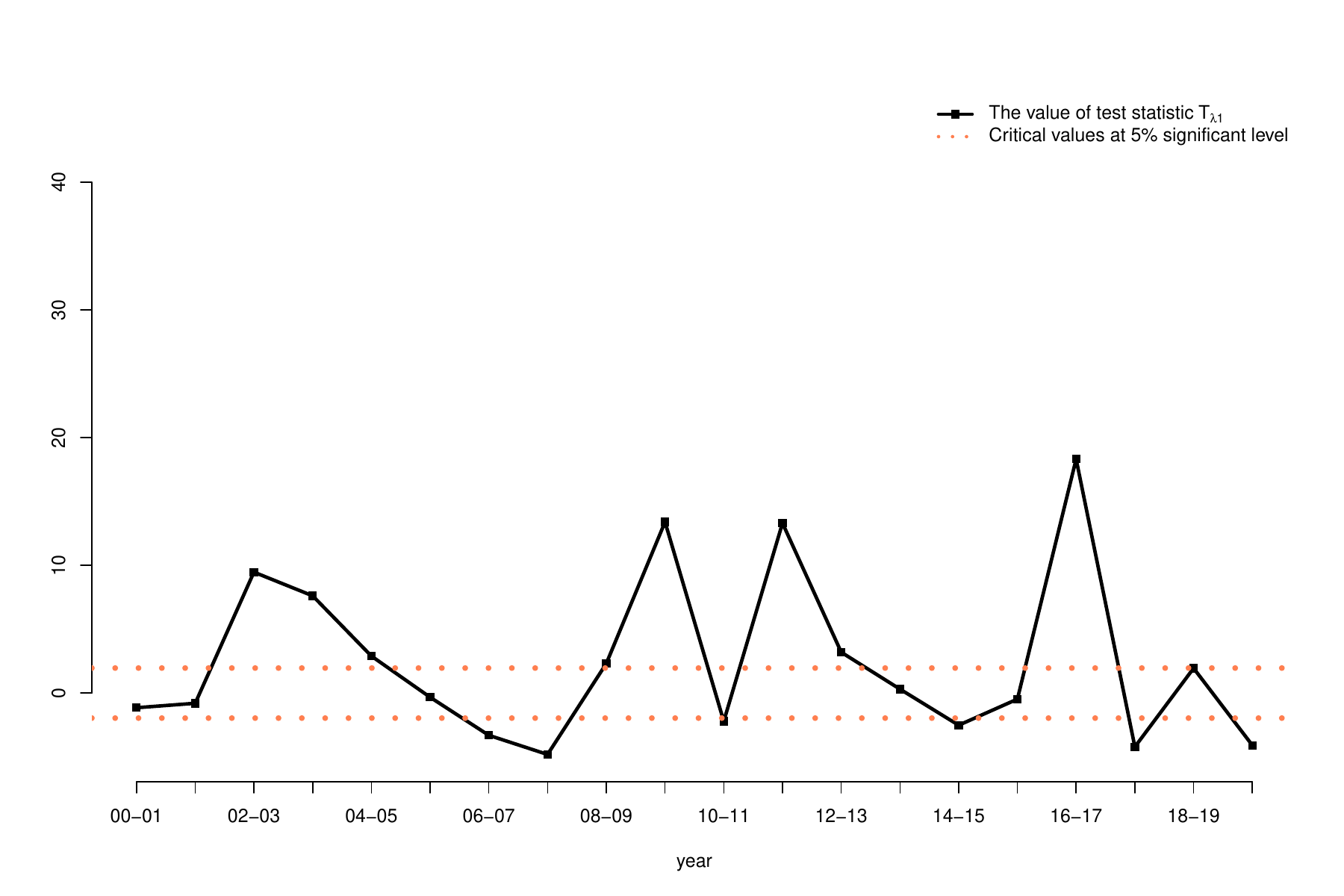}\\
\includegraphics[width=0.6\textwidth]{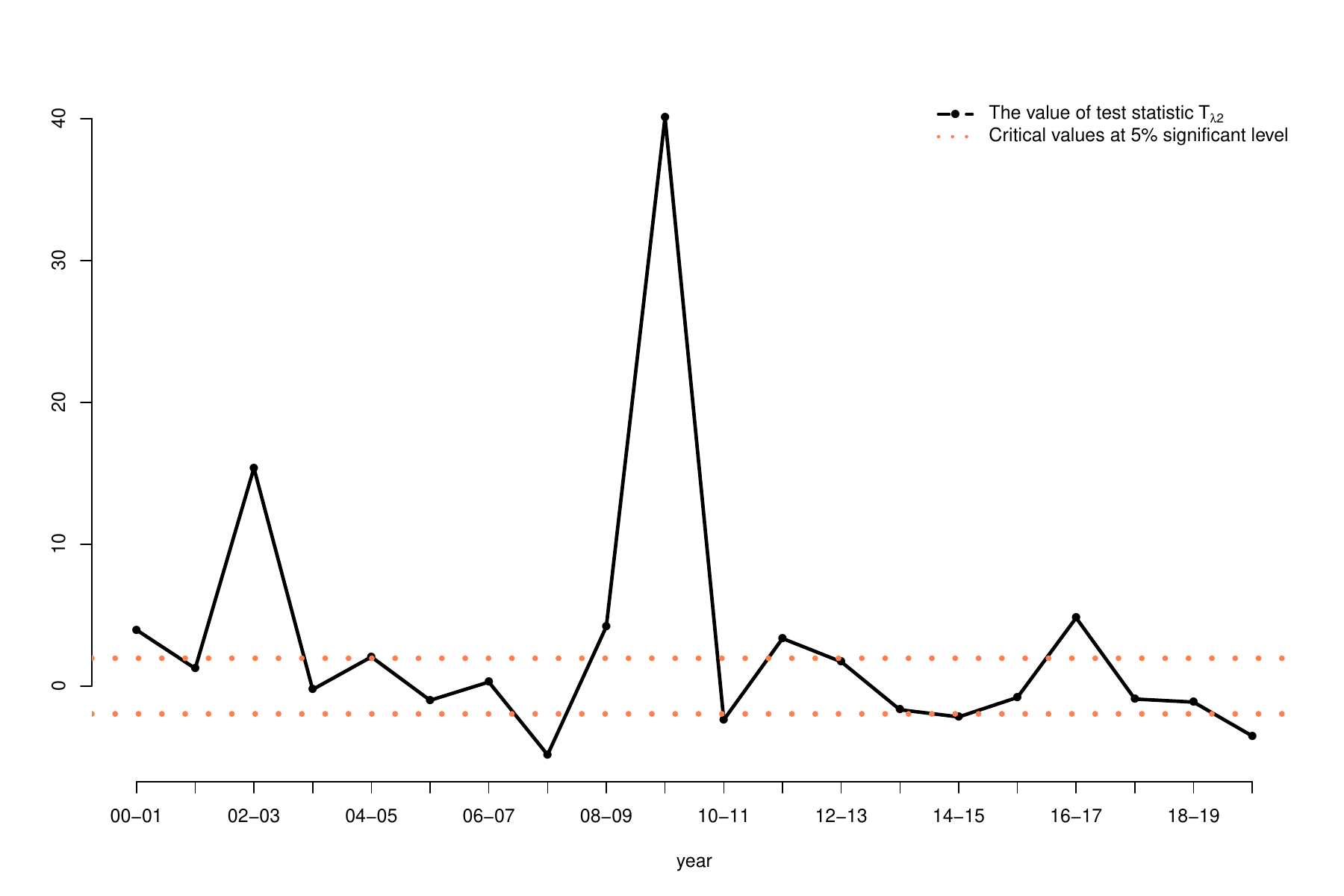}\\
\includegraphics[width=0.6\textwidth]{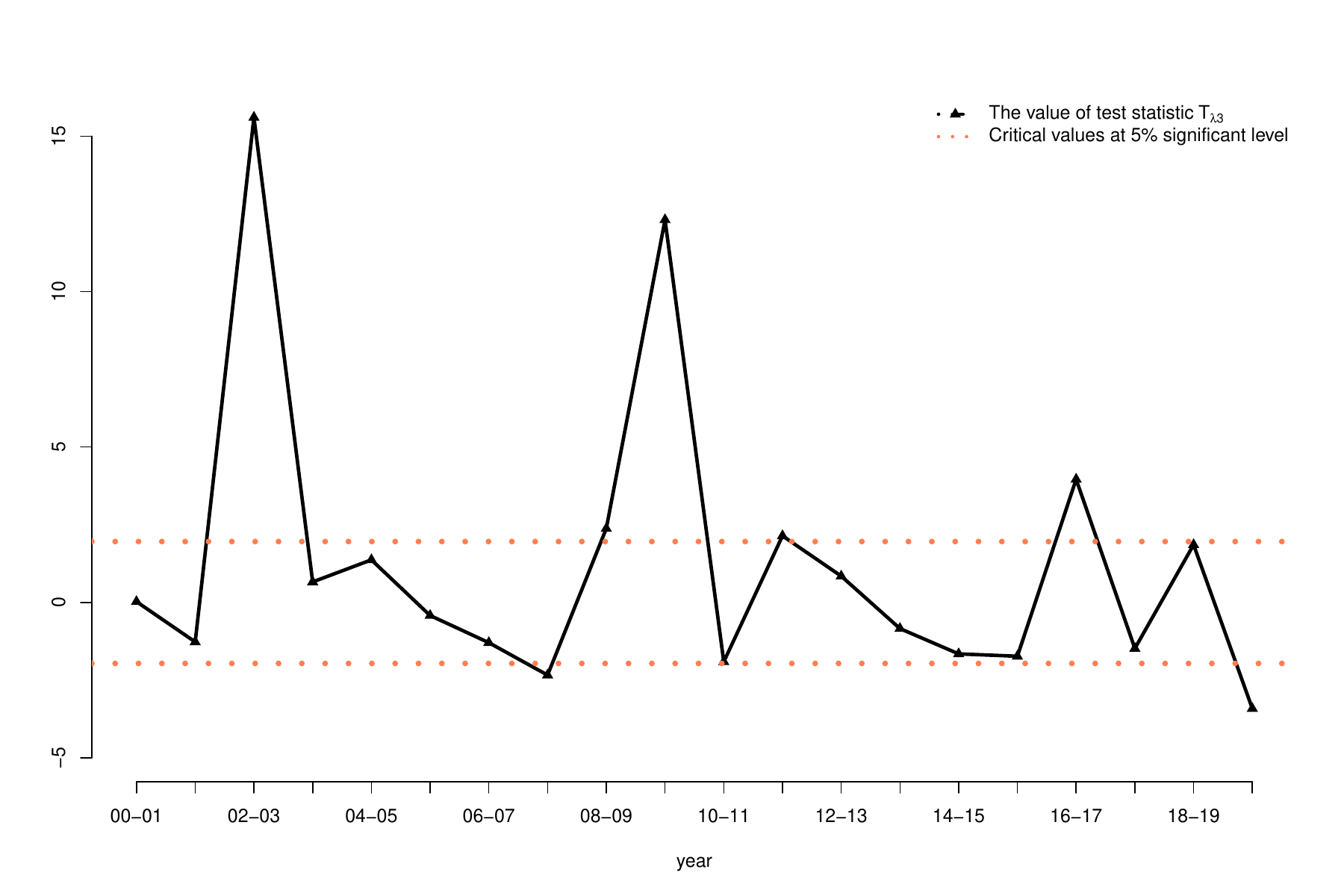}
\caption{Results of testing for equality of the first three principal eigenvalues between two consecutive years during 2000-2020. From top to bottom: testing equality of the first  principal eigenvalue, the second  principal eigenvalue, and the third  principal eigenvalue, respectively.  }
\label{fig:test_value_byyear}
\end{center}
\end{figure}

We see from Figure \ref{fig:test_value_byyear} that for testing equality of the first principal eigenvalue, the test result is statistically significant for more than half of two consecutive years, suggesting that the first principal eigenvalue tends to change over time. The second and third  principal eigenvalues seem a bit more stable.

\subsection{Tests on  eigenvalues ratios}

We plot in Figure \ref{fig:test_ratio_byyear} the  results  of testing equality of eigenvalue ratios.
\begin{figure}[htbp]
\begin{center}
\centering
\includegraphics
[width=0.6\textwidth]{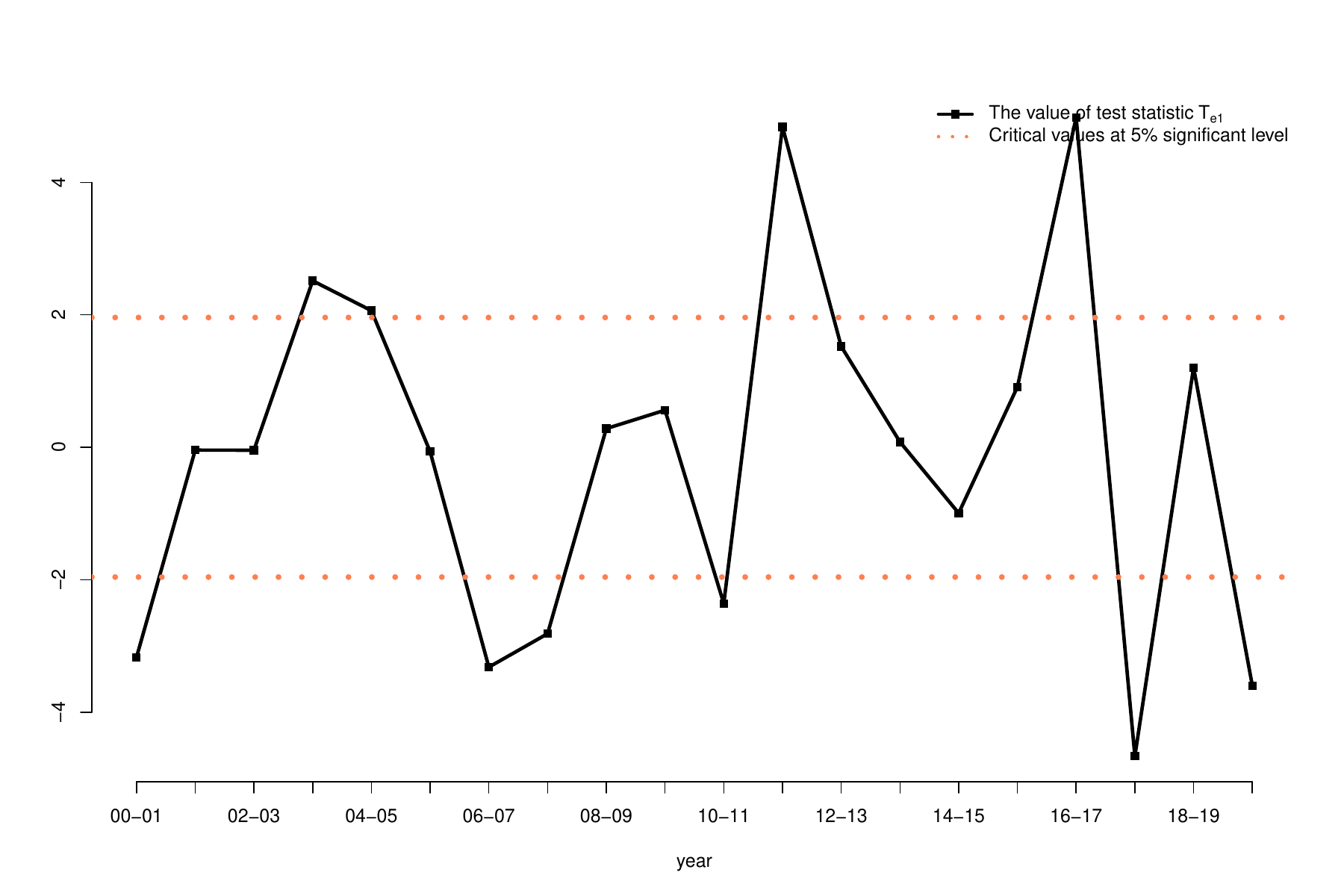}
\includegraphics
[width=0.6\textwidth]{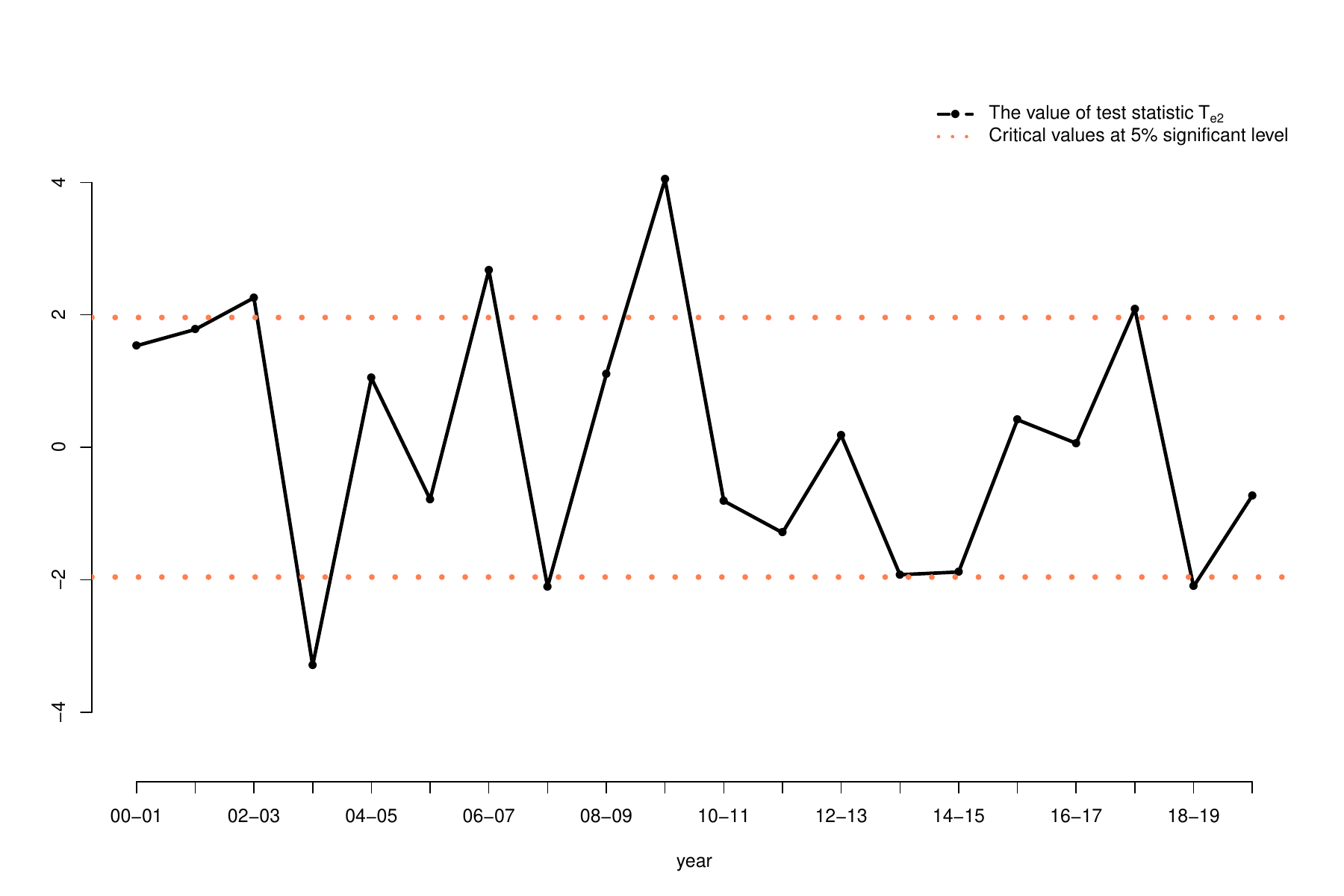}
\includegraphics
[width=0.6\textwidth]{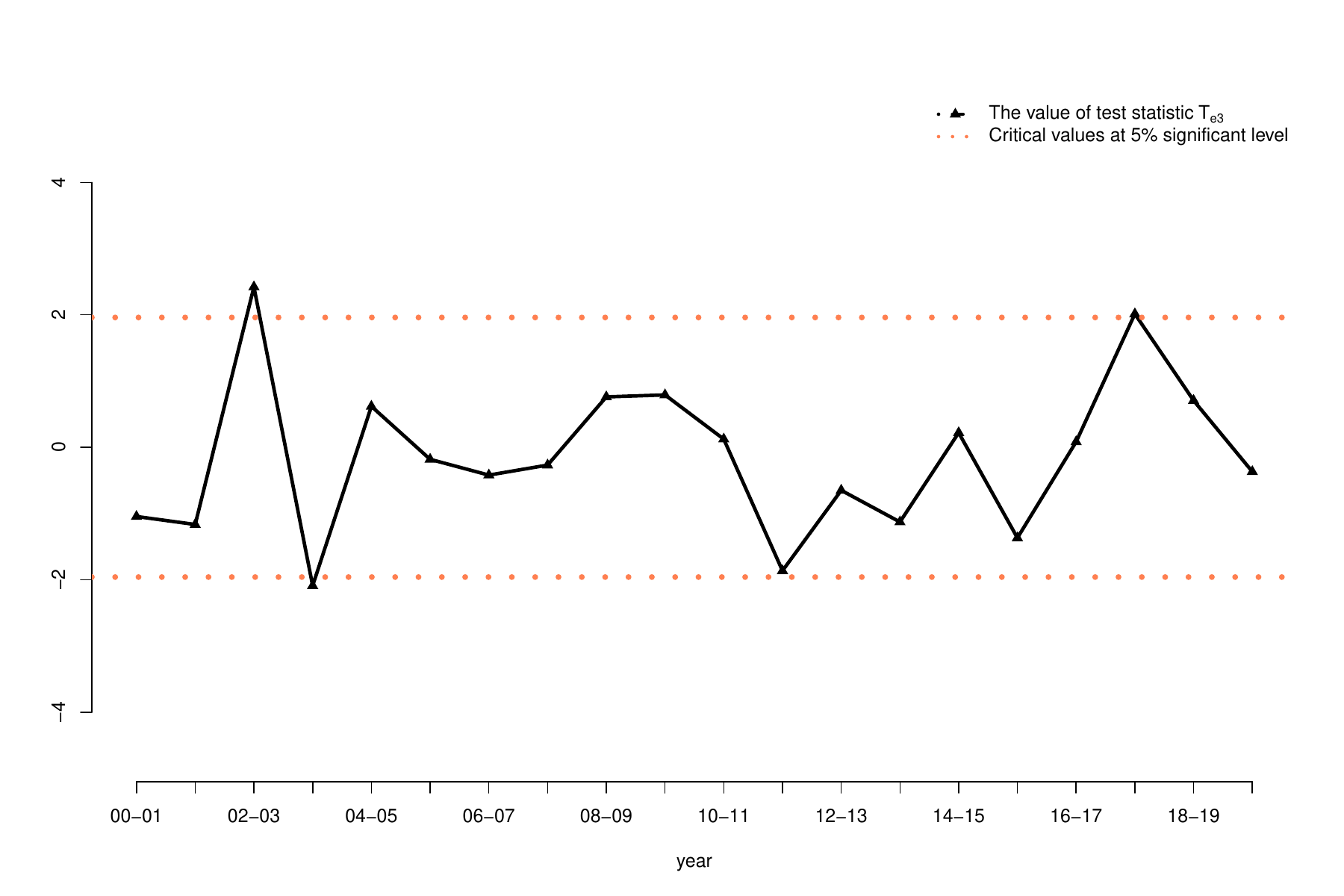}
\caption{Results of testing for equality of the first three principal eigenvalue ratios between two consecutive years during 2000-2020. From top to bottom: testing equality of the first  principal eigenvalue ratio, the second  principal eigenvalue ratio, and the third  principal eigenvalue ratio, respectively.}
\label{fig:test_ratio_byyear}
\end{center}
\end{figure}

An interesting observation is that, in sharp contrast with the tests about eigenvalues, for testing equality of eigenvalue ratios, the rejection rate  is much lower.
Such contrast suggests an interesting difference between the absolute sizes of principal eigenvalues and their relative sizes: while the absolute size appears to change frequently over time, the relative size is more stable.


\subsection{Tests about principal eigenvectors}

Figure \ref{fig:test_vector_byyear} reports the results of
tests about principal eigenvalues.

\begin{figure}[htbp]
\begin{center}
\centering
\includegraphics
[width=0.6\textwidth]{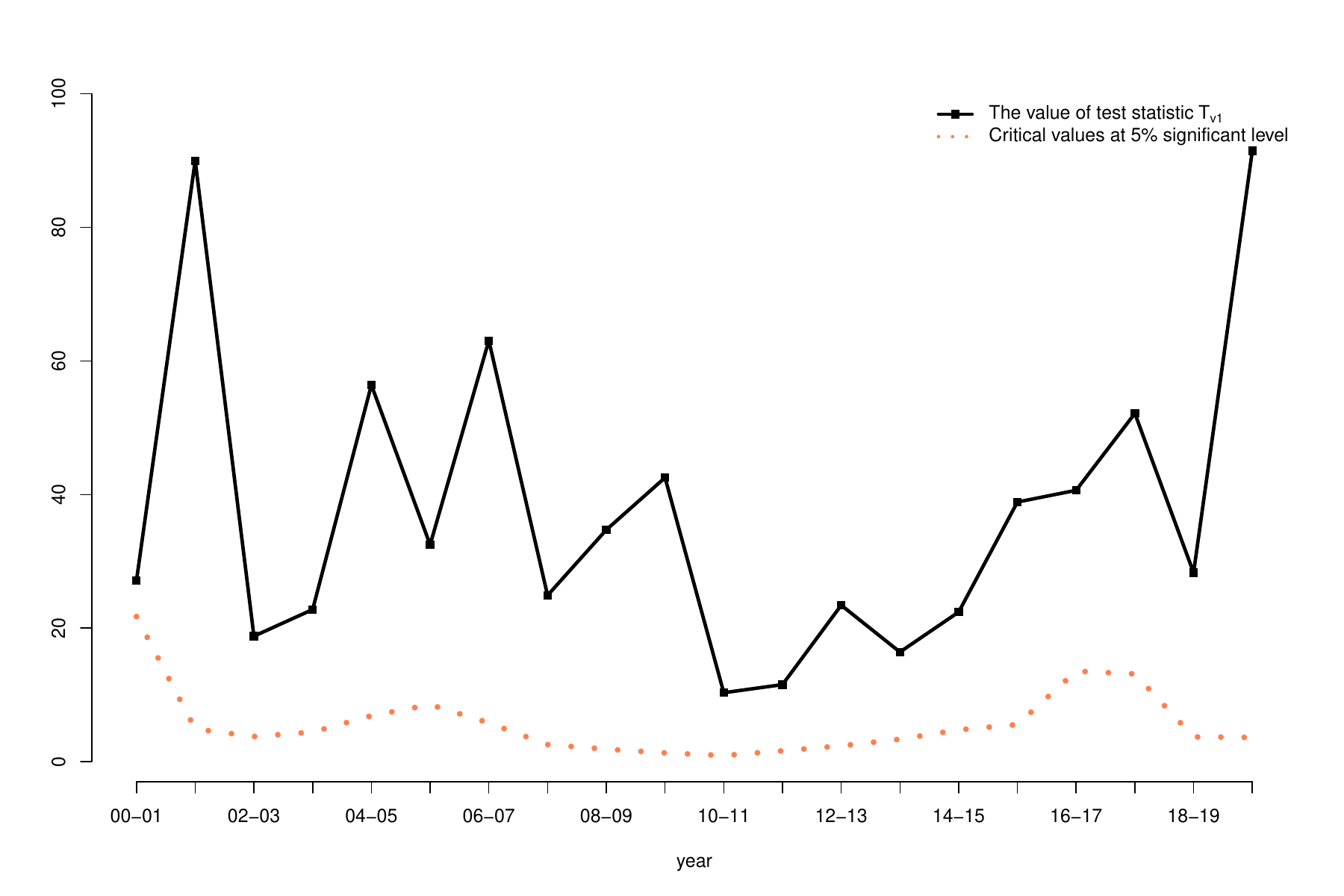}
\includegraphics
[width=0.6\textwidth]{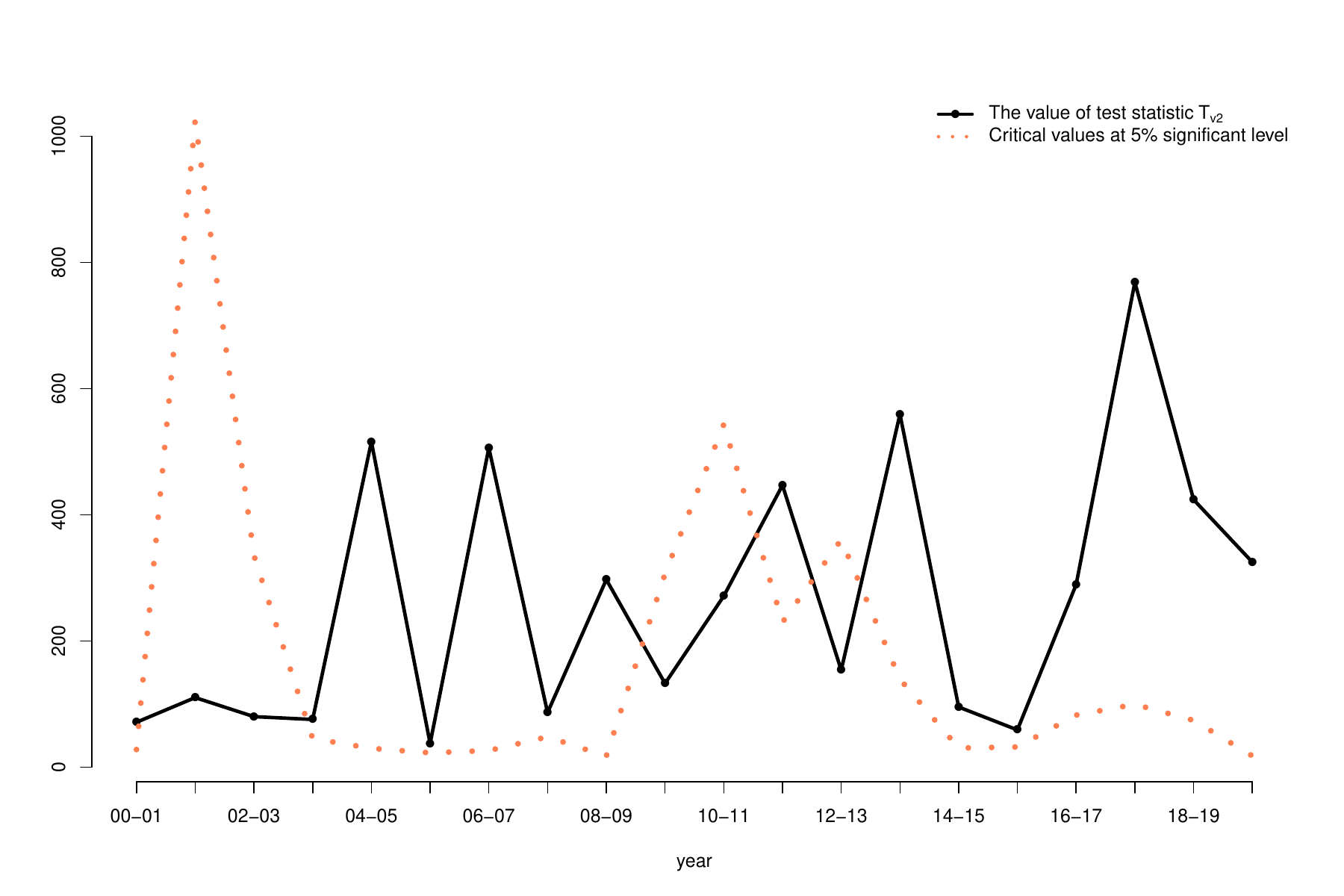}
\includegraphics
[width=0.6\textwidth]{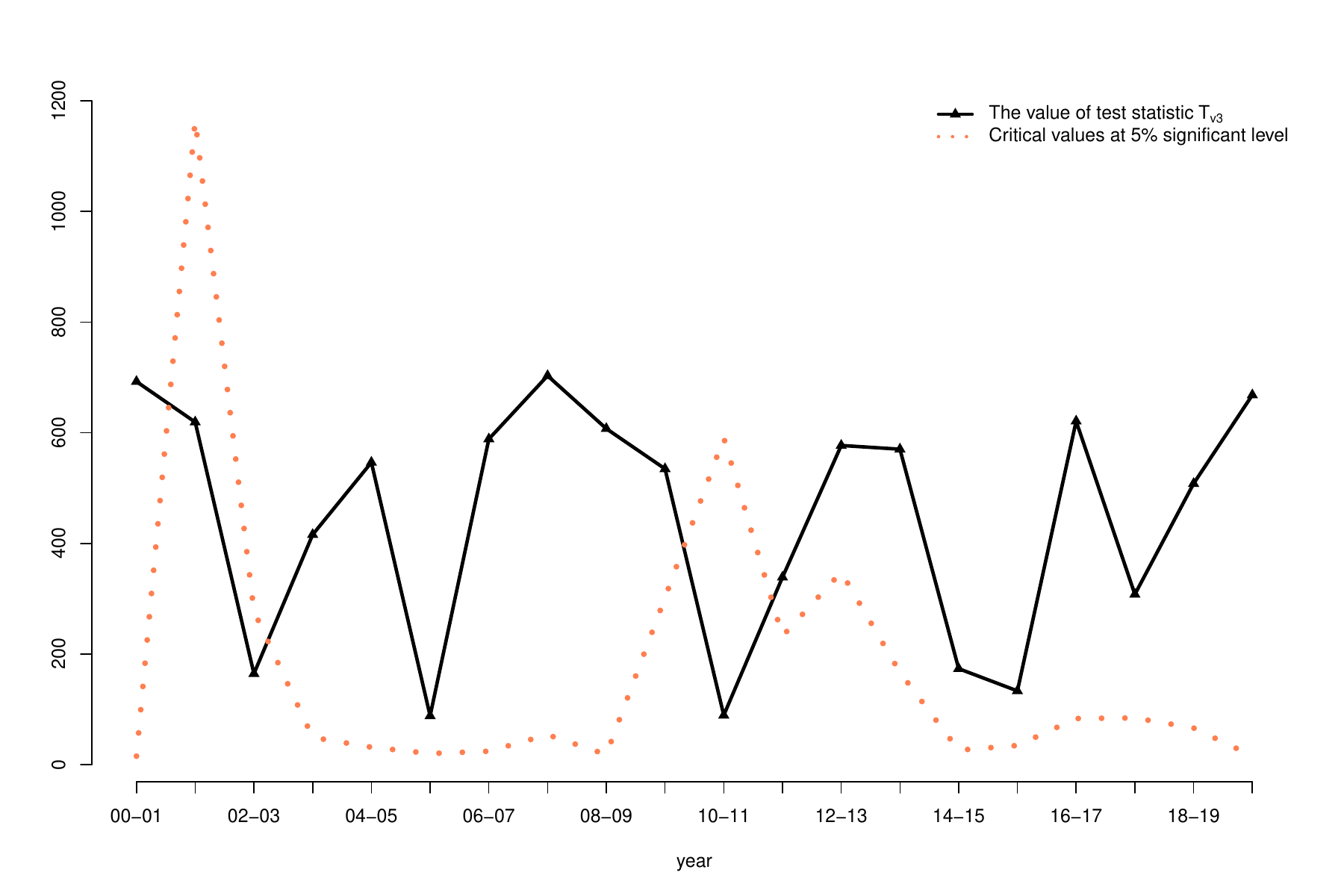}
\caption{Results of testing for equality of the first three principal eigenvectors between two consecutive years during 2000-2020. From top to bottom: testing equality of the first  principal eigenvector, the second  principal eigenvector, and the third  principal eigenvector, respectively.
}
\label{fig:test_vector_byyear}
\end{center}
\end{figure}

Notice that in this case, the asymptotic distribution under the null hypothesis is a complicated generalized $\chi^2$ distribution. There is no explicit formula for computing the critical value. To solve this issue, we simulate a large number of observations from the limiting distribution, based on which we estimate the 95\% quantile. That leads to the red dotted curve in the plots. Note that the critical values change over time. The reason is that the limiting distribution involves both population principal eigenvalues and eigenvectors, which are subject to change over time.
The black curves report test statistic values.

For the test about the first principal eigenvector, we see that for \emph{all} pairs of consecutive years, the value of the test statistic is  well above the 95\% quantile, so we should reject the null hypothesis that the first principal eigenvector is the same between  two consecutive years.
For the tests about the second and third principal eigenvector, we also reject the corresponding null hypothesis for most of  the pairs of consecutive years. These findings have a significant implication on factor modeling. In particular, the results show that structural breaks due to principal eigenvectors  occur more often than what one would have guessed based on stock market condition changes.

\subsection{Summary of the three test results}
Figure \ref{fig:test_123ratio_byyear} summarizes the results of the three tests.

\begin{figure}[htbp] 	
\centering	
\includegraphics[width=15cm]{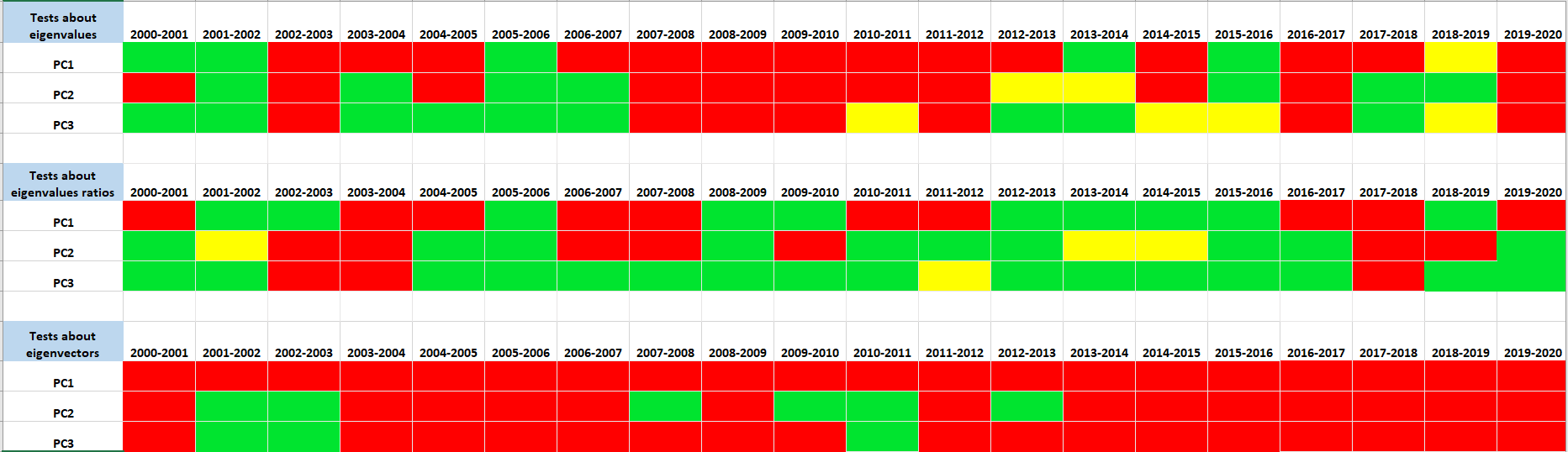}
\caption{Summary of the results of the three tests. Colors represent test results: red indicates rejection at 5\% level, yellow for rejection at 10\% level, and green for retaining at  10\% level.}
\label{fig:test_123ratio_byyear}
\end{figure}	

Figure \ref{fig:test_123ratio_byyear} reveals that testing for equality of  principal eigenvectors between two adjacent years result in more rejections than those about principal eigenvalues or eigenvalue ratios. Moreover, the tests about the first principal eigenvalue and eigenvector are more likely to be rejected than those about the second and third principal components. Let us point out that while it could be indeed the case that the first principal eigenvalue and eigenvector change more frequently than the second or third principal eigenvalue and eigenvector, the difference could also be due to that the first principal component is the strongest so that the related tests are most powerful.

\section{Conclusion}\label{sec:conclusion}
We establish both one-sample and two-sample central limit theorems for principal eigenvalues and eigenvectors under large factor models.
Based on these CLTs, we develop three tests to detect structural changes in large factor models.
Our tests can reveal whether the change is in principal eigenvalues or eigenvectors or both.
Numerically, these tests are found to have good finite sample performance. Applying these tests to daily returns of the S\&P500 Index constituent stocks, we find that, between two consecutive years,  the principal  eigenvalues, eigenvalue ratios and principal eigenvectors all exhibit frequent changes.


\newpage
\begin{center}
{\large\bf SUPPLEMENTARY MATERIAL}
\end{center} 
The supplementary material includes the proof of Theorem \ref{thm:clt_spike}, \ref{thm:clt_ratio}, \ref{thm:clt_ev} and \ref{thm:test_tv}, and Corollary \ref{cor:test_tvk} in the main text.

\section*{S1. Notations}
Recall the spectral decomposition of $\fSi=\fV \fLa \fV^\mT$, where the orthogonal matrix $\fV=(\fv_1,\ldots,\fv_N)$ consists of the eigenvectors of $\fSi$, and  $\fLa=\diag(\lambda_1,\ldots,\lambda_N)$ with eigenvalues $\lam_1\geq\ldots\geq \lam_N$.
Write $\fLa=\diag(\fLa_A, \fLa_B)$, where
\[
\fLa_A=\diag(\lambda_1,\ldots,\lambda_r)
~~~~~
\textrm{and} ~~~~
\fLa_B=\diag(\lambda_{r+1},\ldots,\lambda_N).
\]
Define $\fx_t=\fV^\mT \fy_t$. Then $\cov(\fx_t)=\fLa$, and the eigenvectors of $\fLa$ are the unit vectors $\fe_1,\ldots,\fe_N$, where $\fe_k$ is the unit vector with 1 in the $k$th entry and zeros elsewhere.
Let
$\fS_N=1/T \cdot\sum_{t=1}^\mT \fx_t \fx_t^\mT $, whose eigenvalues are denoted by $\wh{\lam}_1 \geq \wh{\lam}_2 \geq \ldots \geq \wh{\lam}_N$ with corresponding eigenvectors $\fu_1,\fu_2,\ldots,\fu_N$.
To resolve the ambiguity in the direction of an eigenvector, we specify the direction such that $\fu_k[k]\geq 0$ for all $1\leq k\leq N$, namey, the $k$th coordinate of the $k$th eigenvector is nonnegative (although when the $k$th coordinate is zero, the direction is still not specified, in which case we take an arbitrary direction.)
Note that the eigenvalues of $\wh{\fSi}_N=\dfrac 1T \sum_{t=1}^T \fy_t\fy_t^\mT$ are the same as $\fS_N$, and the eigenvectors are $\wh{\fv}_k=\fV \fu_k$. It follows that
\[
\langle \fv_k, \wh{\fv}_k \rangle
=\langle \fv_k, \fV \fu_k \rangle
=\fv_k^\mT \fV \fu_k = \langle \fe_k, \fu_k \rangle.
\]
In the below, we focus on the analysis of principal  eigenvalues and eigenvectors of~$\fS_N$.

\noindent\textit{Notation}: For any square matrix $\fA$, $\tr(\fA)$ denotes its trace,
 $|\fA|$ its determinant, and $\|\fA\|$ its spectral norm. For any vector $\fv$,
$\|\fv\|$ stands for its $\ell_2$ norm.
Write the $(i,j)$th entry of any matrix $\fW$ as $[\fW]_{ij}$ and $\fv[k]$ as the $k$th entry of a vector~$\fv$.
Use $\ga_j(\fA)$ to denote the $j$th largest eigenvalue of matrix $\fA$.
The notation $\pc$ stands for convergence in probability, $\wk$ represents convergence in law,
$Y_{n}=o_p (f (n))$ means that $Y_{n}/f (n) \pc 0$, and $Y_{n}=O_p (f (n))$
for that the sequence $(Y_n/f(n))$ is tight.
Write $a_n\asymp b_n$ if $c_1b_n\le a_n\le c_2b_n$ for some constants $c_1,c_2>0$.
For any sequence of random matrices  $(\fW_N)$ with fixed dimension,  write $\fW_N=o_p(1)$ or $O_p(1)$ if all the entries of $\fW_N$ are $o_p(1)$ or $O_p(1)$, respectively.
We say an event $\cA_n$ holds with high probability if $P(\cA_n) \geq 1-O(n^{-\ell})$ for any constant $\ell >0$.
Let $\fe_k,\wt{\fe}_{kA},\wt{\fe}_{kB}$ be the unit vectors with 1 in the $k$th coordinate and zeros elsewhere of dimensions $N,r,(N-r)$, respectively.
We use $\fI$ to denote the identity matrix
and $\indic(\cdot)$ to denote the indicator function.
Denote $\bC^+=\{z\in \bC: \im(z)>0 \}$, where $\im(z)$ is the imaginary part of a complex number $z$.
For  any probability distribution $G(x)$, its {\it Stieltjes transform} $m_G(z)$ is defined by
\[
m_G(z) = \int \dfrac{1}{x-z} \ dG(x), ~~~~~ z\in \bC^+.
\]
In all the sequel, $C$ is a generic constant whose value may vary from place  to place.

\section*{S2. Proof of Theorem \ref{thm:clt_spike} }
\label{ssec:pf_thm:clt_spike}

\begin{proof}[Proof of Theorem \ref{thm:clt_spike}]
Recall that $\fx_t=\fV^\mT \fy_t$. Write
\[
\fX_{N\times T}=(\fx_1,\ldots,\fx_T) :=
\begin{pmatrix}
{\fx_{(1)}}^\mT \\
\vdots \\
{\fx_{(N)}}^\mT
\end{pmatrix}.
\]
Let $\fz_t=\fLa^{-1/2}\fx_t$, $\fz_{(\ell)}=\lambda_{\ell}^{-1/2}\fx_{(\ell)}$ for $t=1,\ldots,T$, $\ell=1,\ldots,N$, and
$\fZ=(\fz_1,\ldots,\fz_T)=(\fz_{(1)},\ldots,\fz_{(N)})^\mT$. Then
\begin{eqnarray*}
\fX=\fLa^{1/2}\fZ =
\begin{pmatrix}
\sqrt{\lambda_1}{\fz_{(1)}}^\mT \\
\vdots \\
\sqrt{\lambda_N}{\fz_{(N)}}^\mT
\end{pmatrix}
=(\fLa^{1/2}\fz_1, \ldots, \fLa^{1/2}\fz_T).
\end{eqnarray*}
Write
\[
\fZ_A=(\fZ_A)_{r\times T}
=
\begin{pmatrix}
{\fz_{(1)}}^\mT \\
\vdots \\
{\fz_{(r)}}^\mT
\end{pmatrix}
=(\fz_1^{(A)}, \ldots, \fz_T^{(A)}),
\]
and
\[
\fZ_B=(\fZ_B)_{(N-r)\times T}
=
\begin{pmatrix}
{\fz_{(r+1)}}^\mT \\
\vdots \\
{\fz_{(N)}}^\mT
\end{pmatrix}
=(\fz_1^{(B)}, \ldots, \fz_T^{(B)}).
\]
Write
\[
\fz_t=
\begin{pmatrix}
\fz_t^{(A)}  \\
\fz_t^{(B)}
\end{pmatrix}, ~~~~~
\fz_t^{(A)}: r\times 1,~~
\fz_t^{(B)}: (N-r)\times 1.
\]
Define the companion matrix of $\fS_N$ as
\begin{eqnarray*}
\ul{\fS}_N &:=& \dfrac 1T \fX^\mT \fX
= \dfrac 1T \fZ^\mT \fLa \fZ
=\dfrac 1T \sum_{j=1}^\mT \lambda_j \fz_{(j)} \fz_{(j)}^\mT \\
&=& \dfrac 1T \sum_{j=1}^r \lambda_j \fz_{(j)} \fz_{(j)}^\mT
+ \dfrac 1T \sum_{j=r+1}^N \lambda_j \fz_{(j)} \fz_{(j)}^\mT \\
&=:& \ul{\fS}_{11} + \ul{\fS}_{22},
\end{eqnarray*}
where
\begin{eqnarray*}
\ul{\fS}_{11} &=& \dfrac 1T \sum_{j=1}^r \lambda_j \fz_{(j)} \fz_{(j)}^\mT
=\dfrac 1T \fZ_A^\mT \fLa_A \fZ_A , \\
\ul{\fS}_{22} &=& \dfrac 1T \sum_{j=r+1}^N \lambda_j \fz_{(j)} \fz_{(j)}^\mT
=\dfrac 1T \fZ_B^\mT \fLa_B \fZ_B.
\end{eqnarray*}
Further denote the companion matrices of $\ul{\fS}_{11}$ and $\ul{\fS}_{22}$ as
\[
\fS_{11}= \dfrac 1T \fLa_A^{1/2} \fZ_A \fZ_A^\mT \fLa_A^{1/2}, ~~~~~~~
\fS_{22}= \dfrac 1T \fLa_B^{1/2} \fZ_B \fZ_B^\mT \fLa_B^{1/2}.
\]

Define the event $\cF_s=\{\|\fS_{22}\| \le C_s \}$ for some constant $C_s\in (0,+\infty)$.
By Wely's Theorem,
Assumption (A.ii) and Theorem 9.13 of \cite{BS_book}, for any $\ell>0$, we can choose a $C_s$  sufficiently large such that
\begin{equation}\label{eq:S_22_bd}
  P(\cF_s^c) = o(T^{-\ell}).
\end{equation}

Note that the non-zero eigenvalues of $\ul{\fS}_{11}$ and $\ul{\fS}_{22}$ are the same as their companion matrices $\fS_{11}$ and $\fS_{22}$, respectively. For any principal eigenvalues $\lambda_k$, $k=1,\ldots,r$, the matrix
\[
\dfrac{\fS_{11}}{\lambda_k}=\dfrac 1T \(\dfrac{\fLa_A}{\lambda_k} \)^{1/2} \fZ_A \fZ_A^\mT \(\dfrac{\fLa_A}{\lambda_k} \)^{1/2}
\]
is in the low-dimensional situation as considered in Theorem 1 of \cite{andersonPCA}, by which one has  $\dfrac{\ga_j(\fS_{11})}{\lambda_k}-\dfrac{\lambda_j}{\lambda_k} \to 0$  for $1\le j,k\le r$.
Because $\|\fS_{22}\|=O_p(1)$, by Wely's Theorem that
\begin{eqnarray}\label{eqn:Weyl_ineq}
\ga_j(\fS_{11}) + \ga_T(\fS_{22}) \le \wh{\lam}_j \le \ga_j(\fS_{11}) + \ga_1(\fS_{22}), ~~~~ 1\le j\le r,
\end{eqnarray}
we get
\begin{eqnarray}\label{eqn:limit_evalues}
\dfrac{\wh{\lam}_j  }{\lambda_k} - \dfrac{\lambda_j}{\lambda_k}
\ \pc \ 0  ~~~~~~ \textrm{for}~ 1\le j,k\le r.
\end{eqnarray}
In particular, ${\wh{\lam}_k  }/{\lambda_k} \pc 1$  for $k=1,\ldots,r$.
	
Next, we derive the central limit theorem of $\wh{\lam}_k/\lambda_k$ for $1\le k\le r$.
	
Write $\fx_t=\begin{pmatrix} \fx_t^{(A)} \\ \fx_t^{(B)} \end{pmatrix}$, where $\fx_t^{(A)}=\fLa_A^{1/2} \fz_t^{(A)}$,  $\fx_t^{(B)}=\fLa_B^{1/2} \fz_t^{(B)}$.
Further denote
\[
\fX_A=(\fx_1^{(A)}, \ldots, \fx_T^{(A)})=\fLa_A^{1/2} \fZ_A, ~~~~ \textrm{and} ~~~~
\fX_B=(\fx_1^{(B)}, \ldots, \fx_T^{(B)})=\fLa_B^{1/2} \fZ_B.
\]
The sample covariance matrix $\fS_N$ can be decomposed as
\begin{eqnarray*}
\fS_N &=& \dfrac 1T \sum_{t=1}^\mT \fx_t \fx_t^\mT
=\dfrac 1T
\begin{pmatrix}
\sum_{t=1}^\mT \fx_t^{(A)} {\fx_t^{(A)}}^\mT &
\sum_{t=1}^\mT \fx_t^{(A)} {\fx_t^{(B)}}^\mT \\
\sum_{t=1}^\mT \fx_t^{(B)} {\fx_t^{(A)}}^\mT &
\sum_{t=1}^\mT \fx_t^{(B)} {\fx_t^{(B)}}^\mT
\end{pmatrix}	\\
&=& \dfrac 1T
\begin{pmatrix}
\fX_A \fX_A^\mT & \fX_A \fX_B^\mT \\
\fX_B \fX_A^\mT & \fX_B \fX_B^\mT
\end{pmatrix}	
= \dfrac 1T
\begin{pmatrix}
\fLa_A^{1/2} \fZ_A \fZ_A^\mT \fLa_A^{1/2} &
\fLa_A^{1/2} \fZ_A \fZ_B^\mT \fLa_B^{1/2} \\
\fLa_B^{1/2} \fZ_B \fZ_A^\mT \fLa_A^{1/2} &
\fLa_B^{1/2} \fZ_B \fZ_B^\mT \fLa_B^{1/2}
\end{pmatrix}	\\
&=& \begin{pmatrix}
\fS_{11} & \fS_{12} \\
\fS_{21} & \fS_{22}
\end{pmatrix}.
\end{eqnarray*}
Under Assumptions (A.ii), B and C, by \cite{SB95}, the ESD of $\fS_{22}$ almost surely converges to a non-random probability distribution $F$ whose Stieltjes transform $m(z)$ is the unique solution in the domain $\bC^+$ to the equation
\[
m(z)=\int \dfrac{1}{t(1-\rho-\rho zm(z))-z} \ dH(z), ~~~~
\mbox{for all } z \in \bC^+.
\]
By definition, each principal sample eigenvalues $\wh{\lam}_k$ solves the equation
\begin{equation}\label{eq:hat_lam}
0=|\lam \fI-\fS_N|
=|\lambda\fI -\fS_{22}| \cdot |\lambda\fI -\wt{\fK}_N(\lambda)|,
\end{equation}
where
\begin{equation}\label{dfn:KN_tilde}
\wt{\fK}_N(\lambda) = \fS_{11} +\fS_{12} (\lambda\fI -\fS_{22})^{-1}\fS_{21}
=\dfrac 1T \fX_A(\fI +\fA_N)\fX_A^\mT
=\fLa_A^{1/2} \fK_N(\lambda) \fLa_A^{1/2},
\end{equation}
and
\[
\fA_N = \fA_N(\lambda)=\dfrac 1T \fX_B^\mT
(\lambda\fI-\fS_{22})^{-1} \fX_B,
	~~~~~~~
\fK_N(\lambda) = \dfrac 1T \fZ_A (\fI+\fA_N) \fZ_A^\mT.
\]
Further define
\[
\fR_N=\fR_N(\lambda)=\dfrac{1}{\sqrt T}
\lb \fZ_A(\fI+\fA_N)\fZ_A^\mT -\tr(\fI+\fA_N)\cdot \fI \rb,
\]
then
\begin{eqnarray} \label{eqn:K_R}
\fK_N(\lambda)=\dfrac{1}{\sqrt T} \fR_N(\lambda) + \dfrac 1T
\tr(\fI+\fA_N(\lam))\cdot \fI.
\end{eqnarray}

We first give three  lemmas, which will be repeatedly used in the following proofs.
The first lemma is about  the random matrix $\fA_N$, and the second and third ones are about the limiting distributions of $(\fR_N(\lam) )$.
The proofs of these lemmas are postponed to the end of this subsection.

\begin{lem}\label{lem:limit_A}
Under Assumptions A--C, for $i,j=1,\ldots, r$, we have
\begin{eqnarray}\label{limit_trA}
&& \dfrac 1T \tr \fA_N(\lam_i) =O_p(N^{-1}),
~~~~~
\dfrac 1T \tr \(\fA_N(\lambda_i) \fA_N(\lam_j) \) =O_p(N^{-2}),
~~~~ \nonumber
\\
 \textrm{and} && ~~~~
\dfrac 1T \sum_{\ell=1}^\mT \( [\fA_N(\lam_i)]_{\ell\ell} [\fA_N(\lam_j)]_{\ell\ell} \) = O_p(N^{-2}) .
\end{eqnarray}
\end{lem}
	
\begin{lem}\label{lem:limit_R}
Under Assumptions  A--C and assume that $\lam\asymp N$,  the random matrix $\fR_N(\lam)$ converges weakly to a symmetric Gaussian random matrix $\fR=([\fR]_{ij})$ with zero-mean and the following covariance function:
\begin{eqnarray*}
\cov\([\fR]_{ij}, [\fR]_{i' j' }  \) =
\left\{
\begin{array}{ll}
 0, & \textrm{if} ~ i\neq i', ~ \textrm{or} ~ j\neq j'  ,
\\
\var\([\fR]_{ij} \) =1, & \textrm{if} ~ i=i' \neq j=j',
\\
\var\([\fR]_{ii} \) = \E(\fz_1[i] )^4 -1, & \textrm{if} ~ i=i' = j=j'.
\end{array}
\right.
\end{eqnarray*}
\end{lem}

\begin{lem}\label{lem:limit_R_joint}
Under Assumptions  A--C and $\lam\asymp N$, the  block diagonal random matrix $\fR_{J_N} =\diag(\fR_N(\lam_1), \ldots, \fR_N(\lam_r) ) $ converges weakly to a symmetric Gaussian block diagonal random matrix
$\fR_{J} =\diag(\fR_1, \ldots, \fR_r ) $ with zero-mean and the following covariance function, for any $1\le m, m' \le r$,
\begin{eqnarray*}
	\cov\([\fR_m]_{ij}, [\fR_{m'}]_{i' j' }  \) =
	\left\{
	\begin{array}{ll}
		 0, & \textrm{if} ~ i\neq i', ~ \textrm{or} ~ j\neq j'  ,
		\\
		1, & \textrm{if} ~ i=i' \neq j=j',
		\\
	 \E(\fz_1[i] )^4 -1, & \textrm{if} ~ i=i' = j=j'.
	\end{array}
	\right.
\end{eqnarray*}
\end{lem}

We now return to the analysis of principal sample eigenvalues $\wh{\lam}_k$.
Noting that the principal eigenvalues of $\fS_N$ go to infinity and the estimate \eqref{eq:S_22_bd}, without loss of generality, we can assume that for $N$ large enough, $\wh{\lam}_k$ is not an eigenvalue of~$\fS_{22}$. It follows that  $\wh{\lam}_k$ is the $k$th eigenvalue of matrix $\wt{\fK}_N(\wh{\lam}_k)$.

Note that
\begin{eqnarray*}
\fK_N(\lam_k) -  \fK_N(\wh{\lam}_k)
&=&
\dfrac{1}{T^2} \fZ_A \fX_B^\mT \lb (\wh{\lambda}_k\fI-\fS_{22})^{-1} -(\lambda_k\fI-\fS_{22})^{-1} \rb \fX_B \fZ_A^\mT \\
&=& \(1-\dfrac{\wh{\lam}_k}{\lambda_k}\)\fQ_N,
\end{eqnarray*}
where
\[
\fQ_N = \dfrac{1}{T^2\wh{\lam}_k} \fZ_A \fX_B^\mT \(\fI-\wh{\lam}_k^{-1}\fS_{22}\)^{-1} \(\fI-\lambda_k^{-1}\fS_{22}\)^{-1}\fX_B\fZ_A^\mT.
\]
By the elementary formulae,
\[
\fX^\mT (\fI - \fX \fX^\mT)^{-1} = (\fI-\fX^\mT \fX)^{-1} \fX^\mT,
~~~~~
 (\fI - \fX \fX^\mT)^{-1} \fX =\fX (\fI-\fX^\mT \fX)^{-1} ,
\]
it follows that
\begin{eqnarray*}
\fQ_N
&=& \dfrac{1}{T^2\wh{\lam}_k}\fZ_A \(\fI-\dfrac{1}{T\wh{\lam}_k}\fX_B^\mT\fX_B\)^{-1} \fX_B^\mT \fX_B \(\fI-\dfrac{1}{T\lambda_k}\fX_B^\mT\fX_B\)^{-1} \fZ_A^\mT
\\
&=& O_p(1/N),
\end{eqnarray*}
where the last step comes from the fact that eigenvalues of $\fS_{22}$ are  $O_p(1)$  and an   analysis similar to that of  $\fR_N$ in the proof of Lemma \ref{lem:limit_R} below.
It follows from~\eqref{eqn:limit_evalues} that
\begin{eqnarray}\label{KK_hat}
\fK_N(\lam_k) -  \fK_N(\wh{\lam}_k)  = o_p(1/N).
\end{eqnarray}

Recall that $\wh{\lam}_k$ is the $k$th largest eigenvalue of matrix $\wt{\fK}_N(\wh{\lam}_k)$.
Denote matrix $\fD=\( \wh{\lam}_k\fI-\wt{\fK}_N(\wh{\lam}_k) \)/\lam_k$. From Assumption A, Lemmas \ref{lem:limit_A}, \ref{lem:limit_R}, and equations \eqref{eqn:limit_evalues}, \eqref{eqn:K_R} and  \eqref{KK_hat}, it follows that
\[
[\fD]_{kk}=\(\frac{\wh{\lam}_k}{\lam_k}-1\)-\dfrac{1}{\sqrt T}[\fR_N(\lam_k)]_{kk}+O_p(1/T),
\]
for $i\neq k$,
\begin{equation}\label{eqn:diag_D_ii}
[\fD]_{ii}=\(\frac{\wh{\lam}_k}{\lam_k}-\dfrac{\lam_i}{\lam_k} \)-\dfrac{\lam_i}{\lam_k}\cdot \dfrac{1}{\sqrt T}[\fR_N(\lam_k)]_{ii}+O_p(1/T)
\pc 1-\th_i/\th_k \neq 0.
\end{equation}
and
\[
[\fD]_{ij}=-\dfrac{\sqrt{\lam_i\lam_j}}{\lam_k}\cdot \dfrac{1}{\sqrt T}[\fR_N(\lam_k)]_{ij} +o_p(1/N)=O_p(1/\sqrt{T}).
\]
Let det$(\fA)$ be the determinant of a matrix $\fA$, then
\begin{eqnarray*}
0&=&\textrm{det}\(\wh{\lam}_k\fI - \wt{\fK}_N(\wh{\lam}_k)\)/\lam_k \\
&=& \textrm{det}\(
\begin{array}{ccccc}
[\fD]_{11} &  & & & \\
       & \ddots & & & O_p(1\sqrt{T}) \\
 & &  [\fD]_{kk} & & \\
O_p(1/\sqrt{T}) & & & \ddots & \\
& & & & [\fD]_{rr}
\end{array}
\)\\
&=& \(\prod_{i=1,i\neq k}^{r} [\fD]_{ii}\) \cdot \left[\(\frac{\wh{\lam}_k}{\lam_k}-1\)-\frac{1}{\sqrt T} [\fR_N]_{kk} \right] \\
&& ~~~~~~ +O_p(1/T)  \left[\(\frac{\wh{\lam}_k}{\lam_k}-1\)-\frac{1}{\sqrt T} [\fR_N]_{kk} \right]
 +O_p(1/T).
\end{eqnarray*}
Using \eqref{eqn:diag_D_ii} we then obtain that
\[
\dfrac{\wh{\lam}_k}{\lam_k} = 1+ \dfrac{1}{\sqrt T} [\fR_N(\lam_k)]_{kk} + O_p(1/T).
\]
In particular,
\begin{eqnarray} \label{eqn:exp_CLT_evalue}
\sqrt T \(\dfrac{\wh{\lam}_k}{\lam_k} - 1 \) - [\fR_N(\lam_k)]_{kk} \ \pc \ 0.
\end{eqnarray}
By Lemma \ref{lem:limit_R}, we obtain
\[
\sqrt T \(\dfrac{\wh{\lam}_k}{\lam_k} - 1  \) \ \wk \ N(0, \E(\fz_1[k])^4 -1 ).
\]
Similarly, the joint convergence of
$\sqrt T \({\wh{\lam}_1}/{\lam_1} - 1, \cdots,
{\wh{\lam}_r}/{\lam_r} - 1 \)$ follows from \eqref{eqn:exp_CLT_evalue} and Lemma  \ref{lem:limit_R_joint}.
\end{proof}

We now prove Lemmas \ref{lem:limit_A}--\ref{lem:limit_R_joint}.
\begin{proof}[Proof of Lemma \ref{lem:limit_A}]
By the estimate \eqref{eq:S_22_bd}, it suffices to prove Lemma \ref{lem:limit_A} for $\fA_N(\lam_i) \indic(\cF_s)$ and  $\fA_N(\lam_j) \indic(\cF_s)$.
Under Assumption A, we have
\[
\dfrac 1T \tr\fA_N(\lambda_i) \indic(\cF_s)
=\dfrac 1T \tr\lb\(\lambda_i\fI-\fS_{22} \)^{-1} \fS_{22} \rb \indic(\cF_s)
=O_p(N^{-1}),
\]
and
\[
\aligned
\dfrac 1T \tr \(\fA_N(\lambda_i)  \fA_N(\lam_j) \) \indic(\cF_s)
&= \dfrac 1T \tr\((\lam_i\fI-\fS_{22} )^{-1} \fS_{22} (\lam_j\fI -\fS_{22})^{-1} \fS_{22} \)
\indic(\cF_s)\\
&=O_p\( N^{-2}\).
\endaligned
\]
To prove $(1/T) \sum_{\ell=1}^T [\fA_N(\lam_i)]_{\ell\ell} [\fA_N(\lam_j)]_{\ell\ell} \indic(\cF_s)=O_p(N^{-2})$, it suffices to show that
\begin{eqnarray}\label{eqn:temp}
\max_i\max_\ell \E \lb [\fA_N(\lam_i)]_{\ell\ell}^2 \indic(\cF_s)\rb=  O(N^{-2}).
\end{eqnarray}
Note that
$
[\fA_N]_{\ell\ell} = T^{-1} {\fz_{\ell}^{(B)}}^T \fLa_B^{1/2}
(\lambda\fI -\fS_{22})^{-1} \fLa_B^{1/2} \fz_{\ell}^{(B)}
$
are identically distributed for different $\ell$, hence
\begin{eqnarray*}
\E[\fA_N(\lam_i)]_{\ell\ell}^2 \indic(\cF_s)
&=& \dfrac 1T \sum_{\ell=1}^T \E\lb [\fA_N(\lam_i)]_{\ell\ell}^2 \indic(\cF_s)\rb
\le \dfrac 1T \E \lb \tr\(\fA_N(\lam_i)\)^2 \indic(\cF_s)\rb \\
&=& \dfrac 1T \E \lb \tr\((\lam_i\fI -\fS_{22})^{-1} \fS_{22} \)^2 \indic(\cF_s)\rb
\le \dfrac{C}{N^2}.
\end{eqnarray*}

\end{proof}

\begin{proof}[Proof of Lemma \ref{lem:limit_R}]
Recall that $\fZ_A = (\fz_{(1)}, \ldots, \fz_{(r)} )^\mT$. We have
\[
[\fZ_A (\fI + \fA_N(\lam) ) \fZ_A^\mT ]_{ij} = \fz_{(i)}^\mT ( \fI + \fA_N(\lam)) \fz_{(j)} .
\]
Consider the random vector of dimension $K=\dfrac 12 r(r+1)$:
\[
\fW_N=\fW_N(\lam) := \dfrac{1}{\sqrt T} \(\fz_{(i)}^\mT (\fI_N + \fA_N(\lam)) \fz_{(j)} - \tr(\fI + \fA_N(\lam))
\cdot\E(\fz_{(i)}[1] \fz_{(j)}[1]) \)_{1\le i\le j \le r }.
\]
For any  $1\le \ell , \ell' \le K$, there exist two pairs $(i,j)$ and $(i', j')$, $1\le i\le j\le r, \ 1\le i' \le j' \le r$ such that
\[
\fW_N[\ell] = \dfrac{1}{\sqrt T} \(  \fz_{(i)}^\mT (\fI + \fA_N ) \fz_{(j)} - \tr(\fI + \fA_N) \cdot\E(\fz_{(i)}[1] \fz_{(j)}[1]) \) ,
\]
and
\[
\fW_N[\ell'] =\dfrac{1}{\sqrt T} \(   \fz_{(i')}^\mT (\fI + \fA_N ) \fz_{(j')}  - \tr(\fI + \fA_N)
\cdot\E(\fz_{(i')}[1] \fz_{(j')}[1]) \).
\]
By Lemma \ref{lem:limit_A}, we have
\[
\lim_{T\to\infty} \dfrac 1T \sum_{t=1}^\mT \([\fI + \fA_N]_{ii} \)^2 = 1, ~~~~
\textrm{and} ~~~~
\lim_{T\to\infty}  \dfrac 1T \tr(\fI + \fA_N)^2 = 1.
\]
By Corollary 7.1 of \cite{BY08}, the random vector $\fW_N$ converges weakly to a $K$-dimensional Gaussian vector with mean zero and covariance matrix $\fGa_W$ satisfying
$[\fGa_W]_{\ell\ell'} =\rho_{(i,j)(i',j')} $, where
\begin{eqnarray*}
\rho_{(i,j)(i',j')} &=& \E\(\fz_{(i)}[1] \ \fz_{(j)}[1] \ \fz_{(i')}[1] \ \fz_{(j')}[1] \)
- \E\(\fz_{(i)}[1] \ \fz_{(j)}[1] \) \E\(\fz_{(i')}[1] \ \fz_{(j')}[1] \)    \\
&=&  \E\(\fz_1[i] \ \fz_1[j] \ \fz_1[i'] \ \fz_1[j'] \)
- \E\(\fz_1[i] \ \fz_1[j] \) \E\(\fz_1[i'] \ \fz_1[j'] \).
\end{eqnarray*}
The result follows.
\end{proof}

\begin{proof}[Proof of Lemma \ref{lem:limit_R_joint}]
Consider the block diagonal random matrix
\[
\fR_{J_N}=\diag\(\fR_{N}(\lam_1), \ldots, \fR_N(\lam_r)\)
\]
as an $M=r \times \dfrac 12 r(r+1)$
dimensional vector
\begin{eqnarray*}
&& \dfrac{1}{\sqrt T} \(
\(\fz_{(i)}^\mT (\fI +\fA_N(\lam_1)) \fz_{(j)} - \tr(\fI+ \fA_N(\lam_1)) \cdot\E(\fz_{(i)}[1] \fz_{(j)}[1])\)_{1\le i\le j\le r} ~ , \ldots,  \right. \\
&&  \left.  ~~~~~~~~~~~~~~~~~
\(\fz_{(i)}^\mT (\fI +\fA_N(\lam_r)) \fz_{(j)} - \tr(\fI + \fA_N(\lam_r)) \cdot\E(\fz_{(i)}[1] \fz_{(j)}[1])\)_{1\le i\le j\le r}
\).
\end{eqnarray*}
By Assumption A--C and Lemma \ref{lem:limit_A}, the block diagonal random matrix $\fR_{J_N}$ converges weakly to a symmetric Gaussian block diagonal random matrix
\[
\fR_J=\diag(\fR_1, \ldots, \fR_r)
\]
with mean zero and covariance function as follows: for any $1\le m, m' \le r$,
\[
\cov\([\fR_m]_{ij}, [\fR_{m'}]_{i' j' }  \) =
\E\(\fz_1[i] \ \fz_1[j] \ \fz_1[i'] \ \fz_1[j'] \) -
\E\(\fz_1[i] \ \fz_1[j] \) \cdot \E\(\fz_1[i'] \ \fz_1[j']  \).
\]
The conclusion follows.
\end{proof}

\section*{S3. Proof of Theorem \ref{thm:clt_ratio} }
\label{subsec:proof_thm:clt_ratio}

\begin{proof}[Proof]
Write
\begin{eqnarray*}
\tr(\fSi_{-r})= \sum_{j=r+1}^{N} \lambda_j, ~~~~~~
\tr(\fS_{-r})= \sum_{j=r+1}^{N} \wh{\lam}_j.
\end{eqnarray*}
To prove Theorem 2, we first show that $ \tr(\fS_{-r})/N$ has a faster convergence rate than $\sqrt T$, that is,
\begin{eqnarray}\label{eqn:delta}
\De:= \sqrt T \(\dfrac 1N \tr(\fS_{-r}) - \dfrac 1N \tr(\fSi_{-r} ) \) \ \pc \ 0.
\end{eqnarray}
Decompose $\De=\De_1 + \De_2$, where
\begin{eqnarray*}
\De_1 = \sqrt T \( \dfrac 1N \tr(\fS_{-r}) - \dfrac 1N \tr(\fS_{22}) \) , ~~~~
\De_2 = \sqrt T \( \dfrac 1N  \tr(\fS_{22}) - \dfrac 1N \tr(\fSi_{-r}) \) .
\end{eqnarray*}
Note that $\De_2 \pc 0$ by Theorem 2.1 of \cite{ZBY2015}.

Next we analyze $\De_1$. Note that
\[
\tr(\fS_N)=\tr(\fS_{11}) + \tr(\fS_{22})
= \sum_{j=1}^r \wh{\lam}_j + \tr(\fS_{-r}) .
\]	
Hence
\[
\De_1 = \sum_{j=1}^r \sqrt T \( \ga_j(\fS_{11}/N ) - \wh{\lam}_j/N \) .
\]
By inequality \eqref{eqn:Weyl_ineq}, we have
\[
\dfrac{r\sqrt T}{N} \ \ga_T(\fS_{22}) \ \le \
\sum_{j=1}^r \sqrt T \( \dfrac{\wh{\lam}_j}{N} - \ga_j(\fS_{11}/N) \)
\ \le \
\dfrac{r\sqrt T}{N} \ \ga_1(\fS_{22}).
\]
By Assumption (A.ii) and Assumption C, we have $\De_1 \pc 0$. Hence, \eqref{eqn:delta} holds.

We can then rewrite the result of Theorem 1 as follows:
\[
\sqrt T \left[
\begin{pmatrix}
\wh{\lam}_1/N  \\
\vdots \\
\wh{\lam}_r/N \\
\tr(\fS_{-r})/N
\end{pmatrix}
-
\begin{pmatrix}
\lam_1/N  \\
\vdots \\
\lam_r/N \\
\tr(\fSi_{-r})/N
\end{pmatrix}
\right]
\ \to \
N(0, \fSi_J),
\]
where $\fSi_J = \diag( \theta_1^2 \si_{\lam_1}^2, \ldots, \theta_r^2 \si_{\lam_r}^2, 0 )$.
For any $k=1, \ldots, r$, by considering the function
\[
f(\fx) = f(x_1, \ldots, x_{r+1}) = \dfrac{x_k}{\sum_{i\neq k, i=1}^r x_i },
\]
and using the Delta method, we get that
\[
\sqrt T
\(
\dfrac{\wh{\lam}_k}{\tr(\wh{\fSi}_N) - \wh{\lam}_k } - \dfrac{\lam_k}{\tr(\fSi)-\lam_k } \)
\ \wk \ N(0, \si_{-k}^2),
\]
where
\begin{eqnarray*}
\si_{-k}^2
&=& \dfrac{\theta_k^2 }{\( \sum_{i\neq k, i=1}^r \theta_i + \int t dH(t) \)^2 } \
\left[
\si_{\lam_k}^2 + \dfrac{\sum_{j\neq k, j=1}^r \theta_j^2 \ \si_{\lam_j}^2  }{\(  \sum_{i\neq k, i=1}^r \theta_i + \int t dH(t) \)^2 }
\right]  .
\end{eqnarray*}
Finally, $\wh{\si_{-k}^2}$ defined in Theorem 2 is a consistent estimator of $\si_{-k}^2$ by Theorem 1 and \eqref{eqn:delta}.
\end{proof}

\section*{S4. Some preliminary results for proving Theorems \ref{thm:clt_ev} and \ref{thm:test_tv}}\label{subsec:pre_results}
We first derive some preliminary results in preparation for the proofs of Theorems 3 and 6.

Recall that we write $\fu_k=(\fu_{kA}^\mT, \fu_{kB}^\mT)^\mT$ as the eigenvector of $\fS_N$ corresponding to the eigenvalue $\wh{\lam}_k$, where $\fu_{kA}$ and $\fu_{kB}$ are of dimensions $r$ and $N-r$, respectively.  Also recall that $\wt{\fu}_{kA}=\fu_{kA}/\|\fu_{kA}\|$. Further denote by $\wt{\fe}_{kA}$ the $r$ dimensional vector with 1 in the $k$th coordinate and 0's elsewhere.

\begin{prop}\label{thm:clt_ev_spike}
Under Assumptions A--C, for $1\le k \le r$, we have
\[
T \(1- \(\wt{\fu}_{kA}[k]\)^2 \) \ \wk \   \sum_{i\neq k, i=1}^{r} \om_{ki} \cdot Z_i^2,
\]
where $\om_{ki}= \dfrac{\theta_k \theta_i}{(\theta_k-\theta_i)^2}$ and $Z_i \stackrel{iid}{\sim} N(0,1)$.
\end{prop}
\begin{rmk}\label{cor:clt_ev_spike}
As a corollary, we have
\[
N \(1-|\wt{\fu}_{kA}[k]| \) = N \( 1-\dfrac{|\fu_k[k]|}{\|\fu_{kA} \|} \)   \  \wk \
\dfrac{\rho}{2} \sum_{i\neq k,i=1}^r \om_{ki} \cdot Z_i^2.
\]
\end{rmk}

\begin{prop}\label{thm:clt_spike_spike}
Under Assumptions A--C,
\begin{enumerate}[(i)]
\item for $1\le k\le r$, we have
\[
\sqrt{T} (\wt{\fu}_{kA} -\wt{\fe}_{kA} ) \ \wk \ N(0,\fSi_k),
\]
where
\[
\fSi_{k}= \sum_{i\neq k,i=1}^r  \om_{ki} \wt{\fe}_{iA} \wt{\fe}_{iA}^\mT,
	~~~~~~
{\textrm and}~~~~ \om_{ki}= \dfrac{\theta_k \theta_i}{(\theta_k-\theta_i)^2} ;
\]
\item for any fixed $r$-dimensional vectors $\fc_k$, $k=1,\ldots,r$,
if there exist $i\neq j$ such that $\fc_i[j]\neq \fc_j[i]$, then
\[
\sqrt{T} \sum_{k=1}^r \fc_k^\mT (\wt{\fu}_{kA} - \wt{\fe}_{kA}) \ \wk \
N\(0, \sum_{k=1}^{r-1} \sum_{i=k+1}^r \om_{ki} (\fc_k[i] - \fc_i[k] )^2 \) ;
\]
\item for $1\le \ell, k \le r$, the $\ell$th principal eigenvalue $\sqrt T (\wh{\lam}_{\ell}/\lam_\ell -1 )$ and the $k$th principal eigenvector $\sqrt T (\wt{\fu}_{kA} - \wt{\fe}_{kA})$ are asymptotically independent.
\end{enumerate} 	
\end{prop}

\begin{rmk} The conclusion in (i) coincides with  Theorem 3.2 in \cite{WF17}, which is proved under the sub-Gaussian assumption.
\end{rmk}

\begin{prop}\label{thm:clt_ev_norm}
Under Assumptions A--C, for $1\le \ell\le r$, we have
\[
\sqrt T \(\lambda_k\(1-\|\fu_{kA}\|^2\)-
\dfrac{1}{T} \sum_{j=r+1}^N \dfrac{\wh{\lam}_j}{(1-\wh{\lam}_j/\lambda_k)^2} \)
\ \wk \ N(0,\si_{kA}^2),
\]
where $\si_{kA}^2=  \(\E(\fz_1[k])^4 -3 \) \cdot \rho^2
\(\int xdF(x)\)^2 + 2\rho\int x^2 dF(x) $
and
the function $F(\cdot)$ is a distribution function whose Stieltjes transform, $m_F$, is the unique solution in the set $\{m_F\in \bC^+: -(1-\rho)/z + \rho m_F \in \bC^+ \}$ to the equation
\[
m_F=\int \dfrac{dH(\tau)}{\tau(1-\rho-\rho z m_F)-z},
\]
where $H$ is given in Assumption (A.ii).
\end{prop}
\begin{rmk}\label{cor:clt_ev_norm}
Under Assumptions A--C,
Proposition \ref{thm:clt_ev_norm} can be rewritten as
\[
N^{3/2} \( 1-\|\fu_{kA}\|^2 - \dfrac{1}{T\lambda_k} \sum_{j=r+1}^N \dfrac{\wh{\lam}_j}{(1-\wh{\lam}_j/\lambda_k)^2}  \)  \ \wk \ N(0,\rho \sigma_{kA}^2/\theta_k^2),
\]
or
\begin{eqnarray}\label{eqn:limit_u_kA}
N^{3/2} \( 1-\|\fu_{kA}\| - \dfrac{1}{2T\lambda_k} \sum_{j=r+1}^N \dfrac{\wh{\lam}_j}{(1-\wh{\lam}_j/\lambda_k)^2}  \)  \ \wk \ N(0,\rho \sigma_{kA}^2/(4\theta_k^2)).
\end{eqnarray}
In particular,
$N\(1-\|\fu_{kA}\| \) \pc  \dfrac{\rho}{2\theta_k} \int xdF(x)$.
\end{rmk}

\subsection*{A. Proof of Proposition \ref{thm:clt_ev_spike}}
\begin{proof}
By definition,
	\[
	\fS_N \fu_k =\wh{\lam}_k \fu_k.
	\]
Writing $\fu_k=(\fu_{kA}^\mT, \fu_{kB}^\mT)^\mT$ gives us that
	\begin{eqnarray}
	\fS_{11} \fu_{kA} + \fS_{12} \fu_{kB} &=& \wh{\lam}_k \fu_{kA} \label{eqn:def_u_kA} , \\
	\fS_{21} \fu_{kA} + \fS_{22} \fu_{kB} &=& \wh{\lam}_k \fu_{kB} \label{eqn:def_u_kB} .
	\end{eqnarray}
Solving \eqref{eqn:def_u_kB} for $\fu_{kB}$ yields
\begin{eqnarray} \label{eqn:u_kB_exp}
\fu_{kB}=(\wh{\lam}_k\fI-\fS_{22})^{-1}\fS_{21} \fu_{kA}.
\end{eqnarray}
Replacing $\fu_{kB}$ in \eqref{eqn:def_u_kA} with \eqref{eqn:u_kB_exp} gives
\[
\lb \fS_{11} +\fS_{12} (\wh{\lam}_k\fI -\fS_{22})^{-1} \fS_{21} \rb \fu_{kA} =\wh{\lam}_k \fu_{kA},
\]
in other words,
$\wt{\fK}_N(\wh{\lam}_k) \fu_{kA} = \wh{\lam}_k \fu_{kA}$.

To prove Proposition \ref{thm:clt_ev_spike}, we first show that $\wt{\fu}_{kA}-\wt{\fe}_{kA} \pc 0$, where
$\wt{\fu}_{kA}=\fu_{kA}/\|\fu_{kA}\|$.
It follows from the definition of $\fK_N$ that
\begin{eqnarray*}
\wt{\fu}_{kA} &=& \dfrac{1}{\wh{\lam}_k} \wt{\fK}_N(\wh{\lam}_k) \wt{\fu}_{kA}  \\
&=& \dfrac{1}{\wh{\lam}_k} \wt{\fK}_N(\lam_k) \wt{\fu}_{kA}
+\dfrac{1}{\wh{\lam}_k} \fLa_A^{1/2} \(\fK_N(\wh{\lam}_k)-\fK_N(\lambda_k) \) \fLa_A^{1/2} \ \wt{\fu}_{kA} \\
&=&  \dfrac{1}{\wh{\lam}_k} \wt{\fK}_N(\lambda_k) \wt{\fu}_{kA}
+ \dfrac{1}{\wh{\lam}_k^2} \(1-\dfrac{\wh{\lam}_k}{\lambda_k} \) \fLa_A^{1/2} \fM_N \fLa_A^{1/2} \wt{\fu}_{kA},
\end{eqnarray*}
where
\[
\fM_N=\dfrac{1}{T^2} \fZ_A \fX_B^\mT (\fI - \wh{\lam}_k^{-1} \fS_{22})^{-1} (\fI-\lam_k^{-1}\fS_{22})^{-1} \fX_B \fZ_A^\mT .
\]
Because $\|\fS_{22}\|$ and $\|(1/T)Z_AZ_A^\mT\|$ are both $O_p(1)$, and  $\lam_k\asymp N$ and $\wh{\lam}_k=O_p(N)$, we have $\|\fM_N\| = O_p(1)$.
Thus, by Theorem 1 and Assumption A, we get
\[
\wt{\fu}_{kA} = \dfrac{1}{\wh{\lam}_k} \wt{\fK}_N(\lam_k) \wt{\fu}_{kA} + O_p(N^{-3/2}).
\]
For the first term, consider the following decomposition:
\begin{eqnarray*}
	\dfrac{1}{\wh{\lam}_k} \wt{\fK}_N(\lambda_k) \wt{\fu}_{kA}
	&=& \dfrac{1}{\wh{\lam}_k} \fLa_A^{1/2} \fK_N(\lambda_k) \fLa_A^{1/2} \wt{\fu}_{kA} \\
	&=& \dfrac{1}{\wh{\lam}_k} \fLa_A^{1/2} \( \dfrac{1}{\sqrt T} \fR_N(\lambda_k) + \dfrac 1T \tr(\fI+\fA_N(\lambda_k)) \) \fLa_A^{1/2} \wt{\fu}_{kA}  \\
	&=& \dfrac{1}{\wh{\lam}_k} \dfrac{1}{\sqrt T} \fLa_A^{1/2} \fR_N(\lambda_k) \fLa_A^{1/2} \wt{\fu}_{kA}
	+\dfrac{\fLa_A}{\wh{\lam}_k} \wt{\fu}_{kA}
	+\dfrac 1T \tr \fA_N \cdot \dfrac{\fLa_A}{\wh{\lam}_k} \wt{\fu}_{kA} \\
&=& \dfrac{1}{\wh{\lam}_k} \dfrac{1}{\sqrt T} \fLa_A^{1/2} \fR_N(\lambda_k) \fLa_A^{1/2} \wt{\fu}_{kA}
+\dfrac{\fLa_A}{\wh{\lam}_k} \wt{\fu}_{kA}
+ O_p(1/N),
\end{eqnarray*}
where the last step comes from Assumption A and Lemma \ref{lem:limit_A}.
Hence,
\begin{eqnarray} \label{eqn:u_kA_0}
\wt{\fu}_{kA} = \dfrac{1}{\wh{\lam}_k \sqrt T} \fLa_A^{1/2} \fR_N(\lam_k) \fLa_A^{1/2} \wt{\fu}_{kA} + \dfrac{\fLa_A}{\wh{\lam}_k} \wt{\fu}_{kA} + O_p(1/N).
\end{eqnarray}
Subtracting $(\fLa_A/\lam_k)\wt{\fu}_{kA}$ on both sides of \eqref{eqn:u_kA_0} yields
\begin{eqnarray} \label{eqn:u_kA_1}
\(\fI - \dfrac{\fLa_A}{\lam_k} \) \wt{\fu}_{kA}
&=& \dfrac{1}{\wh{\lam}_k \sqrt T} \ \fLa_A^{1/2} \fR_N(\lam_k) \fLa_A^{1/2} \wt{\fu}_{kA} \nonumber \\
&& ~~~
+ \dfrac{\fLa_A}{\wh{\lam}_k} \(1-\dfrac{\wh{\lam}_k}{\lam_k} \)  \wt{\fu}_{kA}
+ O_p(1/N).
\end{eqnarray}
Further define
\[
\cO_{N,k} = \sum_{i\neq k,i=1}^r \dfrac{\lambda_k}{\lambda_k-\lambda_i} \ \wt{\fe}_{iA} \wt{\fe}_{iA}^\mT.
\]
It is easy  to see that
\[
\cO_{N,k}\(\fI-\dfrac{\fLa_A}{\lambda_k} \)
=\sum_{i\neq k,i=1}^r \wt{\fe}_{iA} \wt{\fe}_{iA}^\mT
=\fI  -\wt{\fe}_{kA} \wt{\fe}_{kA}^\mT .
\]
Left-multiplying $\cO_{N,k}$ on both sides of \eqref{eqn:u_kA_1} yields
\begin{eqnarray} \label{eqn:u_kA_2}
 \wt{\fu}_{kA} - \langle \wt{\fu}_{kA}, \wt{\fe}_{kA} \rangle \wt{\fe}_{kA}
&=& \dfrac{1}{\wh{\lam}_k \sqrt T} \cO_{N,k} \fLa_A^{1/2} \fR_N(\lam_k) \fLa_A^{1/2} \wt{\fu}_{kA }  \nonumber\\
&& ~~~ + \(1-\dfrac{\wh{\lam}_k}{\lam_k}\) \dfrac{1}{\wh{\lam}_k} \cO_{N,k} \fLa_A \wt{\fu}_{kA}
+ O_p(1/N).
\end{eqnarray}
By Lemma \ref{lem:limit_R} and Theorem 1, we get
\[
\wt{\fu}_{kA} - \langle \wt{\fu}_{kA}, \wt{\fe}_{kA} \rangle \wt{\fe}_{kA} = O_P(1/\sqrt{T}).
\]
It follows that $\wt{\fu}_{kA}-\wt{\fe}_{kA}\pc 0$.

Replacing
$\wh{\lam}_k$ and $\wt{\fu}_{kA}$ on the right hand side of equation \eqref{eqn:u_kA_0} with $\lam_k$ and $\wt{\fe}_{kA}$, respectively, yields
\begin{eqnarray*}
\wt{\fu}_{kA} - \wt{\fe}_{kA}
&=& \dfrac{1}{\lam_k \sqrt T} \fLa_A^{1/2} \fR_N(\lam_k) \fLa_A^{1/2} \wt{\fe}_{kA}
+ \(\dfrac{\lam_k}{\wh{\lam}_k} -1 \) \dfrac{\fLa_A}{\lam_k} \wt{\fe}_{kA}  \\
&& ~~~ + \dfrac{\fLa_A}{\lam_k} \(\wt{\fu}_{kA} - \wt{\fe}_{kA} \)
+\(\dfrac{\fLa_A}{\lam_k} - \fI \) \wt{\fe}_{kA} + o_p(1/\sqrt T).
\end{eqnarray*}
Rewrite the above equation as
\begin{eqnarray*}
\sqrt T \(\fI - \dfrac{\fLa_A}{\lam_k} \)	\( \wt{\fu}_{kA} - \wt{\fe}_{kA}  \)
&=& \dfrac{1}{\lam_k} \fLa_A^{1/2} \fR_N(\lam_k) \fLa_A^{1/2} \wt{\fe}_{kA}
+ \sqrt T  \(\dfrac{\lam_k}{\wh{\lam}_k} -1 \) \dfrac{\fLa_A}{\lam_k} \wt{\fe}_{kA} \\
&& ~~~ + \sqrt T \(\dfrac{\fLa_A}{\lam_k} - \fI \) \wt{\fe}_{kA} + o_p(1).
\end{eqnarray*}
Multiplying $\cO_{N,k}$ on both sides yields
\begin{eqnarray} \label{eqn:u_kA-e_kA}
\sqrt T \(\wt{\fu}_{kA} - \wt{\fe}_{kA} \)
&=& \dfrac{1}{\lam_k} \cO_{N,k} \fLa_A^{1/2} \fR_N(\lam_k) \fLa_A^{1/2} \wt{\fe}_{kA}
\nonumber \\
&& ~~~ + \sqrt T \(\langle \wt{\fu}_{kA}, \wt{\fe}_{kA} \rangle -1 \) \wt{\fe}_{kA} + o_p(1),
\end{eqnarray}
where the last step  comes from the facts that $\cO_{N,k} \fLa_A \wt{\fe}_{kA}=0$ and $\cO_{N,k} \wt{\fe}_{kA} = 0$.
Write
\[
\fW_k^{\perp}=\sum_{i\neq k,i=1}^r \wt{\fe}_{iA} \wt{\fe}_{iA}^\mT.
\]
Then
\[
\wt{\fu}_{kA} =\langle \wt{\fu}_{kA}, \wt{\fe}_{kA} \rangle \wt{\fe}_{kA} + \fW_k^{\perp} \wt{\fu}_{kA}.
\]
Notice that  $\wt{\fu}_{kA}$ and $\wt{\fe}_{kA}$ are both unit vectors, thus
\begin{eqnarray}\label{eqn:inner_square}
1=\langle \wt{\fu}_{kA}, \wt{\fe}_{kA} \rangle^2 + \|\fW_k^{\perp} \wt{\fu}_{kA} \|^2.
\end{eqnarray}
From \eqref{eqn:u_kA-e_kA} and the fact that
$\fW_k^{\perp} \wt{\fe}_{kA} =0$, we get
\begin{eqnarray}\label{eqn:w_k_pen}
\fW_k^{\perp} \wt{\fu}_{kA}
&=& \fW_k^{\perp} \(\wt{\fu}_{kA} -\wt{\fe}_{kA}\) \nonumber\\
&=& \dfrac{1}{\lambda_k\sqrt T} \fW_k^{\perp} \cO_{N,k} \fLa_A^{1/2} \fR_N(\lambda_k) \fLa_A^{1/2} \ \wt{\fe}_{kA} + o_p(1/\sqrt{T}).
\end{eqnarray}
Combining \eqref{eqn:inner_square} and \eqref{eqn:w_k_pen} gives
\begin{eqnarray}
&& T\(1-\langle \wt{\fu}_{kA}, \wt{\fe}_{kA}\rangle^2 \)
= T \|\fW_k^{\perp} \wt{\fu}_{kA} \|^2 \nonumber\\
&=& \dfrac{1}{\lambda_k^2} \ \wt{\fe}_{kA}^\mT \fLa_A^{1/2} \fR_N(\lambda_k) \fLa_A^{1/2} \cO_{N,k} \fW_k^{\perp} \cO_{N,k} \fLa_A^{1/2} \fR_N(\lambda_k) \fLa_A^{1/2} \wt{\fe}_{kA} +o_p(1) \nonumber\\
&=& \sum_{i\neq k,i=1}^r \dfrac{\lambda_k^2}{(\lambda_k-\lambda_i)^2} \left[\dfrac{1}{\lambda_k} \fLa_A^{1/2} \fR_N(\lambda_k) \fLa_A^{1/2}\right] _{ki}^2 +o_p(1) \nonumber\\
&=&\sum_{i\neq k,i=1}^r \dfrac{\lambda_k\lambda_i}{(\lambda_k-\lambda_i)^2}  [\fR_N(\lambda_k)]_{ki}^2 +o_p(1).
\label{eqn:inner_u_kA}
\end{eqnarray}
By Lemma \ref{lem:limit_R}, the conclusion follows.
\end{proof}

\subsection*{B. Proof of Proposition \ref{thm:clt_spike_spike}}

\begin{proof}
By \eqref{eqn:u_kA-e_kA}, we obtain
\begin{eqnarray}
&& \sqrt T \(\wt{\fu}_{kA} -\wt{\fe}_{kA} \)  \nonumber\\
&=& \dfrac{1}{\lambda_k} \cO_{N,k} \ \fLa_A^{1/2} \fR_N(\lambda_k) \fLa_A^{1/2} \ \wt{\fe}_{kA}
+ \sqrt{T} \(\langle \wt{\fu}_{kA}, \wt{\fe}_{kA}\rangle -1 \) \wt{\fe}_{kA}
+o_p(1)  \nonumber\\
&=& \sum_{i\neq k,i=1}^r \dfrac{1}{\lambda_k-\lambda_i} \left[ \fLa_A^{1/2} \fR_N(\lambda_k) \fLa_A^{1/2} \right]_{ki} \wt{\fe}_{iA}
 - \sqrt{T}\(1-\langle \wt{\fu}_{kA}, \wt{\fe}_{kA}\rangle \) \wt{\fe}_{kA}
 +o_p(1)    \nonumber\\
&=& \sum_{i\neq k,i=1}^r \dfrac{\sqrt{\lambda_k\lambda_i}}{\lambda_k-\lambda_i} [\fR_N(\lambda_k)]_{ki} \wt{\fe}_{iA}
 - \sqrt{T}\(1-\langle \wt{\fu}_{kA}, \wt{\fe}_{kA}\rangle \) \wt{\fe}_{kA}  +o_p(1).
 \label{eqn:u_kA-e_kA_final}
\end{eqnarray}
By Lemma \ref{lem:limit_R}, the conclusion in Part (i) follows.

Next, for any fixed vectors $\fc_k, \ k=1,\ldots,r$, if there exist $i\neq j$ such that $\fc_i[j] \neq \fc_j[i]$, then by  \eqref{eqn:u_kA-e_kA_final} and \eqref{eqn:inner_u_kA} and Lemma \ref{lem:limit_R}, we have
\begin{eqnarray}\label{eqn:ck_uA}
\sqrt T \sum_{k=1}^r \fc_k^\mT \(\wt{\fu}_{kA} - \wt{\fe}_{kA} \) = I_1 -\dfrac{1}{\sqrt T} I_2 +o_p(1),
\end{eqnarray}
where
\begin{eqnarray*}
I_1 &=& \sum_{k=1}^r \sum_{i\neq k, i=1}^r \dfrac{\sqrt{\lam_k \lam_i}}{\lam_k - \lam_i} \
	[\fR_N(\lam_k)]_{ki} \ \fc_k[i]  \\
	&=& \sum_{k=1}^{r-1} \sum_{i=k+1}^r \dfrac{\sqrt{\lam_k \lam_i}}{\lam_k - \lam_i}
	\(\fc_k[i] - \fc_i[k] \) \ [\fR_N(\lam_k)]_{ki} \\
	&\wk & N\(0, \sum_{k=1}^{r-1} \sum_{i=k+1}^r \om_{ki} (\fc_k[i] - \fc_i[k] )^2 \) ,
\end{eqnarray*}
and
\begin{eqnarray*}
	I_2 &=& \sum_{k=1}^r \fc_k[k] \cdot T\(1- \langle \wt{\fu}_{kA}, \wt{\fe}_{kA}\rangle \)  \\
	&=& \dfrac 12 \sum_{k=1}^r \sum_{i\neq k,i=1}^r  \dfrac{\lam_k\lam_i}{(\lam_k-\lam_i)^2} \fc_k[k]
	[\fR_N(\lam_k)]_{ki}^2 + o_p(1) \\
	&=& \dfrac 12 \sum_{k=1}^{r-1} \sum_{i=k+1}^r  \dfrac{\lam_k\lam_i}{(\lam_k-\lam_i)^2}
	\( \fc_k[k]+\fc_i[i]\)  [\fR_N(\lam_k)]_{ki}^2 + o_p(1) \\
	&\wk& \dfrac 12 \sum_{k=1}^{r-1} \sum_{i=k+1}^r \om_{ki} \( \fc_k[k]+\fc_i[i]\)  Z_{ki}^2,
\end{eqnarray*}
where $Z_{ki}\stackrel{iid}{\sim} N(0,1)$.
The conclusion in Part (ii) follows.

Finally, by \eqref{eqn:exp_CLT_evalue} and \eqref{eqn:u_kA-e_kA_final},  Proposition \ref{thm:clt_ev_spike} and  Lemma \ref{lem:limit_R},
$\sqrt T (\wh{\lam}_\ell/\lam_\ell -1 )$ and $\sqrt T \(\wt{\fu}_{kA}-\fe_{kA} \)$ are asymptotically independent.
\end{proof}

\subsection*{C. Proof of Proposition \ref{thm:clt_ev_norm}}

\begin{proof}
Recall that
$ \fu_{kB}=\(\wh{\lam}_k\fI-\fS_{22} \)^{-1} \fS_{21} \fu_{kA} $. We have
\[
\|\fu_{kA}\|^2 = 1-\fu_{kB}^\mT \fu_{kB}
= 1-\fu_{kA}^\mT \fS_{12} (\wh{\lam}_k\fI-\fS_{22})^{-2} \fS_{21} \fu_{kA}.
\]
Dividing both sides by $\|\fu_{kA}\|^2$ yields
\begin{eqnarray}\label{eqn:u_kA_inv}
1=\dfrac{1}{\|\fu_{kA}\|^2} - \wt{\fu}_{kA}^\mT \fS_{12} \(\wh{\lam}_k\fI -\fS_{22} \)^{-2} \fS_{21} \wt{\fu}_{kA}.
\end{eqnarray}
Hence,
\begin{eqnarray} \label{eqn:star_2}
1-\|\fu_{kA}\|^2 &=& 1- \dfrac{1}{1+\wt{\fu}_{kA}^\mT \fS_{12} \(\wh{\lam}_k\fI-\fS_{22} \)^{-2} \fS_{21} \wt{\fu}_{kA} }  \nonumber\\
&=& \wt{\fu}_{kA}^\mT \fS_{12} \(\wh{\lam}_k\fI-\fS_{22} \)^{-2} \fS_{21} \wt{\fu}_{kA} + \wt{\vep}_{kA},
\end{eqnarray}	
where
\[
\wt{\vep}_{kA} =  \dfrac{ - \( \wt{\fu}_{kA}^\mT \fS_{12} \(\wh{\lam}_k\fI-\fS_{22} \)^{-2} \fS_{21} \wt{\fu}_{kA} \)^2 }{1+\wt{\fu}_{kA}^\mT \fS_{12} \(\wh{\lam}_k\fI-\fS_{22} \)^{-2} \fS_{21} \wt{\fu}_{kA} }.
\]
To derive the CLT of $\|\fu_{kA}\|^2$, we need to analyze the term $\fS_{12} (\wh{\lam}_k\fI-\fS_{22})^{-2} \fS_{21}$.
We first study the difference when replacing $\wh{\lam}_k$ with $\lambda_k$ in  $\fS_{12} (\wh{\lam}_k\fI-\fS_{22})^{-2} \fS_{21}$.
We have
\begin{eqnarray*}
&& \fS_{12} \lb (\wh{\lam}_k\fI-\fS_{22})^{-2} - (\lam_k\fI-\fS_{22})^{-2} \rb \fS_{21} \\
&=& \fS_{12} \lb (\wh{\lam}_k\fI-\fS_{22})^{-1} - (\lambda_k\fI-\fS_{22})^{-1} \rb
\lb (\wh{\lam}_k\fI-\fS_{22})^{-1} + (\lambda_k\fI-\fS_{22})^{-1} \rb \fS_{21}  \\	&=& L_1+L_2,
\end{eqnarray*}
where
\begin{eqnarray*}
L_1 &=&
(\lambda_k-\wh{\lam}_k)\fS_{12} (\wh{\lam}_k\fI-\fS_{22})^{-1} (\lambda_k\fI-\fS_{22})^{-2} \fS_{21} ,  \\
L_2 &=&    (\lambda_k-\wh{\lam}_k)\fS_{12} (\wh{\lam}_k\fI-\fS_{22})^{-2} (\lambda_k\fI-\fS_{22})^{-1} \fS_{21} .
\end{eqnarray*}
Define
\[
\wt{\fQ}_N(\lambda_k) = \dfrac{1}{\sqrt T}
\left\{ \dfrac 1T \fZ_A \fX_B^\mT (\fI-\lambda_k^{-1}\fS_{22})^{-3} \fX_B \fZ_A^\mT - \tr \lb (\fI-\lambda_k^{-1}\fS_{22})^{-3} \fS_{22} \rb \cdot \fI \right\} .
\]
By Theorem 7.1 of \cite{BY08}, $\wt{\fQ}_N(\lambda_k)$ converges weakly to a symmetric Gaussian random matrix $\fQ^*$ with zero mean  and  finite covariance functions.
Using the definitions of $\fS_{12}$ and $\fS_{21}$, we can rewrite $L_1$ as
\begin{eqnarray*}
L_1 &=&
(\lambda_k-\wh{\lam}_k) \ \fS_{12} (\lambda_k\fI-\fS_{22})^{-3} \fS_{21} \\
&& +
(\lambda_k-\wh{\lam}_k)^2 \  \fS_{12} (\wh{\lam}_k\fI-\fS_{22})^{-1}(\lambda_k\fI-\fS_{22})^{-3} \fS_{21} \\
&=& (\lambda_k-\wh{\lam}_k) \dfrac{1}{T^2} \fLa_A^{1/2} \fZ_A \fX_B^\mT (\lambda_k\fI-\fS_{22})^{-3} \fX_B \fZ_A^\mT \fLa_A^{1/2} \\
&& +(\lambda_k-\wh{\lam}_k)^2 \dfrac{1}{T^2} \fLa_A^{1/2} \fZ_A \fX_B^\mT (\wh{\lam}_k\fI-\fS_{22})^{-1} (\lambda_k\fI-\fS_{22})^{-3} \fX_B \fZ_A^\mT \fLa_A^{1/2}  \\
&=& \dfrac{1}{\lambda_k} \(1-\dfrac{\wh{\lam}_k}{\lambda_k} \) \dfrac{\fLa_A}{\lambda_k} \dfrac 1T \tr \lb (\fI-\lambda_k^{-1}\fS_{22})^{-3} \fS_{22}\rb +
O_p(1/N^2),
\end{eqnarray*}
where the last two steps follow from Assumption A and the facts that $1-\wh{\lam}_k/\lam_k=O_p(1/\sqrt T)$, $\wt{\fQ}_N(\lam_k)=O_p(1)$ and  $\|\fS_{22}\|=O_p(1)$.
Similarly, we obtain
\[
L_2 = \dfrac{1}{\lambda_k} \(1-\dfrac{\wh{\lam}_k}{\lambda_k} \) \dfrac{\fLa_A}{\lambda_k} \dfrac 1T \tr \lb (\fI-\lambda_k^{-1}\fS_{22})^{-3} \fS_{22}\rb +
O_p(1/N^2).
\]
In addition, we have
\begin{eqnarray*}
 \fS_{12} \(\lambda_k\fI-\fS_{22} \)^{-2} \fS_{21}
 &=&  \dfrac{1}{T^2} \fLa_A^{1/2} \fZ_A \fX_B^\mT \(\lambda_k\fI-\fS_{22} \)^{-2} \fX_B \fZ_A^\mT \fLa_A^{1/2} \\
 &=&  \dfrac{1}{\lambda_k^2\sqrt T} \fLa_A^{1/2} \wt{\fR}_N(\lambda_k) \fLa_A^{1/2} + \dfrac{\fLa_A}{\lambda_k^2} \dfrac 1T \tr\lb \(\fI-\lambda_k^{-1} \fS_{22} \)^{-2} \fS_{22} \rb ,
\end{eqnarray*}
where
\[
\wt{\fR}_N(\lambda_k) = \dfrac{1}{\sqrt T} \left\{
\dfrac 1T \fZ_A \fX_B^\mT \(\fI-\lambda_k^{-1}\fS_{22} \)^{-2} \fX_B \fZ_A^\mT -\tr \lb\(\fI-\lambda_k^{-1}\fS_{22} \)^{-2} \fS_{22} \rb\cdot \fI
\right\} .
\]
By Theorem 7.1 of \cite{BY08}, $\wt{\fR}_N(\lam_k)$ converges weakly to a symmetric Gaussian  random matrix $\wt{\fR}$ with zero mean  and finite covariance functions.  Combining the results above,
we obtain
\begin{equation} \label{temp:result_1}
\aligned
&\fS_{12} \(\wh{\lam}_k\fI-\fS_{22} \)^{-2} \fS_{21} \\
=& \fS_{12} \(\lambda_k\fI-\fS_{22} \)^{_2} \fS_{21}
+ \fS_{12} \lb \(\wh{\lam}_k\fI-\fS_{22} \)^{-2} -\(\lambda_k\fI-\fS_{22} \)^{-2} \rb \fS_{21} \\
=& \dfrac{1}{\lambda_k^2\sqrt T} \fLa_A^{1/2} \wt{\fR}_N(\lambda_k) \fLa_A^{1/2} + \dfrac{\fLa_A}{\lambda_k^2} \dfrac 1T \tr\lb \(\fI-\lambda_k^{-1} \fS_{22} \)^{-2} \fS_{22} \rb \\
& + \dfrac{2}{\lambda_k} \(1-\dfrac{\wh{\lam}_k}{\lambda_k} \) \dfrac{\fLa_A}{\lambda_k} \dfrac 1T \tr\lb \(\fI-\lambda_k^{-1} \fS_{22} \)^{-3} \fS_{22} \rb
+O_p(1/N^2)\\
=& O_p\(\dfrac{1}{\lam_k} \) .
\endaligned
\end{equation}
It follows that
$\wt{\vep}_{kA}=O_p(1/\lam_k^2)$, and
\begin{equation}\label{eq:uka_norm_1}
1-\|\fu_{kA}\|^2 = \wt{\fu}_{kA}^\mT \fS_{12} \(\wh{\lam}_k\fI-\fS_{22} \)^{-2} \fS_{21} \wt{\fu}_{kA} + O_p(1/\lam_k^2).
\end{equation}

Next, we derive the limit of $\lam_k\(1-\|\fu_{kA}||^2 \)$.
By Assumption A and \eqref{temp:result_1}, we have
\begin{eqnarray*}
&& \lam_k \(1-\|\fu_{kA}\|^2 \) =
\lam_k \cdot  \wt{\fu}_{kA}^\mT \fS_{12} \(\wh{\lam}_k\fI-\fS_{22} \)^{-2} \fS_{21} \wt{\fu}_{kA}  + O_p(1/\lam_k)   \\
&=& \dfrac{1}{\lambda_k\sqrt T} \wt{\fu}_{kA}^\mT \fLa_A^{1/2} \wt{\fR}_N(\lambda_k) \fLa_A^{1/2} \wt{\fu}_{kA}
 + \dfrac 1T \tr\lb \(\fI-\lambda_k^{-1}\fS_{22} \)^{-2}\fS_{22} \rb \cdot \dfrac{ \wt{\fu}_{kA}^\mT \fLa_A \wt{\fu}_{kA} }{\lambda_k} \\
&& + 2\(1-\dfrac{\wh{\lam}_k}{\lambda_k} \) \dfrac 1T \tr\lb \(\fI-\lambda_k^{-1}\fS_{22} \)^{-3}\fS_{22} \rb \cdot \dfrac{\wt{\fu}_{kA}^\mT\fLa_A \wt{\fu}_{kA} }{\lambda_k}+o_p(1/\sqrt T).
\end{eqnarray*}
Replacing  $\wt{\fu}_{kA}$ with $\wt{\fe}_{kA}$ and using Proposition \ref{thm:clt_spike_spike}, we get
\begin{equation}\label{eq:uka_norm_2}
\aligned
\lam_k \(1-\|\fu_{kA}\|^2 \)
=&
\dfrac{1}{\sqrt T} [\wt{\fR}_N(\lambda_k)]_{kk} + \dfrac 1T \tr\lb \(\fI-\lambda_k^{-1}\fS_{22} \)^{-2}\fS_{22} \rb  \\
& + 2\(1-\dfrac{\wh{\lam}_k}{\lambda_k}\) \dfrac 1T \tr \lb \(\fI-\lambda_k^{-1}\fS_{22} \)^{-3}\fS_{22} \rb +o_p(1/\sqrt T).
\endaligned
\end{equation}
Under Assumptions A--C, we have
\[
\dfrac 1T \tr\lb \(\fI-\lambda_k^{-1}\fS_{22} \)^{-2}\fS_{22} \rb
\ \pc \ \rho \int x dF(x),
\]
where $F(\cdot)$ is the LSD of $\fS_{22}$.
According to Theorem 1.1 of \cite{SB95},  the Stieltjes transform of $F$, $m_F$, is the unique solution in the set $\{m_F\in \bC^+: -(1-\rho)/z + \rho m_F \in \bC^+ \}$ to the equation
\[
m_F=\int \dfrac{dH(\tau)}{\tau(1-\rho-\rho z m_F)-z}.
\]
Therefore, $ \lam_k \(1-\|\fu_{kA}\|^2 \)  \pc  \rho \int x dF(x)$.

We now consider the limiting distribution of $\lam_k \(1-\|\fu_{kA}\|^2 \)$.
By \eqref{eq:uka_norm_2} and \eqref{eqn:exp_CLT_evalue}, we get
\begin{eqnarray*}
&& \sqrt T \left\{
\lambda_k\(1-\|\fu_{kA}\|^2 \) - \dfrac 1T \tr \lb \(\fI-\lambda_k^{-1} \fS_{22} \)^{-2} \fS_{22}\rb
\right\} \\
&=& [\wt{\fR}_N(\lambda_k)]_{kk}  - 2 [\fR_N(\lam_k) ]_{kk}  \cdot
\dfrac 1T \tr \lb \(\fI-\lambda_k^{-1}\fS_{22} \)^{-3}\fS_{22} \rb +o_p(1). 
\end{eqnarray*}
Notice that
\begin{eqnarray*}
	\begin{pmatrix}
		[\fR_N(\lambda)]_{kk} \\
		[\wt{\fR}_N(\lambda)]_{kk}
	\end{pmatrix}
	= \begin{pmatrix}
		\dfrac{1}{\sqrt T} [\fZ_A\wt{\fC}_N \fZ_A^\mT -\tr(\wt{\fC}_N) \cdot \fI  ]_{kk} \\
		\dfrac{1}{\sqrt T} [\fZ_A\wt{\fD}_N \fZ_A^\mT -\tr(\wt{\fD}_N) \cdot \fI  ]_{kk}
	\end{pmatrix}
	=
	\begin{pmatrix}
		\dfrac{1}{\sqrt T} \(\fz_{(k)}^\mT \wt{\fC}_N \fz_{(k)} -\tr(\wt{\fC}_N) \) \\
		\dfrac{1}{\sqrt T}  \(\fz_{(k)}^\mT \wt{\fD}_N \fz_{(k)} -\tr(\wt{\fD}_N) \)
	\end{pmatrix},
\end{eqnarray*}
where
\begin{eqnarray*}
\wt{\fC}_N &=& \wt{\fC}_N(\lam)
= \fI + \dfrac 1T \fX_B^\mT \(\lambda\fI -\fS_{22} \)^{-1} \fX_B ,
\\
\wt{\fD}_N &=& \wt{\fD}_N(\lam)
= \dfrac 1T \fX_B^\mT \(\fI -\lambda^{-1}\fS_{22} \)^{-2} \fX_B.
	\end{eqnarray*}
To finish the proof, we need the following lemma, which will be proved at the end of this subsection.

\begin{lem} \label{lem:limit_double_R}
Under Assumptions A--C,
if $\lam \asymp T$, then
$\begin{pmatrix}
[\fR_N(\lambda)]_{kk} \\
[\wt{\fR}_N(\lambda)]_{kk}
\end{pmatrix}$
converges weakly to a zero-mean Gaussian vector with covariance matrix
$ \fOm_k =
\begin{pmatrix}
\fOm_{k,11}  & \fOm_{k,12}  \\
\fOm_{k,21}  & \fOm_{k,22}
\end{pmatrix},
$
where
\begin{eqnarray*}
\fOm_{k,11} &=& \E(\fz_1[k])^4 -1 ,   \\
\fOm_{k,22} &=& \(\E(\fz_1[k])^4 -3 \) \cdot
\rho^2 \(\int x dF(x) \)^2
+ 2\rho \int x^2 dF(x),  \\
\fOm_{k,12} &=& \(\E(\fz_1[k])^4 -1 \) \cdot \rho \int x dF(x),
\end{eqnarray*}
where $F(\cdot)$ is the LSD of $\fS_{22}$.
\end{lem}

Based on the lemma above, we conclude that
\[
\sqrt T \left\{
\lambda_k\(1-\|\fu_{kA}\|^2 \) - \dfrac 1T \tr \lb \(\fI-\lambda_k^{-1} \fS_{22} \)^{-2} \fS_{22}\rb
\right\}
\ \wk \ N(0, \si_{kA}^2),
\]
where
\begin{eqnarray*}
\si_{kA}^2 &=& \(-2\rho\int x dF(x), ~~ 1 \)
\begin{pmatrix}
	\fOm_{k,11}  & \fOm_{k,12} \\
	\fOm_{k,12}  & \fOm_{k,22}
\end{pmatrix}
\begin{pmatrix}
	-2\rho\int x dF(x) \\
	1
\end{pmatrix} \\
&=& \(\E(\fz_1[k])^4 -3 \) \cdot \rho^2 \(\int xdF(x)\)^2 + 2\rho\int x^2 dF(x).
\end{eqnarray*}

Further, from Assumption A that $\lam_k=O_p(N)$ for $k\le r$, and the boundedness of the eigenvalues of $\fS_{22}$ and $\max_{j>r} \wh{\lam}_j$ with high probability, it follows that
\[
\sqrt T \left[ \dfrac 1T \tr\(\(\fI-\lam_k^{-1}\fS_{22}\)^{-2}\fS_{22} \) - \dfrac 1T \tr(\fS_{22})\right] \pc 0,
\]
and
\[
\sqrt T \left[ \dfrac 1T \sum_{j=r+1}^N
\dfrac{\wh{\lam}_j}{(1-\wh{\lam}_j/\lam_k)^2}
- \dfrac 1T \sum_{j=r+1}^N \wh{\lam}_j
\right] \pc 0.
\]
Recall that
\[
\sqrt T \left[ \dfrac 1T \tr(\fS_{22}) - \dfrac 1T \sum_{j=r+1}^N \wh{\lam}_j \right] \pc 0,
\]
which has been shown in the proof of \eqref{eqn:delta}.
Therefore, Proposition \ref{thm:clt_ev_norm} follows.

\end{proof}

At last, we prove Lemma \ref{lem:limit_double_R}.
\begin{proof}[Proof of Lemma \ref{lem:limit_double_R}]

By Theorem 2.1 of \cite{WSY2014},
$\begin{pmatrix}
[\fR_N(\lambda)]_{kk} \\
[\wt{\fR}_N(\lambda)]_{kk}
\end{pmatrix}$
converges weakly to a zero-mean Gaussian vector with covariance matrix
$ \fOm_k =
\begin{pmatrix}
\fOm_{k,11}  & \fOm_{k,12}  \\
\fOm_{k,21}  & \fOm_{k,22}
\end{pmatrix},
$
where
\begin{eqnarray*}
	\fOm_{k,11} &=& \om_1 A_1 + (\tau_1-\om_1) (A_2+A_3),  \\
	\fOm_{k,22} &=& \om_2 A_1 + (\tau_2-\om_2) (A_2+A_3),  \\
	\fOm_{k,12} &=& \om_3 A_1 + (\tau_3-\om_3) (A_2+A_3),
\end{eqnarray*}
with
\[
	A_1=\E(\fz_{(k)}[1])^4 -1 = \E(\fz_1[k])^4 -1,
	~~~~~~~ A_2=1, ~~~~~~~~ A_3=1,
\]
\[
	\tau_1 = \lim_N \dfrac 1T \tr\(\wt{\fC}_N^2\),
	~~~~
	\tau_2 = \lim_N \dfrac 1T \tr\(\wt{\fD}_N^2\),
	~~~~
	\tau_3 = \lim_N \dfrac 1T \tr\(\wt{\fC}_N\wt{\fD}_N\),
\]
\[
	\om_1 = \lim_N \dfrac 1T \tr\(\wt{\fC}_N \circ \wt{\fC}_N \),
	~~~~
	\om_2 = \lim_N \dfrac 1T \tr\(\wt{\fD}_N \circ \wt{\fD}_N \),
	~~~~
	\om_3 = \lim_N \dfrac 1T \tr\(\wt{\fC}_N \circ \wt{\fD}_N \),
\]
where $\fA \circ \fB$ denotes the Hadamard product of two symmetric matrices $\fA$ and $\fB$, i.e.
$[\fA \circ \fB]_{ij}=[\fA]_{ij}\cdot[\fB]_{ij}$.

To prove Lemma \ref{lem:limit_double_R}, we need to compute the values of $\tau_i$ and $\om_i$, $i=1,2,3$.
We start with $\tau_i$'s. From the definitions of $\wt{\fC}_N$ and $\wt{\fD}_N$, it is easy to check that
\begin{eqnarray*}
\tau_1&=&\lim_N \dfrac 1T \tr(\wt{\fC}_N^2) \\
&=& \lim_N \dfrac 1T \tr\lb
\fI +\dfrac 2T \fX_B^\mT \(\lambda\fI-\fS_{22} \)^{-1} \fX_B + \dfrac 1T \fX_B^\mT (\lambda\fI-\fS_{22})^{-1} \fX_B \cdot \dfrac 1T \fX_B^\mT \(\lambda\fI-\fS_{22}\)^{-1} \fX_B
\rb \\
&=& 1+ 2 \lim_N \dfrac 1T \tr \lb
\(\lambda\fI-\fS_{22} \)^{-1} \fS_{22} \rb
+ \lim_N \dfrac 1T \tr \lb
\(\lambda\fI-\fS_{22} \)^{-1} \fS_{22}
\(\lambda\fI-\fS_{22} \)^{-1} \fS_{22}  \rb \\
&=& 1;
\end{eqnarray*}
	
\begin{eqnarray*}
	\tau_2&=&\lim_N \dfrac 1T \tr(\wt{\fD}_N^2)
	= \lim_N \dfrac 1T \tr\lb
	\dfrac 1T \fX_B^\mT (\fI-\lambda^{-1}\fS_{22})^{-2} \fX_B \cdot \dfrac 1T \fX_B^\mT (\fI-\lambda^{-1}\fS_{22})^{-2} \fX_B
	\rb \\
	&=& \lim_N \dfrac 1T \tr\lb
	\dfrac 1T \fX_B^\mT (\fI-\lambda^{-1}\fS_{22})^{-2} \fS_{22} (\fI-\lambda^{-1}\fS_{22})^{-2} \fS_{22}
	\rb \\
	&=& \rho \int x^2 dF(x);
\end{eqnarray*}
and
\begin{eqnarray*}
	\tau_3&=&\lim_N \dfrac 1T \tr(\wt{\fC}_N\wt{\fD}_N) \\
	&=& \lim_N \dfrac 1T \tr\lb
	\dfrac 1T \fX_B^\mT (\fI-\lambda^{-1}\fS_{22})^{-2} \fX_B
	+\dfrac 1T \fX_B^\mT (\lambda\fI-\fS_{22})^{-1} \fX_B
	\cdot \dfrac 1T \fX_B^\mT (\fI-\lambda^{-1}\fS_{22})^{-2} \fX_B \rb \\
	&=& \lim_N \dfrac 1T \tr \lb
	\(\fI-\lambda^{-1}\fS_{22}\)^{-2}\fS_{22}
	\rb
	+\lim_N \dfrac 1T \tr \lb
	\(\lambda\fI-\fS_{22}\)^{-1} \fS_{22}
	\(\fI-\lambda^{-1}\fS_{22}\)^{-2} \fS_{22}
	\rb  \\
	&=& \rho \int x\, dF(x).
\end{eqnarray*}

Next, we calculate the values of $\om_i$, $i=1,2,3$.
Denote by $\fX_{Bi}$ the matrix obtained from $\fX_B$ by deleting its $i$th column.    Then
\[
\fS_{22} =\dfrac 1T \fX_{Bi} \fX_{Bi}^\mT
+ \dfrac 1T \fx_i^{(B)} {\fx_i^{(B)}}^\mT .
\]
Recall that for any invertible matrix $\fA$ and vector $\fr$,  one has
\[
(\fA+\fr\fr^\mT)^{-1} = \fA^{-1} -
\dfrac{\fA^{-1}\fr\fr^\mT\fA^{-1}}{1+\fr^\mT\fA^{-1}\fr},
\]
and
\[
\fr^\mT(\fA+\fr\fr^\mT)^{-2}\fr =
\dfrac{\fr^\mT\fA^{-2}\fr}{\(1+\fr^\mT\fA^{-1}\fr \)^2}.
\]

By Assumption A, we have
\begin{eqnarray}\label{eqn:limit_Dii}
[\wt{\fD}_N]_{ii}
&=&  \dfrac 1T {\fx_i^{(B)}}^\mT (\lam^{-1} \fS_{22}-\fI )^{-2} {\fx_i^{(B)}} \nonumber \\
&=& \dfrac{\dfrac 1T
{\fx_i^{(B)}}^\mT \(1/(T\lam)\fX_{Bi}\fX_{Bi}^\mT -\fI \)^{-2} \fx_i^{(B)}
}
{\(
1+\dfrac 1T {\fx_i^{(B)}}^\mT \(T^{-1}\fX_{Bi}\fX_{Bi}^\mT -\lam\fI \)^{-1} \fx_i^{(B)}
\)^2	
} \nonumber \\
&\pc& \rho \int x dF(x),
\end{eqnarray}
and
\begin{eqnarray}\label{eqn:limit_Cii}
[\wt{\fC}_N]_{ii}
&=& 1-\dfrac 1T  {\fx_i^{(B)}}^\mT (\fS_{22}-\lam\fI)^{-1}  {\fx_i^{(B)}}
\nonumber \\
&=& 1-  \dfrac{ \dfrac 1T  {\fx_i^{(B)}}^\mT \(T^{-1} \fX_{Bi}\fX_{Bi}^\mT -\lam\fI \)^{-1} \fx_i^{(B)} }{1+\dfrac 1T  {\fx_i^{(B)}}^\mT \(T^{-1} \fX_{Bi}\fX_{Bi}^\mT -\lam\fI \)^{-1} \fx_i^{(B)} }  \nonumber \\
&\pc& 1.
\end{eqnarray}
It is easy to see that
\[
\lim_N \dfrac 1T \E \tr\(T^{-1}\fX_B^T (\fI-\lam^{-1}\fS_{22} )^{-2} \fX_B \)^4 \indic(\cF_s)
< \infty,
\]
and
\[
\lim_N \dfrac 1T \E\tr\(T^{-1} \fX_B^T (\lam\fI-\fS_{22})^{-1} \fX_B \)^4 \indic(\cF_s) =0.
\]
Hence
\[
\sup_N \E([\wt{\fD}_N]_{ii}  \indic(\cF_s) )^4
\le \sup_N \dfrac 1T \E\tr\(T^{-1}\fX_B^T (\fI-\lam^{-1}\fS_{22})^{-2}\fX_B \)^4 \indic(\cF_s)
< \infty
\]
and $\sup_N \E([\wt{\fC}_N]_{ii}\indic(\cF_s) )^4  <\infty$,
which implies that the family of random variables $\{[\wt{\fD}_N]_{ii}^2 \indic(\cF_s) \}$ and $\{[\wt{\fC}_N]_{ii}^2\indic(\cF_s) \}$ are uniformly integrable.
Together with \eqref{eqn:limit_Dii} and \eqref{eqn:limit_Cii} and the fact that $\indic(\cF_s)=1$ with high probability, we get
\[
\E\left|
\dfrac 1T \sum_{i=1}^T [\wt{\fD}_N]_{ii}^2 \indic(\cF_s) - \(\rho\int xdF(x) \)^2
\right|
\le \E \left|
[\wt{\fD}_N]_{11}^2\indic(\cF_s) - \(\rho\int xdF(x) \)^2
\right|
\to 0,
\]
and
\[
\E\left|\dfrac 1T \sum_{t=1}^T [\wt{\fC}_N]_{tt}^2\indic(\cF_s) -1 \right| \to 0.
\]
Thus, $\dfrac 1T \sum_{t=1}^T [\wt{\fD}_N]_{tt}^2\indic(\cF_s) \pc \(\rho\int xdF(x) \)^2$ and $\dfrac 1T \sum_{t=1}^T [\wt{\fC}_N]_{tt}^2\indic(\cF_s) \pc 1$.
Moreover, noting that in the event $\cF_s$, $[\wt{\fC}_N]_{ii}$ and $[\wt{\fD}_N]_{ii}$, $i=1,\ldots, T$, are uniformly bounded and that $P(\cF_s)\to 1$, we obtain
\[
\dfrac 1T \sum_{i=1}^T [\wt{\fD}_N]_{ii}^2 \pc \(\rho\int xdF(x) \)^2,
~~ \textrm{and} ~
\dfrac 1T \sum_{i=1}^T [\wt{\fC}_N]_{ii}^2 \pc 1
\]
Therefore,
\begin{eqnarray*}
\om_1 = \lim_N \dfrac 1T \tr\(\wt{\fC}_N \circ \wt{\fC}_N \)
= 1 ,
\end{eqnarray*}	
\begin{eqnarray*}
\om_2 = \lim_N \dfrac 1T \tr\(\wt{\fD}_N \circ \wt{\fD}_N \)
= \rho^2 \ \(\int x dF(x) \)^2 ,
\end{eqnarray*}
and
\begin{eqnarray*}
\om_3 = \lim_N \dfrac 1T \tr\(\wt{\fC}_N \circ \wt{\fD}_N \)
= \rho \  \int x dF(x) .
\end{eqnarray*}
In summary,
\[
\begin{pmatrix}
	[\fR_N(\lambda)]_{kk} \\
	[\wt{\fR}_N(\lambda)]_{kk}
\end{pmatrix}
	~~ \wk ~~ N\(0, \fOm_k=
\begin{pmatrix}
	\fOm_{k,11}  & \fOm_{k,12} \\
	\fOm_{k,12}  & \fOm_{k,22}
\end{pmatrix}
	\) ,
\]
where
\begin{eqnarray*}
\fOm_{k,11} &=& \E(\fz_1[k])^4 -1   \\
\fOm_{k,22} &=& \(\E(\fz_1[k])^4 -3 \) \cdot \rho^2 \ \(\int x dF(x) \)^2 + 2\rho \int x^2 dF(x), \mbox{and }  \\
\fOm_{k,12} &=& \(\E(\fz_1[k])^4 -1 \) \cdot \rho \int x dF(x) .
\end{eqnarray*}
\end{proof}

\section*{S5. Proof of Theorem \ref{thm:clt_ev} }
\label{subsec:proof_thm:clt_ev}
\begin{proof}
Recall that $\langle \fv_k, \wh{\fv}_k \rangle=\langle \fu_k, \fe_k \rangle$, where $\fv_k, \wh{\fv}_k$ are the $k$th principal eigenvector of $\fSi$ and $k$th principal sample eigenvector of $\wh{\fSi}_N$, respectively, and $\fu_k$ is the $k$th eigenvector of sample covariance matrix $\fS_N$.
Under  Assumptions A--C, by Propositions~\ref{thm:clt_ev_spike} and \ref{thm:clt_ev_norm}, we get
\begin{eqnarray*}
&& T \(
1-\langle \fu_k, \fe_k \rangle^2 -\dfrac{1}{T\lambda_k}
\sum_{j=r+1}^N \dfrac{\wh{\lam}_j}{(1-\wh{\lam}_j/\lambda_k)^2}
\) \\
&=& T \(
1-\|\fu_{kA}\|^2 + \|\fu_{kA}\|^2 \(1-\langle \wt{\fu}_{kA}, \wt{\fe}_{kA} \rangle^2
\) - \dfrac{1}{T\lambda_k} \sum_{j=r+1}^N \dfrac{\wh{\lam}_j}{(1-\wh{\lam}_j/\lambda_k)^2}
\) \\
&=& \dfrac{T}{\lambda_k}
\(
\lambda_k\(1-\|\fu_{kA}\|^2 \) -\dfrac 1T
\sum_{j=r+1}^N \dfrac{\wh{\lam}_j}{(1-\wh{\lam}_j/\lambda_k)^2}
\)
+T \|\fu_{kA}\|^2 \(1-\langle \wt{\fu}_{kA}, \wt{\fe}_{kA} \rangle^2 \) \\
&\wk& \sum_{i\neq k,i=1}^r \om_{ki}  \ Z_i^2,
\end{eqnarray*}
where $Z_i$ are \hbox{i.i.d.} standard normal random variables.

It remains to prove that
\[
\dfrac{1}{\lambda_k} \sum_{j=r+1}^N
\dfrac{\wh{\lambda}_j}{\(1-\wh{\lambda}_j/\lambda_k \)^2}
- \dfrac{1}{\wh{\lambda}_k} \sum_{j=r+1}^N
\dfrac{\wh{\lambda}_j}{\(1-\wh{\lambda}_j/\wh{\lambda}_k \)^2}
\pc 0.
\]
Rewrite the term as
\begin{eqnarray*}
&&\dfrac{1}{\lambda_k} \sum_{j=r+1}^N
\dfrac{\wh{\lambda}_j}{\(1-\wh{\lambda}_j/\lambda_k \)^2}
- \dfrac{1}{\wh{\lambda}_k} \sum_{j=r+1}^N
\dfrac{\wh{\lambda}_j}{\(1-\wh{\lambda}_j/\lambda_k \)^2}
\\
&& +\dfrac{1}{\wh{\lambda}_k} \sum_{j=r+1}^N
\dfrac{\wh{\lambda}_j}{\(1-\wh{\lambda}_j/\lambda_k \)^2}
- \dfrac{1}{\wh{\lambda}_k} \sum_{j=r+1}^N
\dfrac{\wh{\lambda}_j}{\(1-\wh{\lambda}_j/\wh{\lambda}_k \)^2}
\\
&=& \dfrac{1}{\wh{\lambda}_k}
\(\dfrac{\wh{\lambda}_k}{\lambda_k} -1 \)
\sum_{j=r+1}^N \dfrac{\wh{\lambda}_j}{\(1-\wh{\lambda}_j/\lambda_k \)^2}
\\
&&+ \dfrac{1}{\wh{\lambda}_k^2} \(1-\dfrac{\wh{\lambda}_k}{\lambda_k} \)
\sum_{j=r+1}^N \dfrac{\wh{\lambda}_j^2\(2-\wh{\lambda}_j/\lambda_k -\wh{\lambda}_j/\wh{\lambda}_k \) }{\(1-\wh{\lambda}_j/\wh{\lambda}_k \)^2 \(1-\wh{\lambda}_j/\lambda_k \)^2}.
\end{eqnarray*}
By Assumption A--C and Theorem 1, the term converge to zero in probability.
\end{proof}

\section*{S6. Proofs of Theorem \ref{thm:test_tv} and Corollary \ref{cor:test_tvk} }
\label{ssec:pf_thm:test_tv}

\begin{proof}[Proof of Theorem 6]
By \eqref{eqn:u_kA-e_kA_final}, we have
\begin{equation}\label{star1}
\sqrt{N} \(\wt{\fu}_{kA}^{(i)} -\wt{\fe}_{kA} \)
=\sqrt{\frac{N}{T_i}} \sum_{\ell\neq k, \ell=1}^r
\frac{\sqrt{\lambda_k^{{(i)}}\lambda_\ell^{{(i)}}}}{\lambda_k^{{(i)}}-\lambda_\ell^{{(i)}}} [\fR_N^{{(i)}}(\lambda_k^{{(i)}})]_{k \ell} \ \wt{\fe}_{\ell A} +o_p(1), \q i=1,2.
\end{equation}
Hence, when $\ell\neq k$,
\begin{equation}\label{star2}
\sqrt{N}  \ \dfrac{u_k^{{(i)}}[\ell]}{\|\fu_{kA}^{{(i)}} \|}
= \sqrt{\frac N T_i} \cdot \dfrac{\sqrt{\lambda_k^{{(i)}}\lambda_\ell^{{(i)}}}}{\lambda_k^{{(i)}}-\lambda_\ell^{{(i)}}} \  [\fR_N^{{(i)}}(\lambda_k^{{(i)}})]_{k \ell} \  +o_p(1), \q i=1,2.
\end{equation}
Similarly, by \eqref{eqn:inner_u_kA},
\begin{equation}\label{star3}
	N \(1-\dfrac{|u_k^{{(i)}}[k]|}{\|\fu_{kA}^{{(i)}}\|} \)
	= \dfrac{N}{2T_i} \cdot \sum_{\ell\neq k,\ell=1}^r \dfrac{\lambda_k^{{(i)}}\lambda_\ell^{{(i)}}}{(\lambda_k^{{(i)}}-\lambda_\ell^{{(i)}})^2} [\fR_N^{{(i)}}(\lambda_k^{{(i)}})]_{k\ell}^2 +o_p(1), \q i=1,2.
\end{equation}	
Proposition \ref{thm:clt_ev_norm} implies that
\begin{equation}\label{star4}
N(1-\|\fu_{kA}^{{(i)}}\|) \ \pc  \ \dfrac{\rho_i}{2\theta_k^{{(i)}}} \int x dF^{{(i)}}(x), \q i=1,2.
\end{equation} 	
Write the two population eigen-matrices ${\fV}^{(1)}, {\fV}^{(2)}$ as
\[
{\fV}^{(i)} = ( {\fv}_1^{(i)}, \ldots, {\fv}_N^{(i)} ), ~~~~~~ i=1,2,
\]
and define
\[
\fXi=([\fXi]_{ij})={\fV}^{(1)}{}^\mT {\fV}^{(2)}
=\begin{pmatrix}
\fXi_{11} & \fXi_{12} \\
\fXi_{21} & \fXi_{22},
\end{pmatrix}	
\]
where $[\fXi]_{ij}={\fv}_i^{(1)}{}^\mT {\fv}_j^{(2)}$, and
$\fXi_{11}, \fXi_{12}, \fXi_{21}, \fXi_{22}$ are of sizes $r\times r, r\times (N-r), (N-r)\times r$ and $(N-r)\times (N-r)$, respectively.

Under null hypothesis $H_0^{(III,k)}: |\langle \fv_k^{(1)}, \fv_k^{(2)} \rangle|=1$, the $k$th row and $k$th column of $\fXi$ are zero except that the $k$th diagonal entry is one. To prove the theorem, note that
\begin{eqnarray*}
	\langle \wh{\fv}_k^{(1)}, \wh{\fv}_k^{(2)} \rangle
	&=& \langle {\fV}^{(1)} \fu_k^{(1)}, {\fV}^{(2)} \fu_k^{(2)} \rangle
	= \fu_k^{(1)}{}^\mT \fXi \fu_k^{(2)} \\
	&=& \fu_{kA}^{(1)}{}^\mT \fXi_{11} \fu_{kA}^{(2)} +
	\fu_{kA}^{(1)}{}^\mT \fXi_{12} \fu_{kB}^{(2)} +
	\fu_{kB}^{(1)}{}^\mT \fXi_{21} \fu_{kA}^{(2)} +
	\fu_{kB}^{(1)}{}^\mT \fXi_{22} \fu_{kB}^{(2)}.
\end{eqnarray*}
We start with the first term $\fu_{kA}^{(1)}{}^\mT \fXi_{11} \fu_{kA}^{(2)}$, and will  show later that
\begin{equation}\label{eqn:3terms}
N \fu_{kA}^{(1)}{}^\mT \fXi_{12} \fu_{kB}^{(2)} =o_p(1), ~~
N \fu_{kB}^{(1)}{}^\mT \fXi_{21} \fu_{kA}^{(2)} =o_p(1), \mbox{ and }
N \fu_{kB}^{(1)}{}^\mT \fXi_{22} \fu_{kB}^{(2)} =o_p(1).
\end{equation}

Because the entries in the $k$th row and $k$th column of $\fXi_{11}$ are zero  except that $[\fXi_{11}]_{kk}=1$, we have
\begin{eqnarray}\label{eqn:term_t1}
&& N \(1-\fu_{kA}^{(1)}{}^\mT \fXi_{11}  \fu_{kA}^{(2)}\) \nonumber\\
&=& N \(1-u_k^{(1)}[k] \cdot u_k^{(2)}[k] \)
-\sum_{i,j=1,i\neq k,j\neq k}^r N \ [\fXi_{11}]_{ij}\cdot u_k^{(1)}[i] \cdot u_k^{(2)}[j] .
\end{eqnarray}
For the first term, we have
\begin{equation}\label{eq:2sample_AA_1}
\aligned
	& N\(1-u_k^{(1)}[k] \cdot u_k^{(2)}[k] \)
	= N(1-u_k^{(1)}[k]) + u_k^{(1)}[k] \cdot N(1-u_k^{(2)}[k]) \\
	=:& N(1-u_k^{(1)}[k]) + N(1-u_k^{(2)}[k]) + \vep_1 \\
	=& N\(1-\dfrac{u_k^{(1)}[k]}{\|\fu_{kA}^{(1)}\|} \)
	+ N\dfrac{u_k^{(1)}[k]}{\|\fu_{kA}^{(1)}\|} (1-\|\fu_{kA}^{(1)}\| )
	+ N\(1-\dfrac{u_k^{(2)}[k]}{\|\fu_{kA}^{(2)}\|} \)\\
	& + N\dfrac{u_k^{(2)}[k]}{\|\fu_{kA}^{(2)}\|} (1-\|\fu_{kA}^{(2)}\| ) +\vep_1 \\
	=:& N\(1-\dfrac{u_k^{(1)}[k]}{\|\fu_{kA}^{(1)}\|} \)
	+ N(1-\|\fu_{kA}^{(1)}\| )
	+ N\(1-\dfrac{u_k^{(2)}[k]}{\|\fu_{kA}^{(2)}\|} \)\\
	&+ N(1-\|\fu_{kA}^{(2)}\| ) +\vep_2 +\vep_1,
\endaligned
\end{equation}
where
\begin{eqnarray*}
\vep_1 &=& N(1-u_k^{(2)}[k])(u_k^{(1)}[k]-1) ,  ~~~~ \textrm{and}\\
\vep_2 &=& N(1-\|\fu_{kA}^{(1)}\|) \(\dfrac{u_k^{(1)}[k]}{\|\fu_{kA}^{(1)}\|}-1 \)
	+N(1-\|\fu_{kA}^{(2)}\|) \(\dfrac{u_k^{(2)}[k]}{\|\fu_{kA}^{(2)}\|}-1 \).
\end{eqnarray*}
By Theorem 3 and Proposition \ref{thm:clt_ev_spike}, both $\vep_1$ and $\vep_2$ are $o_p(1)$.

For the second term on the right-hand side of  \eqref{eqn:term_t1}, by Proposition \ref{thm:clt_ev_norm}, we have
\begin{equation}\label{eq:2sample_AA_2}
\aligned
	& N \sum_{i,j\neq k, i,j=1}^r [\fXi_{11}]_{ij}\cdot  u_k^{(1)}[i] \cdot u_k^{(2)}[j]  \\
	=& N \|\fu_{kA}^{(1)}\| \cdot \|\fu_{kA}^{(2)}\| \cdot
	\sum_{i,j\neq k, i,j=1}^r [\fXi_{11}]_{ij}\cdot  \dfrac{u_k^{(1)}[i]}{\|\fu_{kA}^{(1)}\| } \cdot \dfrac{u_k^{(2)}[j]}{\|\fu_{kA}^{(2)}\|}  \\
	=& N  \sum_{i\neq k, i=1}^r \sum_{j\neq k, j=1}^r
	[\fXi_{11}]_{ij}\cdot  \dfrac{u_k^{(1)}[i]}{\|\fu_{kA}^{(1)}\| } \cdot \dfrac{u_k^{(2)}[j]}{\|\fu_{kA}^{(2)}\|}  +o_p(1).
\endaligned
\end{equation}

Combining \eqref{eq:2sample_AA_1} and \eqref{eq:2sample_AA_2} and using \eqref{star2} and \eqref{star3}, we obtain
\begin{eqnarray*}
	&& N\(1-\fu_{kA}^{(1)}{}^\mT \fXi_{11} \fu_{kA}^{(2)} \) \\
	&=& N\(1-\|\fu_{kA}^{(1)}\| \) + N\(1-\|\fu_{kA}^{(2)}\| \)
	+ N\(1-\dfrac{u_k^{(1)}[k]}{\|\fu_{kA}^{(1)}\|} \)
	+ N\(1-\dfrac{u_k^{(2)}[k]}{\|\fu_{kA}^{(2)}\|} \)  \\
	&& - N \sum_{i\neq k, i=1}^r \sum_{j\neq k, j=1}^r
	[\fXi_{11}]_{ij}\cdot  \dfrac{u_k^{(1)}[i]}{\|\fu_{kA}^{(1)}\| } \cdot \dfrac{u_k^{(2)}[j]}{\|\fu_{kA}^{(2)}\|}  + o_p(1)  \\
	&=& N\(1-\|\fu_{kA}^{(1)}\| \) + N\(1-\|\fu_{kA}^{(2)}\| \)  \\
	&& +\dfrac{N}{2T_1} \sum_{i\neq k,i=1}^r \dfrac{\lambda_k^{(1)} \lambda_i^{(1)} }{(\lambda_k^{(1)}-\lambda_i^{(1)})^2} [\fR_N^{(1)}(\lambda_k^{(1)})]_{ki}^2
	+\dfrac{N}{2T_2} \sum_{j\neq k,j=1}^r \dfrac{\lambda_k^{(2)}\lambda_j^{(1)}}{(\lambda_k^{(2)}-\lambda_j^{(2)})^2} [\fR_N^{(2)}(\lambda_k^{(2)})]_{kj}^2  \\
	&& - \sqrt{\dfrac{N^2}{T_1T_2}} \sum_{i\neq k, i=1}^r \sum_{j\neq k, j=1}^r
	\dfrac{\sqrt{\lambda_k^{(1)}\lambda_i^{(1)}}}{\lambda_k^{(1)}-\lambda_i^{(1)}} \dfrac{\sqrt{\lambda_k^{(2)}\lambda_j^{(2)}}}{\lambda_k^{(2)}-\lambda_j^{(2)}}
	[\fR_N^{(1)}(\lambda_k^{(1)})]_{ki} \cdot [\fR_N^{(2)}(\lambda_k^{(2)})]_{kj}
	\cdot (\fXi_{11})_{ij} \\
&& +o_p(1).
\end{eqnarray*}
For $k=1,\ldots,r$, define two $(r-1)\times 1$ vectors $\fa_k$ and $\fb_k$ as
\begin{eqnarray*}
	\fa_k = \begin{pmatrix}
		\dfrac{\sqrt{\lambda_k^{(1)}\lambda_1^{(1)}}}{\lambda_k^{(1)}-\lambda_1^{(1)}} \  [\fR_N^{(1)}(\lambda_k^{(1)})]_{k1} \\
		\vdots \\
		\dfrac{\sqrt{\lambda_k^{(1)}\lambda_{k-1}^{(1)}}}{\lambda_k^{(1)}-\lambda_{k-1}^{(1)}} \  [\fR_N^{(1)}(\lambda_k^{(1)})]_{k(k-1)} \\
		\dfrac{\sqrt{\lambda_k^{(1)}\lambda_{k+1}^{(1)}}}{\lambda_k^{(1)}-\lambda_{k+1}^{(1)}} \  [\fR_N^{(1)}(\lambda_k^{(1)})]_{k(k+1)} \\
		\vdots \\
		\dfrac{\sqrt{\lambda_k^{(1)}\lambda_{r}^{(1)}}}{\lambda_k^{(1)}-\lambda_r^{(1)}} \  [\fR_N^{(1)}(\lambda_k^{(1)})]_{kr}
	\end{pmatrix}
	, ~~~~~~~
	\fb_k = \begin{pmatrix}
		\dfrac{\sqrt{\lambda_k^{(2)}\lambda_{1}^{(2)}}}{\lambda_k^{(2)}-\lambda_1^{(2)}} \  [\fR_N^{(2)}(\lambda_k^{(2)})]_{k1} \\
		\vdots \\
		\dfrac{\sqrt{\lambda_k^{(2)}\lambda_{k-1}^{(2)}}}{\lambda_k^{(2)}-\lambda_{k-1}^{(2)}} \  [\fR_N^{(2)}(\lambda_k^{(2)})]_{k(k-1)} \\
		\dfrac{\sqrt{\lambda_k^{(2)}\lambda_{k+1}^{(2)}}}{\lambda_k^{(2)}-\lambda_{k+1}^{(2)}} \  [\fR_N^{(2)}(\lambda_k^{(2)})]_{k(k+1)} \\
		\vdots \\
		\dfrac{\sqrt{\lambda_k^{(2)}\lambda_{r}^{(2)}}}{\lambda_k^{(2)}-\lambda_r^{(2)}} \  [\fR_N^{(2)}(\lambda_k^{(2)})]_{kr}
	\end{pmatrix}.
\end{eqnarray*}
Under the assumptions of Theorem 6, by Lemma \ref{lem:limit_R}, we have
\begin{eqnarray}\label{eqn:akbk}
\fa_k \ \wk \ N(0,\fD_{a_k}), ~~~~~
\fb_k \ \wk \ N(0,\fD_{b_k}),
\end{eqnarray}
where
\begin{eqnarray*}
	\fD_{a_k} &=& \diag(\om_{k1}^{(1)},\ldots,\om_{k(k-1)}^{(1)}, \om_{k(k+1)}^{(1)}, \ldots,\om_{kr}^{(1)} ), \\
	\fD_{b_k} &=& \diag(\om_{k1}^{(2)},\ldots,\om_{k(k-1)}^{(2)}, \om_{k(k+1)}^{(2)}, \ldots,\om_{kr}^{(2)} ),
\end{eqnarray*}
and
\[
\om_{kj}^{(i)} =
\dfrac{\theta_k^{(i)}\theta_j^{(i)} }{(\theta_k^{(i)}-\theta_j^{(i)} )^2},
~~~~~~ \textrm{for}
~~~ i=1,2, ~~~ 1\le  j \neq k \le r.
\]
Let $\fXi_{11,-k}$ be the matrix obtained by removing the $k$th row and $k$th column of $\fXi_{11}$. Then by \eqref{eqn:term_t1},  \eqref{star4} and \eqref{eqn:akbk}, we get
\begin{eqnarray*}
	&& N\(1-\fu_{kA}^{(1)}{}^\mT \fXi_{11} \fu_{kA}^{(2)} \) \\
	&=& N\(1-\|\fu_{kA}^{(1)}\| \) + N\(1-\|\fu_{kA}^{(2)}\| \)\\
	&&+ \dfrac{N}{2T_1} \fa_k^\mT \fa_k + \dfrac{N}{2T_2} \fb_k^\mT \fb_k
	-\sqrt{\dfrac{N^2}{T_1T_2}} \fa_k^\mT \fXi_{11,-k} \fb_k +o_p(1) \\
	&=& N\(1-\|\fu_{kA}^{(1)}\| \) + N\(1-\|\fu_{kA}^{(2)}\| \)\\
	&&+ \dfrac 12
	\begin{pmatrix}
		\sqrt{\dfrac{N}{T_1}}\fa_k \\
		\sqrt{\dfrac{N}{T_2}}\fb_k
	\end{pmatrix}^\mT
	\begin{pmatrix}
		\fI_{r-1} & -\fXi_{11,-k} \\
		-\fXi_{11,-k}^\mT & \fI_{r-1}
	\end{pmatrix}
	\begin{pmatrix}
		\sqrt{\dfrac{N}{T_1}}\fa_k \\
		\sqrt{\dfrac{N}{T_2}}\fb_k
	\end{pmatrix} +o_p(1) \\
	&\wk& \dfrac{\rho_1}{2\theta_k^{(1)}} \int x dF^{(1)} + \dfrac{\rho_2}{2\theta_k^{(2)}} \int x dF^{(2)} + \dfrac 12 \fq_k^\mT
	\begin{pmatrix}
		\fI_{r-1} & -\fXi_{11,-k}^* \\
		-{\fXi_{11,-k}^*}^\mT & \fI_{r-1}
	\end{pmatrix}
	\fq_k,
\end{eqnarray*}
where
$\fq_k \sim N(0,\fD_k)$ with
$\fD_k=
\begin{pmatrix}
\rho_1\fD_{a_k} & 0 \\
0 & \rho_2\fD_{b_k}
\end{pmatrix}  .
$

Combining \eqref{eqn:delta}, \eqref{eqn:3terms} and the convergence
\[
\dfrac{N^2}{T_i(N-r) \wh{\lam}_k^{(i)}} \sum_{j=r+1}^N \wh{\lam}_j^{(i)} \ \pc \
\dfrac{\rho_i}{\th_k^{(i)}} \int x dF^{(i)},~~~~ i=1,2,
\]
we get that
our  test statistic
\[
T_{vk}= 2N\(1-\langle\wh{\fv}_k^{(1)},\wh{\fv}_k^{(2)}\rangle \)
-\dfrac{N^2}{T_1(N-r)\wh{\lam}_k^{(1)}} \cdot \sum_{j=r+1}^N \wh{\lam}_j^{(1)}
-\dfrac{N^2}{T_2(N-r)\wh{\lam}_k^{(2)}} \cdot \sum_{j=r+1}^N \wh{\lam}_j^{(2)}
\]
converges weakly to
\[
\fq_k^T
\begin{pmatrix}
\fI_{r-1} & -\fXi_{11,-k}^* \\
-{\fXi_{11,-k}^*}^\mT & \fI_{r-1}
\end{pmatrix}
\fq_k  .
\]

It remains to prove \eqref{eqn:3terms}.
By  \eqref{eqn:u_kB_exp}, we have
\begin{eqnarray*}
N\fu_{kA}^{(1)}{}^\mT \fXi_{12} \fu_{kB}^{(2)}
= N\fu_{kA}^{(1)}{}^\mT \fXi_{12} (\lambda_k^{(2)}\fI-\fS_{22}^{(2)})^{-1} \fS_{21}^{(2)} \fu_{kA}^{(2)} +\vep_3 ,
\end{eqnarray*}
where
\begin{equation}\label{eq:eps_3}
\vep_3=N\fu_{kA}^{(1)}{}^\mT \fXi_{12} [(\wh{\lam}_k^{(2)}\fI-\fS_{22}^{(2)})^{-1} - (\lambda_k^{(2)}\fI-\fS_{22}^{(2)})^{-1}] \fS_{21}^{(2)} \fu_{kA}^{(2)}.
\end{equation}
Write
\[
N\fu_{kA}^{(1)}{}^\mT \fXi_{12} (\lambda_k^{(2)}\fI-\fS_{22}^{(2)})^{-1} \fS_{21}^{(2)} \fu_{kA}^{(2)}
= I_1 + I_2 + \vep_4,
\]
where
\begin{eqnarray*}
I_1 &=& N \ \wt{\fe}_{kA}^\mT \fXi_{12} (\lambda_k^{(2)}\fI-\fS_{22}^{(2)})^{-1} \fS_{21}^{(2)} \wt{\fu}_{kA}^{(2)}  , \\
I_2 &=& N(\wt{\fu}_{kA}^{(1)}-\wt{\fe}_{kA} )^\mT \fXi_{12} (\lambda_k^{(2)}\fI-\fS_{22}^{(2)})^{-1} \fS_{21}^{(2)} \wt{\fu}_{kA}^{(2)} , \q \mbox{and}\\
\vep_4 &=& N(\fu_{kA}^{(1)}-\wt{\fu}_{kA}^{(1)} )^\mT \fXi_{12} (\lambda_k^{(2)}\fI - \fS_{22}^{(2)} )^{-1} \fS_{21}^{(2)} \fu_{kA}^{(2)} \\
		&& + N\wt{\fu}_{kA}^{(1)}{}^\mT \fXi_{12} (\lambda_k^{(2)}\fI - \fS_{22}^{(2)} )^{-1} \fS_{21}^{(2)} (\fu_{kA}^{(2)} - \wt{\fu}_{kA}^{(2)} )\\
&=:& \vep_{41} + \vep_{42}.
\end{eqnarray*}
For term $I_1$, note that $\wt{\fe}_{kA}{}^\mT \fXi_{12}$ is the $k$th row of $\fXi_{12}$ which is zero, hence $I_1=0$.

Next, we prove that $I_2=o_p(1)$. Write
\begin{eqnarray*}
I_2 &=& N(\wt{\fu}_{kA}^{(1)}-\wt{\fe}_{kA} )^\mT \fXi_{12} (\lambda_k^{(2)}\fI-\fS_{22}^{(2)})^{-1} \fS_{21}^{(2)}
		(\wt{\fu}_{kA}^{(2)}-\wt{\fe}_{kA})   \\
		&&+N(\wt{\fu}_{kA}^{(1)}-\wt{\fe}_{kA} )^\mT \fXi_{12} (\lambda_k^{(2)}\fI-\fS_{22}^{(2)})^{-1} \fS_{21}^{(2)}
		\wt{\fe}_{kA} \\
		&=:& I_{21} + I_{22}.
\end{eqnarray*}
By Proposition \ref{thm:clt_spike_spike} and that $\lam_k^{(2)}=O(N)$, we get $I_{21}=o_p(1)$.
As to $I_{22}$, by \eqref{star1}, Assumption A and the facts that $\|\fS_{22}^{(2)}\|=O_p(1)$ and  $\|\fS_{21}^{(2)}\|=O_p(\sqrt{N})$, we obtain
\begin{eqnarray*}
I_{22}&=& \dfrac{N}{\sqrt{T_1}} \sum_{\ell\neq k,\ell=1}^r \dfrac{\sqrt{\lambda_k^{(1)}\lambda_\ell^{(1)}}}{\lambda_k^{(1)}-\lambda_\ell^{(1)}} [\fR_N^{(1)}(\lambda_k^{(1)})]_{k\ell}  \\
&& ~~~~~~~~~~~~~~~~~~~~~~~~
		\times
\wt{\fe}_{\ell A}^\mT \fXi_{12} (\lambda_k^{(2)}\fI-\fS_{22}^{(2)})^{-1}
 \dfrac{1}{T_2} \fX_B^{(2)} {\fZ_A^{(2)}}^\mT {\fLa_A^{(2)}}^{1/2} \wt{\fe}_{kA} +o_p(1) \\
		&=&\dfrac{N}{\sqrt{T_1 \lambda_k^{(2)}}}  \sum_{\ell\neq k,\ell=1}^r
		\dfrac{\sqrt{\lambda_k^{(1)}\lambda_\ell^{(1)}}}{\lambda_k^{(1)}-\lambda_\ell^{(1)}}
		[\fR_N^{(1)}(\lambda_k^{(1)})]_{k\ell} \\
&& ~~~~~~~~~~~~~~~~~~~~~~~~
\times \dfrac{1}{T_2}
		\wt{\fe}_{\ell A}^\mT \fXi_{12} (\fI-1/\lambda_k^{(2)} \cdot \fS_{22}^{(2)})^{-1}  \fX_B^{(2)} \fz_{(k)}^{(2)} +o_p(1).
\end{eqnarray*}
Using the independence between $\fz_{(k)}^{(2)}$ and $\fX_B^{(2)}$, $\fS_{22}^{(2)}$,
Assumption A  and that $\|\fS_{22}\|=O_p(1)$, we have
\[
\dfrac{1}{T_2}\wt{\fe}_{\ell A}^\mT \fXi_{12} (\fI-1/\lambda_k^{(2)} \cdot \fS_{22}^{(2)})^{-1}  \fX_B^{(2)} \fz_{(k)}^{(2)}=o_p(1).
\]
Therefore, $I_{22}=o_p(1)$.

We now analyze $\vep_4$.
For $\vep_{41}$, because $N(\|\fu_{kA}^{(1)}\|-1 )=O_p(1)$, by Proposition \ref{thm:clt_ev_norm} and that $||\fS_{21}^{(2)}|| = O_p(\sqrt{N})$,
we get $\vep_{41}=o_p(1)$. Similarly,  we get $\vep_{42}=o_p(1)$.

To sum up, we have shown that
\begin{equation}\label{eq:2sample_AA_3}
N\fu_{kA}^{(1)}{}^\mT \fXi_{12} (\lambda_k^{(2)}\fI-\fS_{22}^{(2)})^{-1} \fS_{21}^{(2)} \fu_{kA}^{(2)}
=o_p(1).
\end{equation}

Next, we prove that $\vep_3=o_p(1)$. We have
\begin{eqnarray*}
\vep_3 &=&  N(\lambda_k^{(2)}-\wh{\lam}_k^{(2)}) \fu_{kA}^{(1)}{}^\mT \fXi_{12} (\wh{\lam}_k^{(2)}\fI-\fS_{22}^{(2)})^{-1} (\lambda_k^{(2)}\fI-\fS_{22}^{(2)})^{-1} \fS_{21}^{(2)} \fu_{kA}^{(2)} \\
&=& N(\lambda_k^{(2)}-\wh{\lam}_k^{(2)}) \fu_{kA}^{(1)}{}^\mT \fXi_{12} (\lambda_k^{(2)}\fI-\fS_{22}^{(2)})^{-2} \fS_{21}^{(2)} \fu_{kA}^{(2)}  \\
&& + N(\lambda_k^{(2)}-\wh{\lam}_k^{(2)}) \fu_{kA}^{(1)}{}^\mT \fXi_{12}
		[(\wh{\lam}_k^{(2)}\fI-\fS_{22}^{(2)})^{-1} - (\lambda_k^{(2)}\fI-\fS_{22}^{(2)})^{-1}]
		(\lambda_k^{(2)}\fI-\fS_{22}^{(2)})^{-1} \fS_{21}^{(2)} \fu_{kA}^{(2)}  \\
&=& \dfrac{N}{\lambda_k^{(2)}} \(1-\dfrac{\wh{\lam}_k^{(2)}}{\lambda_k^{(2)}} \)
		\fu_{kA}^{(1)}{}^\mT \fXi_{12} (\fI-1/\lambda_k^{(2)} \cdot \fS_{22}^{(2)})^{-2} \fS_{21}^{(2)} \fu_{kA}^{(2)}  \\
&& + \dfrac{N}{\wh{\lam}_k^{(2)}} \(1-\dfrac{\wh{\lam}_k^{(2)}}{\lambda_k^{(2)}} \)^2
		\fu_{kA}^{(1)}{}^\mT \fXi_{12}
		(\fI-1/\wh{\lam}_k^{(2)} \cdot \fS_{22}^{(2)})^{-1}
		(\fI-1/\lambda_k^{(2)} \cdot \fS_{22}^{(2)})^{-2} \fS_{21}^{(2)} \fu_{kA}^{(2)} \\
&=:& \vep_{31} + \vep_{32}.
\end{eqnarray*}
Following the same proof strategy as for \eqref{eq:2sample_AA_3} and applying Theorem 1, we get $\vep_{11}=o_p(1)$.
For $\vep_{32}$, using Assumption (A.i), Theorem 1 and that $||\fS_{21}^{(2)}|| = O_p(\sqrt{N})$,
we get $\vep_{42}=o_p(1)$.

To sum up, we have
\[
N\fu_{kA}^{(1)}{}^\mT \fXi_{12} \fu_{kB}^{(2)} = o_p(1).
\]
Using  the same argument we get
$N\fu_{kB}^{(1)}{}^\mT \fXi_{21} \fu_{kA}^{(2)} = o_p(1).$

Finally, we show that $N\fu_{kB}^{(1)}{}^\mT \fXi_{22} \fu_{kB}^{(2)} = o_p(1).$
By \eqref{eqn:u_kB_exp}, we have
\begin{eqnarray*}
&& N \fu_{kB}^{(1)}{}^\mT \fXi_{22} \fu_{kB}^{(2)} \\
&=& N \fu_{kA}^{(1)}{}^\mT \fS_{12}^{(1)} (\wh{\lam}_k^{(1)}\fI-\fS_{22}^{(1)})^{-1}
		\fXi_{22} (\wh{\lam}_k^{(2)}\fI -\fS_{22}^{(2)} )^{-1} \fS_{21}^{(2)} \fu_{kA}^{(2)} \\
&=& N\|\fu_{kA}^{(1)}\|\cdot \|\fu_{kA}^{(2)}\|  \cdot
		\wt{\fu}_{kA}^{(1)}{}^\mT \fS_{12}^{(1)} (\wh{\lam}_k^{(1)}\fI-\fS_{22}^{(1)})^{-1}
		\fXi_{22} (\wh{\lam}_k^{(2)}\fI -\fS_{22}^{(2)} )^{-1} \fS_{21}^{(2)} \wt{\fu}_{kA}^{(2)} \\
&=& N \wt{\fu}_{kA}^{(1)}{}^\mT \fS_{12}^{(1)} (\wh{\lam}_k^{(1)}\fI-\fS_{22}^{(1)})^{-1}
		\fXi_{22} (\wh{\lam}_k^{(2)}\fI -\fS_{22}^{(2)} )^{-1} \fS_{21}^{(2)} \wt{\fu}_{kA}^{(2)} +o_p(1) ,
\end{eqnarray*}
where the last step follows from Proposition \ref{thm:clt_ev_norm}.
Note that by equation \eqref{star1}, we have
\begin{eqnarray*}
&& N(\wt{\fu}_{kA}^{(1)}-\wt{\fe}_{kA})^\mT \fS_{12}^{(1)} (\wh{\lam}_k^{(1)}\fI-\fS_{22}^{(1)})^{-1}
\fXi_{22} (\wh{\lam}_k^{(2)}\fI -\fS_{22}^{(2)} )^{-1} \fS_{21}^{(2)} \wt{\fu}_{kA}^{(2)} \\
&=& O_p\(\sqrt{N} \cdot \sqrt{\lam_k^{(1)}} \cdot \dfrac{1}{\lam_k^{(1)}} \cdot  \dfrac{1}{\lam_k^{(2)}} \cdot \sqrt{\lam_k^{(2)}} \) =o_p(1).	
\end{eqnarray*}
Similarly,
\[
N \ \wt{\fe}_{kA}^T \fS_{12}^{(1)} (\wh{\lam}_k^{(1)}\fI -\fS_{22}^{(1)} )^{-1} \fXi_{22}
(\wh{\lam}_k^{(2)}\fI -\fS_{22}^{(2)} )^{-1} \fS_{21}^{(2)} (\wt{\fu}_{kA}^{(2)} -\wt{\fe}_{kA}) = o_p(1).
\]
Therefore,
\[
N\fu_{kB}^{(1)}{}^\mT \fXi_{22} \fu_{kB}^{(2)}
=N \ \wt{\fe}_{kA}^\mT \fS_{12}^{(1)} (\wh{\lam}_k^{(1)}\fI-\fS_{22}^{(1)} )^{-1} \fXi_{22} (\wh{\lam}_k^{(2)}\fI-\fS_{22}^{(2)})^{-1} \fS_{21}^{(2)} \wt{\fe}_{kA} + o_p(1).
\]
Note further that by Theorem 1,
\begin{eqnarray*}
&& N \ \wt{\fe}_{kA}^\mT \fS_{12}^{(1)}
\left[
(\wh{\lam}_k^{(1)}\fI -\fS_{22}^{(1)} )^{-1} - (\lambda_k^{(1)}\fI -\fS_{22}^{(1)} )^{-1}
\right]
\fXi_{22}  (\wh{\lam}_k^{(2)}\fI -\fS_{22}^{(2)} )^{-1}
\fS_{21}^{(2)} \wt{\fe}_{kA}  \\
&=& \dfrac{N}{\wh{\lam}_k^{(1)}\wh{\lam}_k^{(2)}} \(1-\dfrac{\wh{\lam}_k^{(1)}}{\lambda_k^{(1)}} \) \\
&& \times \
\wt{\fe}_{kA}^\mT \fS_{12}^{(1)}
(\fI-1/\wh{\lam}_k^{(1)}\cdot\fS_{22}^{(1)} )^{-1}
(\fI-1/\lambda_k^{(1)}\cdot \fS_{22}^{(1)})^{-1} \fXi_{22}
(\fI-1/\wh{\lam}_k^{(2)}\cdot\fS_{22}^{(2)} )^{-1}  \fS_{21}^{(2)} \wt{\fe}_{kA} \\
&=& o_p(1).
\end{eqnarray*}
Similarly,
\[
 N \ \wt{\fe}_{kA}^\mT \fS_{12}^{(1)}
(\wh{\lam}_k^{(1)}\fI -\fS_{22}^{(1)} )^{-1}
\fXi_{22}
\left[
(\wh{\lam}_k^{(2)}\fI -\fS_{22}^{(2)} )^{-1} - (\lambda_k^{(2)}\fI -\fS_{22}^{(2)} )^{-1}
\right]
\fS_{21}^{(2)} \wt{\fe}_{kA}
=o_p(1).
\]
It follows that
\begin{eqnarray*}
 N\fu_{kB}^{(1)}{}^\mT \fXi_{22} \fu_{kB}^{(2)}
= N\wt{\fe}_{kA}^\mT \fS_{12}^{(1)} (\lambda_k^{(1)}\fI-\fS_{22}^{(1)})^{-1} \fXi_{22} (\lambda_k^{(2)}\fI-\fS_{22}^{(2)})^{-1} \fS_{21}^{(2)} \wt{\fe}_{kA} +o_p(1).
\end{eqnarray*}
Note that
\begin{eqnarray*}
&&N\wt{\fe}_{kA}^\mT \fS_{12}^{(1)} (\lambda_k^{(1)}\fI-\fS_{22}^{(1)})^{-1} \fXi_{22} (\lambda_k^{(2)}\fI-\fS_{22}^{(2)})^{-1} \fS_{21}^{(2)} \wt{\fe}_{kA}\\
&=& \dfrac{N}{T_1T_2}
		\wt{\fe}_{kA}^\mT {\fLa_A^{(1)}}^{1/2} \fZ_A^{(1)} {\fX_{B}^{(1)}}^\mT  (\lambda_k^{(1)}\fI-\fS_{22}^{(1)})^{-1} \fXi_{22} (\lambda_k^{(2)}\fI-\fS_{22}^{(2)})^{-1}  \fX_B^{(2)} {\fZ_A^{(2)}}^\mT {\fLa_A^{(2)}}^{1/2} \wt{\fe}_{kA} \\
&=& \dfrac{N}{T_1T_2} \sqrt{\lambda_k^{(1)}\lambda_k^{(2)}} \cdot
		{\fz_{(k)}^{(1)}}^\mT {\fX_B^{(1)}}^\mT
		(\lambda_k^{(1)}\fI-\fS_{22}^{(1)})^{-1} \fXi_{22} (\lambda_k^{(2)}\fI-\fS_{22}^{(2)})^{-1} \fX_B^{(2)} \fz_{(k)}^{(2)}.  \\
\end{eqnarray*}
Using the independence among $\fz_k^{(1)}, \fz_{(k)}^{(2)}, \fX_B^{(1)}$ and $\fX_B^{(2)}$, Assumption A and that $\|\fS_{22}\| =O_p(1)$, one can show that the last term is $o_p(1)$.
It follows that
\[
N\fu_{kB}^{(1)}{}^\mT \fXi_{22} \fu_{kB}^{(2)} =o_p(1),
\]
which completes the proof of Theorem 6. 	
\end{proof}

\begin{proof}[Proof of Corollary 1]
If  $({\fv}_1^{(1)}, \ldots, {\fv}_r^{(1)})=({\fv}_1^{(2)}, \ldots, {\fv}_r^{(2)})$,
then $\fXi_{11,-k}=\fI_{r-1}$.
Denote
$
\fq_k=
\begin{pmatrix}
\fq_{kA} \\
\fq_{kB},
\end{pmatrix},
$
where $\fq_{kA} \sim N(0,\rho_1\fD_{a_k})$, $\fq_{kB} \sim N(0,\rho_2\fD_{b_k})$ and $\fq_{kA}$ and $\fq_{kB}$ are independent.
Therefore, the limiting distribution becomes
\begin{eqnarray*}
		\fq_k^\mT
		\begin{pmatrix}
			\fI_{r-1} & \fI_{r-1}  \\
			\fI_{r-1} & \fI_{r-1}	
		\end{pmatrix}
		\fq_k
		&=& \fq_{kA}^\mT \fq_{kA} + 2\fq_{kA}^\mT \fq_{kB} + \fq_{kB}^\mT \fq_{kB}  \\
		&=& \sum_{j\neq k, j=1}^r \(q_{kA}[j] + q_{kB}[j] \)^2 \\
		&\stackrel{d}{=}& \sum_{j\neq k, j=1}^r \(\rho_1\om_{kj}^{(1)} + \rho_2\om_{kj}^{(2)} \) \cdot Z_j^2,	
\end{eqnarray*}
where $Z_j$ are \hbox{i.i.d.} standard normal random variables.
	
\end{proof}

\end{document}